\documentclass[a4paper, 11pt, reqno]{amsart}
\usepackage{amsfonts, amsthm, amssymb, amsmath}
\usepackage{mathrsfs,array}
\usepackage{xy}
\usepackage{hyperref}
\usepackage{verbatim}
\usepackage{cancel}
\usepackage{tikz-cd}
\usepackage[latin1]{inputenc}
\usetikzlibrary{arrows}

\input xy
\xyoption{all}

\theoremstyle{plain}
\newtheorem{theo}{Theorem}[section]
\newtheorem{prop}[theo]{Proposition}
\newtheorem{lemma}[theo]{Lemma}
\newtheorem{lemm}[theo]{Lemma}
\newtheorem{coro}[theo]{Corollary}

\theoremstyle{definition}
\newtheorem{defi}[theo]{Definition}
\newtheorem{exa}[theo]{Example}
\newtheorem{conj}[theo]{Conjecture}
\newtheorem{rema}[theo]{Remark}

\DeclareMathOperator{\coker}{coker}

\DeclareMathOperator{\Ext}{Ext}

\DeclareMathOperator{\GL}{GL}

\DeclareMathOperator{\Sp}{Sp}

\DeclareMathOperator{\Lie}{Lie}

\DeclareMathOperator{\Res}{Res}

\DeclareMathOperator{\an}{an}

\DeclareMathOperator{\Tor}{Tor}

\title{...}
\date {\today}
\author{Ildar Gaisin and Joaquin Rodrigues Jacinto}
\email{ildar.gaisin@imj-prg.fr}
\email{joaquin.rodrigues@imj-prg.fr}

\def\Q{{\bf Q}}
\def\Z{{\bf Z}}

\def\N{{\bf N}}

\def\P{{\bf P}}
\def\Robba{{\mathscr{R}}}
\def\zp{{\Z_p}}
\def\zpe{{\mathbf{Z}_p^\times}}
\def\qpe{{\mathbf{Q}_p^\times}}

\def\qp{{\Q_p}}

\def\matrice#1#2#3#4{{\big(\begin{smallmatrix}#1&#2\\ #3&#4\end{smallmatrix}\big)}}

\begin{document}
\title{Arithmetic families of $(\varphi,\Gamma)$-modules and locally analytic representations of $\GL_2(\qp)$}
\begin{abstract}
Let $A$ be a $\qp$-affinoid algebra, in the sense of Tate. We develop a theory of locally convex $A$-modules parallel to the treatment in the case of a field by Schneider and Teitelbaum. We prove that there is an integration map linking a category of locally analytic representations in $A$-modules and separately continuous \emph{relative} distribution modules. There is a suitable theory of locally analytic cohomology for these objects and a version of Shapiro's Lemma. In the case of a field this has been substantially developed by Kohlhaase. As an application we propose a $p$-adic Langlands correspondence in \emph{families}: For a \emph{regular} trianguline $(\varphi,\Gamma)$-module of dimension 2 over the relative Robba ring $\Robba_A$ we construct a locally analytic $\GL_2(\qp)$-representation in $A$-modules. 
\end{abstract}

\maketitle

\makeatletter
\def\@tocline#1#2#3#4#5#6#7{\relax
  \ifnum #1>\c@tocdepth 
  \else
    \par \addpenalty\@secpenalty\addvspace{#2}%
    \begingroup \hyphenpenalty\@M
    \@ifempty{#4}{%
      \@tempdima\csname r@tocindent\number#1\endcsname\relax
    }{%
      \@tempdima#4\relax
    }%
    \parindent\z@ \leftskip#3\relax \advance\leftskip\@tempdima\relax
    \rightskip\@pnumwidth plus4em \parfillskip-\@pnumwidth
    #5\leavevmode\hskip-\@tempdima
      \ifcase #1
       \or\or \hskip 1em \or \hskip 2em \else \hskip 3em \fi%
      #6\nobreak\relax
    \dotfill\hbox to\@pnumwidth{\@tocpagenum{#7}}\par
    \nobreak
    \endgroup
  \fi}
\makeatother

\setcounter{secnumdepth}{4}
\setcounter{tocdepth}{4}
\tableofcontents

\section{Introduction}

\subsection{An extension of the $p$-adic Langlands correspondence}

The aim of this article is to study the $p$-adic Langlands correspondence for $\mathrm{GL}_2(\qp)$ in arithmetic families. To put things into context, let us recall the general lines of this correspondence. In \cite{colmez2010}, \cite{Pasimcmf} and \cite{cpdpadiclc}, a bijection $V \mapsto \Pi(V)$ between absolutely irreducible $2$-dimensional continuous $L$-representations\footnote{During all this text, $L$ will denote the coefficient field, which is a finite extension of $\qp$.} of the absolute Galois group $\mathcal{G}_\qp$ of $\qp$ and admissible unitary non-ordinary Banach $L$-representations of $\mathrm{GL}_2(\qp)$ which are topologically absolutely irreducible is established. The inverse functor $\Pi \mapsto V(\Pi)$ is sometimes referred to as  the Montr\'eal functor, cf. \cite[\S IV]{colmez2010}. 

The basic strategy of the construction of the functor $V \mapsto \Pi(V)$ is the following: by Fontaine's equivalence, the category of local Galois representations in $L$-vector spaces is equivalent to that of \'etale $(\varphi, \Gamma)$-modules over Fontaine's field $\mathscr{E}_L$ \footnote{The field $\mathscr{E}_L$ is defined as the Laurent series $\sum_{n \in \Z} a_n T^n$ such that $a_n \in L$ are bounded and $\lim_{n \to - \infty} a_n = 0$. $\mathscr{E}$ is equipped with continuous actions of $\Gamma = \zpe$ (we note $\sigma_a$, $a \in \zpe$, its elements) and an operator $\varphi$ defined by the formulas $\sigma_a(T) = (1 + T)^a - 1$ and $\varphi(T) = (1 + T)^p - 1$. Recall that a $(\varphi, \Gamma)$-module is a free $\mathscr{E}$-module equipped with semi-linear continuous actions of $\Gamma$ and $\varphi$. }. The latter (linearized-) category is considered to be an upgrade as one can often perform explicit computations. Any such $(\varphi, \Gamma)$-module $D$ can be naturally seen as a $P^+$-equivariant sheaf \footnote{The matrix ${\matrice p 0 0 1}$ codifies the action of $\varphi$, ${\matrice a 0 0 1}$ the action of $\sigma_a \in \Gamma$ and ${\matrice 1 b 0 1}$ the multiplication by $(1 + T)^b$.} over $\zp$, where $P^+ = {\matrice {\zp - \{0\}} \zp 0 1}$ is a sub-semi-group of the mirabolic subgroup ($= {\matrice {\qpe} \qp 0 1}$) of $\mathrm{GL}_2(\qp)$. If $U$ is a compact open subset of $\zp$, we denote by $D \boxtimes U$ the sections over $U$ of this sheaf. In \cite{colmez2010}, a magical involution $w_D$ acting on $D \boxtimes \zpe$ is defined, allowing one (noting tht $\P^1(\qp)$ is built by glueing two copies of $\zp$ along $\zpe$) to extend $D$ to a $\mathrm{GL}_2(\qp)$-equivariant sheaf over\footnote{From now on, $\P^1$ will always mean $\P^1(\qp)$.} $\P^1$, which is denoted $D \boxtimes_{\omega} \P^1$, where $\omega = (\det D) \chi^{-1}$ \footnote{The character $\det D$ is the character of $\qpe$ defined by the actions of $\varphi$ and $\Gamma$ on $\wedge^2 D$. If $D$ is \'etale, it can also be seen as a Galois character via local class field theory. The character $\chi \colon x \mapsto x|x|$ denotes the cyclotomic character. We see both characters as characters of $\mathrm{GL}_2(\qp)$ by composing with the determinant.}. One then cuts out the desired Banach representation $\Pi(V)$ (and its dual) from the constituents of $D \boxtimes_{\omega} \P^1$.  More precisely, we have a short exact sequence of topological $\mathrm{GL}_2(\qp)$-modules $$ 0 \to \Pi(V)^* \otimes \omega \to D \boxtimes_{\omega} \P^1 \to \Pi(V) \to 0. $$

Let $\Robba_L$ denote the Robba ring \footnote{It is defined as the ring of Laurent series $\sum_n a_n T^n$, $a_n \in L$, converging on some annulus $0 < v_p(T) \leq r$ for some $r > 0$.} with coefficients in $L$. By a combination of results of Cherbonnier-Colmez (\cite{chercolsurpa}) and Kedlaya (\cite{kedpadicmothm}), the categories of \'etale $(\varphi, \Gamma)$-modules over $\mathscr{E}_L$ and $\Robba_L$ are equivalent. Call $D \mapsto D_{\rm rig}$ this correspondence. We have analogous constructions as above for $D_{\rm rig}$ and, in particular, we have a $\mathrm{GL}_2(\qp)$-equivariant sheaf $U \mapsto D_{\rm rig} \boxtimes U$ over $\P^1$. If we note $\Pi(V)^{\rm an}$ the locally analytic vectors of $\Pi(V)$, we get an exact sequence $$ 0 \mapsto (\Pi(V)^{\rm an})^* \otimes \omega \to D_{\rm rig} \boxtimes_{\omega} \P^1 \to \Pi(V)^{\rm an} \to 0. $$  The construction however of $D_{\rm rig} \boxtimes_\omega \P^1$ is not a straightforward consequence of $D \boxtimes_\omega \P^1$. This is mainly because the formula defining the involution does not converge for a $(\varphi, \Gamma)$-module over $\Robba_L$ \footnote{To construct the involution on $D_{\rm rig}$ in the \'etale case, one shows that $w_D$ stabilises $D^\dagger \boxtimes \zpe$, where $D^\dagger$ is the $(\varphi, \Gamma)$-module over the overconvergent elements $\mathscr{E}^\dagger_L$ of $\mathscr{E}_L$ corresponding to $D$ by the Cherbonnier-Colmez correspondence, and that it defines by continuity an involution on $D_{\rm rig} \boxtimes \zpe$.}.

Inspired by the calculations of the $p$-adic local correspondence for trianguline \footnote{A rank $2$ $(\varphi, \Gamma)$-module is trianguline if it is an extension of rank $1$ $(\varphi, \Gamma)$-modules. } \'etale $(\varphi, \Gamma)$-modules, Colmez (\cite{colmez2015}) has recently given a direct construction, for a (not necessarily \'etale) $(\varphi, \Gamma)$-module $\Delta$ (of rank 2) over $\Robba_L$, of a locally analytic $L$-representation $\Pi(\Delta)$ of $\GL_2(\qp)$. More precisely, we have the following theorem: 

\begin{theo} [\cite{colmez2015}, Th\'eor\`eme 0.1] There exists a unique extension of $\Delta$ to a $\mathrm{GL}_2(\qp)$-equivariant sheaf of $\qp$-analytic type \footnote{A sheaf $U \mapsto M \boxtimes U$ is of $\qp$-analytic type if, for every open compact $U \subseteq \P^1$ and every compact $K \subseteq \mathrm{GL}_2(\qp)$ stabilizing $U$, the space $M \boxtimes U$ is of LF-type and a continuous $\mathcal{D}(K)$-module, where $\mathcal{D}(K)$ is the distribution algebra over $K$.} $\Delta \boxtimes_\omega \P^1$ over $\P^1$ with central character $\omega$. Moreover, there exists a unique admissible locally analytic $L$-representation $\Pi(\Delta)$, with central character $\omega$, of $\GL_2(\qp)$, such that $$ 0 \to \Pi(\Delta)^* \otimes \omega \to \Delta \boxtimes_{\omega} \P^1 \to \Pi(\Delta) \to 0. $$
\end{theo}

The purpose of the present work is to study this correspondence in the context of arithmetic families of $(\varphi, \Gamma)$-modules. Results in this direction on the $\ell$-adic side (i.e. the \emph{classical} local Langlands correspondence, cf. \cite{HarTayllc}) have been achieved by Emerton-Helm in \cite{EmHellafam}. The arguments in \cite{colmez2015} strongly rely on the cohomology theory of locally analytic representations developed in \cite{kohl2011}, and specifically on Shapiro's lemma. Since the authors are not aware of any reference for these results in the relative setting, we develop, in an appendix (cf. \S \ref{ap:lagrep}), the necessary definitions and properties of locally analytic $\GL_2(\qp)$-representations in $A$-modules. Since this point might carry some interest on its own, we describe it in more detail in \S \ref{introla} below. We will exclusively work with affinoid spaces in the sense of Tate, rather than Berkovich or Huber. Let $A$ be a $\qp$-affinoid algebra and let $\Robba_A$ be the relative Robba ring over $A$. Our main result can be stated as follows:

\begin{theo} \label{result}
Let $A$ be a $\qp$-affinoid algebra and let $\Delta$ be a trianguline $(\varphi,\Gamma)$-module over $\Robba_A$ of rank $2$ which is an extension of $\mathscr{R}_A(\delta_2)$ by $\mathscr{R}_A(\delta_1)$, where $\delta_1,\delta_2: \qpe \rightarrow A^\times$ are locally analytic characters satisfying some regularity assumptions \footnote{Precisely, we suppose that $\delta_1 \delta_2^{-1}$ is pointwise never of the form $\chi x^i$ or $x^{-i}$ for some $i \geq 0$.}. Then there exists a unique extension of $\Delta$ to a $\mathrm{GL}_2(\qp)$-equivariant sheaf of $\qp$-analytic type $\Delta \boxtimes_{\omega} \P^1$ over $\P^1$ with central character $\omega = \delta_1 \delta_2 \chi^{-1}$ and a (not necessarily unique) locally analytic $\GL_2(\qp)$-representation \footnote{See Definition \ref{def:lareph} for the definition of a locally analytic $G$-representation in $A$-modules. } $\Pi(\Delta)$ in $A$-modules with central character $\omega$, living in an exact sequence
$$ 0 \rightarrow \Pi(\Delta)^{*} \otimes \omega \rightarrow \Delta \boxtimes_\omega \P^1 \rightarrow \Pi(\Delta) \rightarrow 0.$$
\end{theo}

This result is expected to have applications to the study of \emph{eigenvarieties}, however in this paper we make no attempt to say anything in this direction. 

\subsection{The construction of the correspondence}

The construction of the correspondence follows the general lines of \cite{colmez2015}, but several technical difficulties appear along the way. Let's briefly describe how to construct the correspondence $\Delta \mapsto \Pi(\Delta)$ and the additional problems that arise in the relative (affinoid) setting.

From the calculation of the locally analytic vectors of the unitary principal series (\cite[Th\'eor\`eme 0.7]{colmezunit}), one knows that, if $D$ is an \'etale trianguline $(\varphi, \Gamma)$-module over $\mathscr{E}_L$ of dimension $2$, then $(\Pi(D))^{\rm an}$ is an extension of locally analytic principal series. The idea of \cite{colmez2015} is to intelligently reverse this \textit{d\'evissage} of $D_{\rm rig} \boxtimes_\omega \P^1$ in order to actually construct it from these pieces.

For the rest of this introduction let $G = \mathrm{GL}_2(\qp)$ and $\overline{B}$ be its lower Borel subgroup and let $\delta_1$, $\delta_2$ and $\omega$ be as in Theorem \ref{result}. Using a relative version of the dictionary of $p$-adic functional analysis, we construct, for $? \in \{ +, -, \emptyset \}$, $G$-equivariant sheaves $\Robba_A^?(\delta_1) \boxtimes_\omega \P^1$ (with central character $\omega$) of $\qp$-analytic type living in an exact sequence 
\[ 0 \to \Robba_A^+(\delta_1) \boxtimes_\omega \P^1 \to \Robba_A(\delta_1) \boxtimes_\omega \P^1 \to \Robba_A^-(\delta_1) \boxtimes_\omega \P^1 \to 0. \]
Moreover, one can get identifications $B_A(\delta_2, \delta_1)^* \otimes \omega \cong \Robba_A^+(\delta_1) \boxtimes_\omega \P^1$ and $B_A(\delta_1, \delta_2) \cong \Robba_A^-(\delta_1) \boxtimes_\omega \P^1$, where $B_A(\delta_1, \delta_2) = \mathrm{Ind}_{\overline{B}}^G(\delta_1 \chi^{-1} \otimes \delta_2)$  denotes the locally analytic principal series. These identifications allow us to consider the locally analytic principal series (and their duals) as (the global sections of) $G$-equivariant sheaves over $\P^1$ of interest.

We then construct the $G$-equivariant sheaf $\Delta \boxtimes_\omega \P^1$ over $\P^1$ as an extension of $\Robba_A(\delta_2) \boxtimes_\omega \P^1$ by $\Robba_A(\delta_1) \boxtimes_\omega \P^1$. This is done, as in \cite{colmez2015}, by showing that extensions of $\Robba_A(\delta_2)$ by $\Robba_A(\delta_1)$ correspond to extensions of $\Robba_A^+(\delta_2) \boxtimes_\omega \P^1$ by $\Robba_A(\delta_1) \boxtimes_\omega \P^1$. One then shows that an extension of $\Robba_A^+(\delta_2) \boxtimes_\omega \P^1$ by $\Robba_A(\delta_1) \boxtimes_\omega \P^1$ uniquely extends to an extension of $\Robba_A(\delta_2) \boxtimes_\omega \P^1$ by $\Robba_A(\delta_1) \boxtimes_\omega \P^1$. Once the sheaf $\Delta \boxtimes_\omega \P^1$ is constructed, one shows that the intermediate extension of $\Robba_A^+(\delta_2) \boxtimes_\omega \P^1$ by $\Robba_A^-(\delta_1) \boxtimes_\omega \P^1$ splits and thus one can separate the spaces that are Fr\'echets from those that are an inductive limit of Banach spaces so as to cut out the desired representation $\Pi(\Delta)$.

The fact that, for $? \in \{ +, -, \emptyset \}$, the $P^+$-module $\Robba_A^?(\delta_1)$ can be seen as sections over $\zp$ of a $G$-equivariant sheaf over $\P^1$, and that the semi-group $\overline{P}^+ = {\matrice{\zp - \{0\}} 0 {p \zp} 1}$ stabilizes $\zp$, show that $\Robba_A^?(\delta_1) = \Robba_A^?(\delta_1) \boxtimes_\omega \zp$ is automatically equipped with an extra action of the matrix ${\matrice 1 0 p 1}$. We call 
\[
\Robba_A^?(\delta_1,\delta_2):= (\Robba_A^?(\delta_1) \boxtimes_{\omega} \zp) \otimes \delta_2^{-1}
\] the $\overline{P}^+$-module thus defined. The technical heart for proving Theorem \ref{result} is a comparison result between the cohomology of the semi-groups $A^+ = {\matrice {\zp - \{0\}} 0 0 1}$ and $\overline{P}^+$ with values in $\Robba_A(\delta_1\delta_2^{-1})$ and $\Robba_A(\delta_1, \delta_2)$, respectively.

\begin{theo}\label{thm:resmincomp}
The restriction morphism from $\overline{P}^+$ to $A^+$ induces an isomorphism
\[
H^1(\overline{P}^+, \Robba_A(\delta_1, \delta_2)) \to H^1(A^+, \Robba_A(\delta_1\delta_2^{-1})).
\]
\end{theo} 

The semi-group $A^+$ should be thought of as encoding the action of $\varphi$ and $\Gamma$. The difficulty of course is to codify the action of the involution and this is the underlying idea for considering the semi-group $\overline{P}^+$. Indeed $\overline{P}^+$ should be thought of as getting closer to tracking the involution. Theorem \ref{thm:resmincomp} is (essentially) saying that a trianguline $(\varphi, \Gamma)$-module as in Theorem \ref{result} admits an extension to a $G$-equivariant sheaf over $\P^1$. 

Let us briefly describe the proof of Theorem \ref{thm:resmincomp}. The main idea is to reduce this bijection to the case of a point (i.e to the case where $A = L$ is a finite extension of $\qp$). The first step is to build a \emph{Koszul} complex which calculates $\overline{P}^+$-cohomology.

\begin{prop}
Let $M$ be an $A[\overline{P}^+]$-module such that the action of $\overline{P}^+$ extends to an action of the Iwasawa algebra $\zp[[\overline{P}^+]]$. Then the complex
\[
\mathscr{C}_{\tau, \varphi, \gamma}(M): 0 \to M \xrightarrow{X} M \oplus M \oplus M \xrightarrow{Y} M \oplus M \oplus M \xrightarrow{Z} M \to 0 
\]
where\footnote{Here $\tau = {\matrice 1 0  p 1}$ and $\delta_x = \frac{\tau^x -1}{\tau -1}$ for all $x \in \zpe$.}
\begin{align*}
X(x) &= ((1 - \tau)x, (1-\varphi)x, (\gamma-1)x)\\
Y(x,y,z) &= ((1- \varphi\delta_p)x + (\tau -1)y, (\gamma\delta_a -1)x + (\tau-1)z, (\gamma-1)y + (\varphi-1)z ) \\
Z(x,y,z) &= (\gamma\delta_a -1)x + (\varphi\delta_p -1)y + (1 - \tau)z
\end{align*}
calculates $\overline{P}^+$-cohomology. That is $H^i(\mathscr{C}_{\tau, \varphi, \gamma}(M)) = H^i(\overline{P}^+, M)$. 
\end{prop}

The asymmetric nature of $\mathscr{C}_{\tau, \varphi, \gamma}(M)$ is due to the non-commutativity of $\overline{P}^+$. A crude estimation of the maps $X$, $Y$ and $Z$ leads to the following corollary. 

\begin{coro}\label{cor:fincohinto}
The complex $\mathscr{C}_{\tau, \varphi, \gamma}(\Robba_A(\delta_1, \delta_2))$ is a pseudo-coherent complex concentrated in degrees $[0,3]$. In the terminology of the body of the paper, $\mathscr{C}_{\tau, \varphi, \gamma}(\Robba_A(\delta_1,\delta_2)) \in \mathcal{D}^{[0,3]}_{\mathrm{pc}}(A)$\footnote{We refer the reader to \S \ref{sec:formdercat} for the notion of a pseudo-coherent complex and the definition of $\mathcal{D}_{\mathrm{pc}}^{-}(A)$.}. In particular the cohomology groups $H^i(\overline{P}^+,\Robba_A(\delta_1, \delta_2))$ are finite $A$-modules. 
\end{coro}
More precisely the proof of Corollary \ref{cor:fincohinto} is reduced to proving finiteness of a \emph{twisted} $(\varphi, \Gamma)$-cohomology of $\Robba_A(\delta_1, \delta_2)$, cf. Lemma \ref{finitecohom}. 

The issue with $\mathscr{C}_{\tau, \varphi, \gamma}(M)$ is that the operators $\delta_x$ are difficult to comprehend,  rendering the complex almost impractical for explicit computations. One can however \emph{linearize} the situation and pass to the Lie algebra, where calculations are often feasible.

\begin{prop}
If $M \in \left\{ \mathscr{R}_L^{+}(\delta_1, \delta_2), \mathscr{R}_L^{-}(\delta_1, \delta_2), \mathscr{R}_L(\delta_1, \delta_2) \right\}$ then the  complex
\begin{equation*} 
\mathscr{C}_{u^{-},\varphi,a^{+}}(M): 0 \rightarrow M \xrightarrow{X'} M \oplus M \oplus M \xrightarrow{Y'} M \oplus M \oplus M \xrightarrow{Z'} M \rightarrow 0,
\end{equation*} 
where\footnote{Here $a^+ = {\matrice 1 0  0 0}$ and $u^- = {\matrice 0 0  1 0}$ are the usual elements of the Lie algebra $\mathfrak{gl}_2$ of $\GL_2$.}
\begin{align*}
X'(x) &= ((\varphi-1)x,a^{+}x,u^{-}x) \\
Y'(x,y,z) &= (a^{+}x - (\varphi-1)y, u^{-}y - (a^{+} +1)z, (p\varphi-1)z - u^{-}x) \\
Z'(x,y,z) &= u^{-}x + (p\varphi -1)y + (a^{+} + 1)z 
\end{align*}
calculates the Lie-algebra cohomology of $\overline{P}^+$. In particular, $H^0(\tilde{P}, H^i(\mathscr{C}_{u^{-}, \varphi, a^{+}}(M))) = H^i(\overline{P}^+,M)$\footnote{Here $\tilde{P} = {\matrice{\zpe} 0 {p \zp} 1}$, is the \emph{non-discrete} subgroup of $\overline{P}^+$.}. 
\end{prop} 

A long, tedious but direct calculation then leads to the following corollary. 

\begin{coro}\label{cor:cosdiminto}
The $L$-vector space $H^{2}(\overline{P}^{+}, \mathscr{R}_L(\delta_1, \delta_2))$ is of dimension 1.
\end{coro}

Corollaries \ref{cor:fincohinto} and \ref{cor:cosdiminto} allow for an analysis of a spectral sequence to take place and prove Theorem \ref{thm:resmincomp} in the case where $A$ is reduced. One then concludes via an induction argument on the index of nilpotence of the nilradical of $A$. Via the complex $\mathscr{C}_{u^-, \varphi, a^+}(M)$ we also obtain an alternative proof of \cite[Proposition 5.18]{colmez2015} in the case of a cyclotomic $(\varphi, \Gamma)$-module. Along the way we show a comparison isomorphism relating continuous cohomology and analytic cohomology for certain $(\varphi, \Gamma)$-modules (cf. Proposition \ref{prop:lazcompcom} for a precise statemtent).

Armed with Theorem \ref{thm:resmincomp}, the reader may notice at this point however, that there is an absence of theory required to conclude (or even make sense of) Theorem \ref{result}. The following questions are therefore unavoidable:
\begin{description}
\item [Q1] What is a locally convex $A$-module?
\item [Q2] What is a locally analytic $G$-representation in $A$-modules?
\item [Q3] What is the relation between locally analytic $G$-representations in $A$-modules and modules equipped with a (separately) continuous action of the relative distribution algebra $\mathscr{D}(G,A)$?
\end{description}

We provide a set of answers to these questions (\textbf{A1}-\textbf{A3}) and prove some fundamental properties regarding the locally analytic cohomology theory of $\mathscr{D}(G,A)$-modules, which we describe in the following section.

\subsection{Analytic families of locally analytic representations} \label{introla}

Recall that for a locally $\qp$-analytic group $H$, a theory of locally analytic representations of the group $H$ in $L$-vector spaces has been developed by Schneider and Teitelbaum (cf. \cite{tetschlad}, \cite{schteiugfin}, \cite{schtelaldisad}). In order to construct the $A$-module $\Pi(\Delta)$ of Theorem \ref{result}, with a locally analytic action of $G$, we need to develop a reasonable framework to make sense of such an object. It turns out that, with some care, much of the existent theory can be extended without serious difficulties to the relative context.

\begin{defi} [\textbf{A1}]
A locally convex $A$-module is a topological $A$-module whose underlying topology is a locally convex $\qp$-vector space. We let $\mathrm{LCS}_A$ be the category of locally convex $A$-modules. Its morphisms are all continuous $A$-linear maps.
\end{defi}

There is a notion of a strong dual in the category $\mathrm{LCS}_A$, however outside of our applications, it is ill-behaved (in the sense that there are few reflexive objects which are not free $A$-modules). Let $H$ be a locally $\qp$-analytic group. 

\begin{defi} [\textbf{A2}]
We define the category $\mathrm{Rep}_A^{\mathrm{la}}(H)$ whose objects are barrelled, Hausdorff, locally convex $A$-modules $M$ equipped with a topological $A$-linear action of $H$ such that, for every $m \in M$, the orbit map $h \mapsto h \cdot m$ is a locally analytic function on $H$ with values in $M$.
\end{defi}

Denote $\mathrm{LA}(H, A)$ the space of locally analytic functions on $H$ with values in $A$ and $\mathscr{D}(H, A) = \mathrm{Hom}_{A,\mathrm{cont}}(\mathrm{LA}(H, A), A)$ (equipped with the strong topology) its strong $A$-dual, the space of $A$-valued distributions on $H$. Both $LA(H,A)$ and $\mathscr{D}(H,A)$ are locally convex $A$-modules. In order to algebrize the situation, one proceeds as in \cite{tetschlad} and shows that a locally analytic representation of $H$ is naturally a module over the relative distribution algebra. More precisely let $\mathrm{Rep}_A^{\mathrm{la}, LB}(H) \subseteq \mathrm{Rep}_A^{\mathrm{la}}(H)$ denote the full subcategory consisting of spaces which are of $A$-LB-type (i.e inductive limit of Banach spaces whose transition morphism are $A$-linear) and complete. Then our main result in \S \ref{ap:lagrep} can be stated as follows:

\begin{theo} [\textbf{A3}]  \label{thm:maiapen} 
Every locally analytic representation of $H$ carries a separably continuous $A$-linear structure of $\mathscr{D}(H, A)$-module \footnote{More precisely, a separately continuous $A$-bilinear map $\mathscr{D}(H,A) \times M \to M$.}. Moreover, the category $\mathrm{Rep}_{A}^{\mathrm{la},LB}(H)$ is equivalent to the category of complete, Hausdorff locally convex $A$-modules which are of $A$-LB-type equipped with a separately continuous $\mathscr{D}(H,A)$-action with morphisms all continuous $\mathscr{D}(H,A)$-linear maps. 
\end{theo}

The idea to prove Theorem \ref{thm:maiapen} is of course to reduce to the well known result of Schneider-Teitelbaum, cf. \cite[Theorem 2.2]{tetschlad}. To achieve this, the main intermediary result required is the following proposition. 

\begin{prop}\label{prop:scbres}
There is an isomorphism of locally convex $A$-modules 
\[ \mathscr{D}(H, A) = \mathscr{D}(H, \qp) \widehat{\otimes}_{\qp, \iota} A. \] 
\end{prop}

In the case where $H$ is compact we show that $\mathrm{LA}(H,A)$ satisfies a boundedness property which we call $A$-regular (we refer the reader to Definition \ref{def:goodspat} and Lemma \ref{lem:LAgoodfu} for the precise statements). This is enough to prove Proposition \ref{prop:scbres}.

\begin{rema}
Proposition \ref{prop:scbres} would also follow immediately if $\mathrm{LA}(H,A)$ is complete (for $H$ compact). To the best of our knowledge this seems to be an open question if the dimension of $H \geq 2$. If $H \cong \zp$, one can identity $\mathrm{LA}(\zp,A)$ with the negative powers of $\Robba_A$ and conclude the result, cf. Lemma \ref{lem:magdimo}. In particular $\mathrm{LA}(\zp,A)$ is an example of an $A$-reflexive object, which is not free.   
\end{rema}

Finally with the equivalence of Theorem \ref{thm:maiapen} in mind, we switch our attention to cohomological questions concerning the category $\mathrm{Rep}_A^{\mathrm{la}}(H)$. 

\begin{defi}
Let $\mathscr{G}_{H,A}$ denote the category of complete Hausdorff locally convex $A$-modules with the structure of a separately continuous $A$-linear $\mathscr{D}(H,A)$-module, taking as morphisms all continuous $\mathscr{D}(H,A)$-linear maps. More precisely we demand that the module structure morphism
\[
\mathscr{D}(H,A) \times M \to M
\] 
is $A$-bilinear and separately continuous.
\end{defi}

Following Kohlhaase (\cite{kohl2011}, \cite{Tayhomcoh}), one can develop a locally analytic cohomology theory for the category $\mathscr{G}_{H,A}$. One can define groups $H^i_{\rm an}(H, M)$ and $\mathrm{Ext}^i_{\mathscr{G}_{H, A}}(M ,N)$ for $i \geq 0$ and objects $M$ and $N$ in $\mathscr{G}_{H, A}$. If $H_2$ is a closed locally $\qp$-analytic subgroup of $H_1$, we also have an induction functor\footnote{This is the dual of the \emph{standard} Induction functor, typically denoted $\mathrm{Ind}_{H_2}^{H_1}$, cf. Lemma \ref{lem:prinbag}.} $\mathrm{ind}_{H_2}^{H_1} \colon \mathscr{G}_{H_2, A} \to \mathscr{G}_{H_1, A}$. Our main purpose in considering such a theory is to show the following relative version of Shapiro's lemma, which is crucially used in the construction of the correspondence $\Delta \mapsto \Pi(\Delta)$ of Theorem \ref{result}:

\begin{prop} [Relative Shapiro's Lemma]
Let $H_1$ be a locally $\qp$-analytic group and let $H_2$ be a closed locally $\qp$-analytic subgroup. If $M$ and $N$ are objects of $\mathscr{G}_{H_2,A}$ and $\mathscr{G}_{H_1,A}$, respectively, then there are $A$-linear bijections
\[
\mathrm{Ext}^q_{\mathscr{G}_{H_1,A}}(\mathrm{ind}_{H_2}^{H_1}(M), N) \to \mathrm{Ext}^{q}_{\mathscr{G}_{H_2,A}}(M, N)
\]
for all $q \geq 0$. 
\end{prop}

\textbf{Structure of the paper.}
In \S \ref{sec:prelimz}, we extend the dictionary of $p$-adic funcational analysis to the relative setting. A key issue is to establish that the sheaf $\Robba_A(\delta_1) \boxtimes_{\omega} \P^1$ is $G$-equivariant over $\P^1$ and is $\qp$-analytic.

In \S \ref{sec:cohphigammch}, we use $(\varphi, \Gamma)$-cohomology to recalculate some results from \cite{chen2013} (in loc.cit. $(\psi, \Gamma)$-cohomology was used). A key result for the subsequent chapter is the nullity of $H^2(A^{+}, \Robba_A(\delta_1\delta_2^{-1})$ iff $\delta_1\delta_2^{-1}$ is (pointwise) never of the form $\chi x^{i}$ or $x^{-i}$ for some $i \geq 0$ (i.e. $\delta_1\delta_2^{-1}$ is regular).  

In \S \ref{sec:maincofinza} and \ref{sec:maincofinz}, the technical heart of the paper is carried out. We begin by proving the finiteness of $\overline{P}^+$-cohomology for $\Robba_A(\delta_1, \delta_2)$. Using the Lie-algebra complex we provide a different proof of \cite[Proposition 5.18]{colmez2015} (in the cyclotomic setting). We show that the dimension of the higher cohomology group $H^2(\overline{P}^+, \Robba_L(\delta_1, \delta_2))$ is constant (of dimension $1$) when $\delta_1\delta_2^{-1}$ is regular. 

In \S \ref{sec:extgrpact}, Theorem \ref{thm:resmincomp} is then established.   

In \S \ref{sec:cooftcor}, the general machinary developed in \cite[\S 6]{colmez2015} is used to construct $\Pi(\Delta)$ from a regular trianguline $(\varphi, \Gamma)$-module of rank 2 $\Delta$, over $\Robba_A$. 

In the appendix (\S \ref{ap:lagrep}) we establish a formal framework for the main result. We introduce the category of locally analytic $G$-representations in $A$-modules. We prove that there is a relationship between this category and a category of modules over the relative distribution algebra in the same spirit of \cite{tetschlad}. There is a locally analytic cohomology theory extending that of \cite{kohl2011} and we establish a relative version of Shapiro's Lemma. These results are used in \S \ref{sec:cooftcor}. \\

\textbf{Acknowledgments}. The debt this paper owes to the work of Pierre Colmez will be obvious to the reader. The authors are grateful to him for suggesting this problem and would like to thank him for many discussions on various aspects of this paper. The first author would like to thank Jean-Fran\c cois Dat for his continuous encouragement throughout. Most of this work has been done while the second author was completing his Ph.D. thesis under the supervision of Pierre Colmez, to whom he thanks heartily. He finally wants to express his gratitude to Sarah Zerbes. Next we want to thank Kiran Kedlaya for spending countless hours answering our questions on Robba rings and suggesting a crucial induction argument. We would also like to thank Jean-Fran\c cois Dat, Jan Kohlhaase and Peter Schneider for several helpful discussions on what the category of locally analytic $G$-representations in $A$-modules should be. Further thanks go to Gabriel Dospinescu and Arthur-C\'esar Le Bras for fruitful conversations on various topics.

\subsection{Notations}
Let $A$ be a $\qp$-affinoid algebra equipped with its Gauss-norm topology (making it a Banach space with norm $|\cdot |_{A}$ and $v_A = - \log_p |\cdot|_A$ a fixed valuation). We will denote
\[ G = \GL_2(\qp), \;\; A^+= \begin{pmatrix} \zp \backslash \{ 0 \} & 0\\ 0 & 1 \end{pmatrix}, \;\; P^+= \begin{pmatrix} \zp \backslash \{ 0 \} & \zp\\ 0 & 1 \end{pmatrix}, \;\; \overline{P}^+= \begin{pmatrix} \zp \backslash \{ 0 \} & 0\\ p\zp & 1 \end{pmatrix}. \] As usual we note $\Gamma = \zpe$, $\P^1 = \P^1(\qp)$ and we assume $p > 2$ throughout. For $n \geq 1$ we set $r_n := \frac{1}{(p-1)p^{n-1}}$ and denote the element $\begin{pmatrix} 1 & 0\\ p & 1 \end{pmatrix}$ by $\tau$. For two continuous characters $\delta_1, \delta_2 \colon \qpe \to A^\times$ we will denote $\delta = \delta_1\delta_2^{-1}\chi^{-1}$ and $\omega = \delta_1\delta_2\chi^{-1}$ where $\chi(x) = x|x|$ corresponds to the cyclotomic character via local class field theory. We denote by $\kappa(\delta_1):=\delta_1'(1)$, the weight of $\delta_1$.

\section{Preliminaries} \label{sec:prelimz}

We start by recalling, in the relative case, some well-known constructions that will play a key role in our construction. 

\subsection{Dictionary of relative functional analysis}

Let us first set up some notation and definitions. 

\subsubsection{Relative Laurent series rings}

The theory of relative Robba rings has been expounded by Kedlaya-Liu in \cite{KedLiurelpa}. For $0 < r < s \leq \infty$ (with $r$ and $s$ rational, except possibly $s = \infty$), the relative Robba ring $\Robba_A$ is defined by setting, $$ \Robba_A^{[r, s]} = \Robba_{\qp}^{[r, s]} \widehat{\otimes}_{\qp} A; \;\;\; \Robba_A^{]0, s]} = \varprojlim_{0 < r < s} \Robba_A^{[r, s]}; \; \; \; \Robba_A = \varinjlim_{s > 0} \Robba_A^{]0, s]}, $$ where $\Robba_{\qp}^{[r, s]}$ is the usual Banach ring of analytic functions on the rigid analytic annulus in the variable $T$ with radii $r \leq v_p(T) \leq s$ with coefficients in $\qp$. The Banach ring $\Robba_A^{[r,s]}$ is equipped with valuation $v^{[r,s]}$ defined by:
\[
v^{[r,s]} = \underset{r \leq v_A(z) \leq s}{\mathrm{inf}} v_A(f(z)) = \mathrm{min} \left(\underset{k \in \Z}{\mathrm{inf}}(v_A(a_k) + rk),\underset{k \in \Z}{\mathrm{inf}}(v_A(a_k) + sk) \right)
\]
for $f = \sum_{k \in \Z} a_k T^k \in \Robba_A^{[r,s]}$.

This definition admits an interpretation in terms of rings of analytic functions over a rigid space: if $I = [r,s]$ and if $\mathbb{A}^{1,\mathrm{rig}}_{\qp} = \mathrm{Sp}(\qp \langle T \rangle)$ denotes the affine rigid line with parameter $T$, then, noting $B_I$ the admissible open subset of $\mathbb{A}^{1,\mathrm{rig}}_{\qp}$ defined by $v_p(T) \in I$, we have a natural isomorphism $$ \Robba^I_A \cong \mathcal{O}(\mathrm{Sp}(A) \times B_I). $$ We can also interpret these rings in terms of Laurent series (power series $\infty \in I$) with coefficients in $A$ in the usual way. For $s <r_1$ we have an $A$-linear ring endomorphism $\varphi: \Robba_A^{[r, s]} \to \Robba_A^{[r/p, s / p]}$, sending $T$ to $(1 + T)^p - 1$, inducing an action of the operator $\varphi$ over $\Robba_A$ and we also have a continuous action of the group $\Gamma$, commuting with that of $\varphi$, whose action is given by the formula $\sigma_a(T) = (1 + T)^a - 1$, $a \in \zpe$, over all rings defined above.

\begin{lemma} [Lemme 1.3 (v) \cite{chen2013}] \label{flatnessR}
For every interval $I \subseteq ]0, \infty]$, the ring $\Robba^I_A$ is a flat $A$-module. In particular, $\Robba_A$ is flat over $A$.
\end{lemma}

\subsubsection{Locally analytic functions and distributions} The Robba ring $\Robba_A$ is well interpreted in terms of distributions and locally analytic functions. Define $\Robba_A^+ := \Robba_A^{]0, \infty]}$ which is stable under the action of $\varphi$ and $\Gamma$ (equipped with the subspace topology), and note $\Robba^-_A := \Robba_A / \Robba^+_A$ (with the induced action of $\varphi$ and $\Gamma$ equipped with the quoitent topology). The algebra of distributions with values in $A$ is defined as\footnote{A priori this is different to Definition\ref{defaltdisalf} where the relative distribution algebra for a general locally $\qp$-analytic group is defined. Lemma \ref{distbasechange} says that these definitions are equivalent.} 
\begin{equation} \label{disrefja}
\mathscr{D}(\zp, A) := \mathscr{D}(\zp, \qp) \widehat{\otimes}_\qp A, 
\end{equation}
where the tensor product in \eqref{disrefja} is independent of the choice of injective or projective tensor product (as $\mathscr{D}(\zp, \qp)$ is Fr\'echet and $A$ is Banach). Let $\mathrm{LA}(\zp, A)$ be the space of locally analytic functions on $\zp$ with values in $A$. If $\mu \in \mathscr{D}(\zp, A)$, it's Amice transform is defined as \[ \mathscr{A}_\mu = \sum_{n \in \N} \int_\zp {x \choose n} T^n \mu(x) \in \Robba^+_A.\] Finally, for $f \in \Robba_A$, we define its Colmez transform as (for all $x \in \zp$) 
\[ 
\phi_f(x) = \mathrm{res}_0((1 + T)^{-x} f(T) \frac{dT}{1 + T}) = \mathrm{res}_{0}((1+T)^{-x}f dt), 
\] where for $f = \sum_{n \in \Z} a_n T^n$, we put $\mathrm{res}_0(f dT) = a_{-1}$ (as usual we set $t:= \log (1+T)$). We then have the following result due to Chenevier, cf. \cite[Proposition 2.8]{chen2013} (cf. also \cite[Lemma 2.1.19]{kedlaya2014}), generalizing those of Colmez, cf. \cite[Th\'eor\`eme 2.3]{colmez2015} (cf. also \cite{padfothste}):

\begin{prop} \label{amice}  \leavevmode
\begin{itemize}
\item The Amice transform $\mu \mapsto \mathscr{A}_\mu$ induces a topological isomorphism $\mathscr{D}(\zp, A) \cong \Robba_A^+$. 
\item The Colmez transform $f \mapsto \phi_f(x)$ induces a topological isomorphism $ \Robba_A^- \cong \mathrm{LA}(\zp, A) \otimes \chi^{-1}$.
\item If $\mu \in \mathscr{D}(\zp, A)$ and $f \in \Robba_A$, then $\int_\zp \phi_f \cdot \mu = \mathrm{res}_0(\sigma_{-1}(\mathscr{A}_\mu) f \frac{dT}{1 + T})$.
\item We have a $(\varphi, \Gamma)$-equivariant short exact sequence
\[ 0 \to \mathscr{D}(\zp, A) \to \Robba_A \to \mathrm{LA}(\zp, A) \otimes \chi^{-1} \to 0. \]
\end{itemize}
\end{prop}

The Robba ring $\mathscr{R}_A$ is equipped with a left inverse of $\varphi$ constructed as follows: For $s < r_1$ the map $\oplus_{i = 0}^{p - 1} \Robba_A^{[r, s]} \to \Robba_A^{[r/p, s/p]}$ given by $(f_i)_{i = 0, \hdots, p - 1} \mapsto \sum_{i = 0}^{p - 1} (1 + T)^i \varphi(f_i)$ is a topological isomorphism and allows us to define $\psi: \Robba_A^{[r/p, s/p]} \to \Robba_A^{[r, s]}$ by $\varphi \circ \psi = p^{-1} \mathrm{Tr}_{\Robba_A^{[r/p, s/p]} / \varphi(\Robba_A^{[r, s]})}$. We also note $\psi: \Robba_A \to \Robba_A$ the induced operator, which is continuous, surjective and is a left inverse of $\varphi$.

\subsubsection{Multiplication by a function}

If $\mu \in \mathscr{D}(\zp, A)$ and $\alpha \in \mathrm{LA}(\zp, A)$, we define the distribution $\alpha \mu$ by the formula \[ \int_\zp \phi \cdot \alpha \mu := \int_\zp \alpha \phi \cdot \mu. \] If $a \in \zp, n \in \N$ and if we take $\alpha = \mathbf{1}_{a + p^n \zp}$ the characteristic function of the compact open $a + p^n \zp \subseteq \zp$, then we note $\mathrm{Res}_{a + p^n \zp}$ the multiplication by $\alpha$. Via the Amice transform this translates to
\[
\mathscr{A}_{\mathrm{Res}_{a + p^n \zp}(\mu)} = \mathrm{Res}_{a + p^n \zp}\mathscr{A}_{\mu},
\]
where the restriction map on the RHS translates to $\Res_{a + p^n \zp} = (1 + T)^a \varphi^n \circ \psi^n (1 + T)^{-a}$, cf. \cite[\S 2.1.1]{colmez2015}.

Since $\psi$ is surjective, we have a $\Gamma$-equivariant exact sequence \[ 0 \to \mathscr{D}(\zpe, A) \to \Robba_A^{\psi = 0} \to \mathrm{LA}(\zpe, A) \otimes \chi^{-1} \to 0. \]

\begin{lemm}
If $\mu \in \mathscr{D}(\zp, A)$ and $f \in \Robba_A$ then we have \[ \psi (\mathscr{A}_\mu) = \mathscr{A}_{\psi(\mu)} \text{ and } \phi_{\psi(f)} = \psi(\phi_f),\] where $\psi(\mu)$ is given by $\int_\zp \phi \cdot \psi(\mu) := \int_{p \zp} \phi(x/p) \cdot \mu$, and for any $\phi \in \mathrm{LA}(\zp, A)$, we set $\psi (\phi) (x) := \phi(px)$
\end{lemm}

\begin{proof}
In the case where $A$ is a finite extension of $\qp$, this is \cite[Proposition 2.2]{colmez2015}. In our setup the same proof carries over.


%

\end{proof}

The following corollary is now immediate.

\begin{coro}
If $\mu \in \mathscr{D}(\zp, A)$ (resp. $\phi \in \mathrm{LA}(\zp, A)$), the condition $\psi(\mu) = 0$ (resp. $\psi(\phi) = 0$) translates into $\mu$ (resp. $\phi$) being supported on $\zpe$.
\end{coro}

\subsubsection{The differential operator $\partial$}

We define an $A$-linear differential operator $\partial: \Robba_A \to \Robba_A$ by the formula \[ \partial f := (1 + T) \frac{df(T)}{dT}. \] This operator plays an important role in the subsequent constructions that we will consider.

\begin{lemma} \leavevmode
\begin{itemize}
\item If $f \in \Robba_A$ then $\phi_{\partial f}(x) = x \phi_f(x)$.
\item If $\mu \in \mathscr{D}(\zp, A)$, then $\partial \mathscr{A}_{\mu} = \mathscr{A}_{x \mu}$.
\item $\partial$ is bijective on $\Robba_A^{\psi = 0}$.
\end{itemize}
\end{lemma}

\begin{proof}
In the case when $A$ is a finite extension of $\qp$, this is \cite[Proposition 2.6 and Lemme 2.7]{colmez2015}. In our setup the same proof carries over.

\end{proof}

\subsubsection{$\qp$-Analytic sheaves and relative $(\varphi, \Gamma)$-modules}

A crucial notion developed by Colmez is that of an analytic sheaf. This plays a greater role in the study of $(\varphi, \Gamma)$-modules for Lubin-Tate extensions (in the sense of \cite{KisRenLtgrp}) associated to a finite extension $F \not= \qp$. In analogy with\cite[Defintion 1.6]{colmez2015}, we define:

\begin{defi}
Let $H$ be a locally $\qp$-analytic semi-group and $X$ an $H$-space (totally disconnected, compact space on which $H$ acts by continuous endomorphisms). An $H$-sheaf $\mathscr{M}$ over $X$ is the datum:
\begin{enumerate}
\item For every compact open $U \subset X$, a topological $A$-module $\mathscr{M} \boxtimes U$, with  $\mathscr{M} \boxtimes \emptyset = 0$
\item For each $U \subset V$ of compact opens, there are continuous $A$-linear restriction maps:
$$\Res^{V}_{U} \colon \mathscr{M} \boxtimes V \rightarrow \mathscr{M} \boxtimes U,$$
such that if $U = \cup_{i=1}^{n}U_i$ and $s_i \in \mathscr{M} \boxtimes U_i$ for $1 \leq i \leq n$, such that
$$\Res^{U_i}_{U_i \cap U_j} s_i = \Res^{U_j}_{U_i \cap U_j} s_j,$$
then there exists a unique $s \in \mathscr{M} \boxtimes U$, such that $\Res^{U}_{U_i} s = s_i$  for all $i$.
\item There are continuous $A$-linear isomorphisms: 
$$g_U \colon \mathscr{M} \boxtimes U \cong \mathscr{M} \boxtimes gU$$
for every $g \in H$ and $U$ compact open, satisfying the cocycle condition, $g_{hU} \circ h_U = (gh)_U$ for every $g,h \in H$ and $U$ compact open. Moreover for every compact open $U$, the morphism $g \mapsto g_{U}$ is a continuous morphism of the stabiliser $H_{U}$ (of $U$) to $\mathrm{Hom}_{A, \mathrm{cont}}(\mathscr{M} \boxtimes U)$.  
\end{enumerate}
\end{defi}

Since we will be primarily interested in attaching $H$-sheaves to $(\varphi,\Gamma)$-modules, we have the following definition, cf. \cite[Definition 2.2.12]{kedlaya2014}.

\begin{defi}\label{def:phigammodde}
Let $r \in (0,r_1)$. A $\varphi$-module over $\Robba_A^{]0,r]}$ is a finite projective $\Robba_A^{]0,r]}$-module $M^{]0,r]}$ equipped with an isomorphism
\[
M^{]0,r]} \otimes_{\Robba_A^{]0,r]},\varphi} \Robba_A^{]0,r/p]} \cong M^{]0,r]} \otimes_{\Robba_A^{]0,r]}} \Robba_A^{]0,r/p]}
\]
A $(\varphi,\Gamma)$-module over $\Robba_A^{]0,r]}$ is $\varphi$-module $M^{]0,r]}$ over $\Robba_A^{]0,r]}$ equipped with a commuting semilinear continuous action of $\Gamma$. A $(\varphi,\Gamma)$-module over $\Robba_A$ is the base change to $\Robba_A$ of a $(\varphi, \Gamma)$-module over $\Robba_A^{]0,r]}$ for some $r$. Let $\Phi\Gamma(\Robba_A)$ denote the category of $(\varphi, \Gamma)$-modules over $\Robba_A$. Morphisms are $\Robba_A$-linear morphisms commuting with the actions of $\varphi$ and $\Gamma$. 
\end{defi}

In order to equip an $H$-sheaf with an action of a Lie algebra (so that one can perform explicit calculations), the following definition beckons.

\begin{defi}\label{deflasheaf}
For $(H,X) \in \{ (P^{+},\zp),(G,\P^1) \}$, we say that an $H$-sheaf $\mathscr{M}$ over $X$ is $\qp$-analytic if for all open compact $U \subset X$, $\mathscr{M} \boxtimes U$ is a locally convex $A$-module of $A$-LF-type (cf. Definition \ref{def:frech}) and a continuous $\mathscr{D}(K,A)$-module for all open compact subgroups $K \subset H$, stabilizing $U$.  
\end{defi}

The point of Definition \ref{deflasheaf} is that a $(\varphi,\Gamma)$-module $\Delta$ over $\mathscr{R}_A$ naturally provides a $\qp$-analytic $P^+$-sheaf over $\zp$, which codifies its $(\varphi, \Gamma)$-structure, cf. \cite[\S 1.3.3]{colmez2015}. For $z \in \Delta$ one sets
\[
\begin{pmatrix} p^{k}a & b \\ 0 & 1 \end{pmatrix} \cdot z := (1+T)^b \varphi^k \circ \sigma_a(z) \text{ if } k \in \mathbf{N}, a\in \zpe, b\in \zp.
\]
If $U$ is an open compact of $\zp$, we can write $U$ as a finite disjoint union $\coprod_{i \in I} i + p^n\zp$ for large enough $n$ and we define $\mathrm{Res}_{U}$ by the formula
\[
\mathrm{Res}_{U} = \sum_{i \in I} \mathrm{Res}_{i + p^n\zp}
\]
where we set 
\[
\mathrm{Res}_{i +p^n\zp} = \begin{pmatrix}1 & i \\ 0 & 1 \end{pmatrix} \circ \varphi^n \circ \psi^n \circ \begin{pmatrix}1 & -i \\ 0 & 1 \end{pmatrix}.
\]
This turns out to be independent of choice as 
$$z = \sum_{i \text{ } \mathrm{mod}\text{ }p} \mathrm{Res}_{i + p\zp}z.$$
One then sets $\Delta \boxtimes U$ to be the image of $\Res_U$. The aim is to show that for a trianguline $(\varphi,\Gamma)$-module over $\mathscr{R}_A$, we can extend the corresponding $P^+$-sheaf over $\zp$ to a $G$-sheaf over $\P^1$. Moreover the global sections of the latter will cut out the locally analytic $G$-representation in $A$-modules that we are attempting to attach to $\Delta$. 

We have the associated definition for $(\varphi,\Gamma)$-module over $\mathscr{R}_A$.

\begin{defi}
A $(\varphi,\Gamma)$-module over $\mathscr{R}_A$ is said to be analytic if its associated $P^+$-sheaf over $\mathbb{Z}_p$ is analytic.  
\end{defi} 

\begin{rema} \label{analytic}
Every $(\varphi,\Gamma)$-module over $\mathscr{R}_A$ is $\qp$-analytic (as the base field is $\qp$), cf. \cite[Lemme 4.1]{Berpaddif} or \cite[Lemma 2.2.14(3)]{kedlaya2014}, for why the action of $\Gamma$ is locally analytic.
\end{rema}

\subsubsection{Multiplication by a character on $\Robba_A$}

Let $N \geq 0$. For $f \in \Robba_A \boxtimes \zpe = \Robba_A^{\psi = 0}$ and $i \in \zpe$, we write $f_i = \psi^N (1 + T)^{-i} f$ so that $$ f = \sum_{i \in (\Z / p^N \Z)^\times} (1 + T)^i \varphi^N f_i.$$ If $k \geq 0$, by the Leibnitz rule we can write \[ \partial^k f = \partial^k \left(\sum_{i \in (\Z / p^N \Z)^\times} (1 + T)^i \varphi^N(f_i) \right) = \sum_{i \in (\Z / p^N \Z)^\times} \sum_{j = 0}^k {k \choose j} i^{k-j} (1 + T)^i p^{Nj} \varphi^N (\partial^j f_i). \] This formula suggests the following proposition, which is a relative version of \cite[Proposition 2.9]{colmez2015}: 

\begin{prop} \label{cor:stup}
 If $\delta: \zpe \to A^\times$ is a locally analytic character and $f \in \Robba_A^{\psi = 0}$, the expression 
\[ \sum_{i \in \zpe (\text{mod }p^N)} \sum_{j = 0}^{+\infty} {\kappa(\delta) \choose j} \delta(i) i^{-j} (1+T)^i p^{Nj} \varphi^N(\partial^j f_i), \] where $\kappa(\delta) = \delta'(1)$ is the weight of $\delta$, converges in $\Robba_A^{\psi = 0}$ for $N$ big enough (depending only on $\delta$) to an element $m_\delta(f)$ that does not depend on $N$ or on the choice of representatives of $\zpe (\text{mod }p^N)$.

Moreover, the map $m_\delta: \Robba_A \boxtimes \zpe \to \Robba_A \boxtimes \zpe$ thus defined, is continuous, stabilizes $\Robba_A^+ \boxtimes \zpe$ and induces the multiplication by $\delta$ on $\mathscr{D}(\zpe,A)$ and $\mathrm{LA}(\zpe,A)$.  
\end{prop}

\begin{proof}
For $A = L$ a finite extension of $\qp$, this is \cite[Proposition 2.9]{colmez2015} (cf. also \cite[\S 3.2]{four}). Since a similar argument works for the general case, we only provide a sketch here. We start by recalling some easy and standard estimations. For $0 < r < s < r_2$ and $g \in \Robba_A^{[r, s]}$ we have
\begin{itemize}
\item $v^{[r/p, s/p]}(\varphi(g)) = v^{[r, s]}(g)$.
\item $v^{[pr , ps]}(\psi(g)) \geq v^{[r, s]}(g) - 1$.
\item $v^{[r, s]}(\sigma_a(g)) = v^{[r, s]}(g)$ for all $a \in \zpe$.
\item $v^{[r, s]}(\partial^k g) \geq v^{[r ,s]}(g) - k s$. 
\end{itemize}
By definition of the topology of the $LF$-space $\Robba_A$, we need to show that there exists an $s > 0$ such that, for all $0 < r < s$, the general term \[ f_j(\delta) = {\kappa(\delta) \choose j} \delta(i) i^{-j} (1+T)^i p^{Nj} \varphi^N(\partial^j f_i) \] of the series defining $m_\delta(f)$ goes to zero (in $\Robba_A^{[r,s]}$) as $j$ goes to $+ \infty$.

Observe that $f_j(\delta) \in \Robba_A^{]0, s]}$ whenever $s < r_N$ and $f \in \Robba_A^{]0, s]}$. We continue estimating the valuation of the terms appearing in the expression for $f_j(\delta)$. For any $0 < r < s < r_N$ we have
\begin{itemize}
\item $v^{[r, s]}((1 + T)^i) = v^{[r, s]}(i^{-j}) = v^{[r, s]}(\delta(i)) = 0$.
\item $v^{[r, s]}({\kappa(\delta) \choose j}) = v_A({\kappa(\delta) \choose j}) \geq j(\min(v_A(\kappa(\delta)), 0) - \frac{1}{p - 1})$. Note by $C_\delta = \min(v_A(\kappa(\delta)), 0) - \frac{1}{p - 1}$.
\item $v^{[r, s]}(\varphi^N(\partial^j f_i)) \geq v^{[r, s]}(f) - N - j p^N s$ (as is shown by an immediate calculation using all the estimations made in the last paragraph).
\end{itemize}
Putting all this together, we get \[ v^{[r, s]}(f_j(\delta)) \geq v^{[r, s]}(f) + j(C_\delta + N - N / j - p^N s). \]
Observe that this estimation does not depend on $r$. So for any $N > 0$  large enough (and $s < r_N$) such that $C_\delta + N - N / j - p^N s > 0$ \footnote{Take, for instance, any $N > - C_\delta + 1$.} and any $0 < r < s$, the general term $f_j(\delta)$ tends to zero in $\Robba_A^{[r, s]}$ as $j \to + \infty$ and thus the series converges. This completes the proof of the existence and continuity of $m_\delta(f)$.

A small calculation shows that the value $C_\delta$ can be bounded only in terms of the valuation of $\delta(1 + 2p)$ and that, for $f$ fixed, the formula defines a rigid analytic function on the $A$-points of the rigid analytic space $\mathfrak{X}$ whose $A$ points parametrize continuous characters $\mathrm{Hom}_{\rm cont}(\zpe, A^\times)$ \footnote{For the existence of $\mathfrak{X}$, cf. \cite[Proposition 6.1.1]{kedlaya2014}.} with values in $\Robba_A^{[r, s]}$. Using the Zariski density of the points $x \mapsto x^k$ in $\mathfrak{X}$ \footnote{i.e that any rigid analytic function on $\mathfrak{X}$ vanishing at those points vanish.}, one shows, on the one hand the independence of $m_\delta(f)$ on the choice of a system of representatives of $\zpe (\text{mod }p^N)$, and on the other hand, using the fact that, if $\delta(x) = x^k$, then $m_\delta(f) = \partial^k f$ and that $\partial$ extends to $\Robba_A$ the multiplication by $x$ on $\mathscr{D}(\zp, A)$ and on $\mathrm{LA}(\zp, A)$, that $m_\delta(f)$ extends also the multiplication by $\delta$ for any locally analytic character $\delta$. This completes the proof.
\end{proof}

\subsection{Duality} \label{sec:dark}
Let $\Phi\Gamma(\Robba_A)$ denote the category of $(\varphi, \Gamma)$-modules over $\Robba_A$, cf. Definition \ref{def:phigammodde}. In the process of constructing $\Pi(\Delta)$ we will need the notion of duality. If $\Delta \in \Phi\Gamma(\Robba_A)$, we set $\check{\Delta} = \mathrm{Hom}_{\Robba_A}(\Delta, \Robba_A(\chi))$ and denote by
\[
\langle \text{ },\text{ } \rangle \colon \check{\Delta} \times \Delta \to \Robba_A(\chi)
\]
the induced pairing. We impose a $(\varphi,\Gamma)$-structure on $\check{\Delta}$ by setting
\[
\langle g \cdot \check{z}, g \cdot z \rangle := g \cdot \langle \check{z},z \rangle
\] 
for all $\check{z} \in \check{\Delta}$, $z \in \Delta$ and $g \in \{\sigma_a, \varphi \}$. Note that $\check{\Delta} \in \Phi\Gamma(\Robba_A)$. 

The pairing $\langle \text{ },\text{ } \rangle$ defines a new pairing
\begin{align*}
\lbrace \text{ },\text{ } \rbrace \colon \check{\Delta} \times \Delta &\to A\\
(\check{z},z) &\mapsto \mathrm{res}_{0}(\langle \sigma_{-1}(\check{z}),z \rangle),
\end{align*} 
where $\mathrm{res}_{0}\left( \sum_{k \in \mathbf{Z}} a_kT^k dT \right) = a_{-1}$. Assuming that $\Delta$ is free over $\Robba_A$, the point is that the pairing $\lbrace \text{ },\text{ } \rbrace$ identifies $\check{\Delta}$ and $\Delta$ as topological duals of $\Delta$ and $\check{\Delta}$ respectively, cf. \cite[Proposition III.2.3]{colmirabol}. 

\subsection{Principal series} \label{subpinc}

Let $\delta_1, \delta_2 \colon \qpe \to A^\times$ be two continuous characters. We define $B_A(\delta_1,\delta_2)$ to be the space of locally analytic functions $\phi \colon \qp \to A$, such that $\delta(x)\phi\left(\frac{1}{x} \right)$ extends to an analytic function on a neighbourhood of $0$. We equip $B_A(\delta_1,\delta_2)$ with an action of $G$ defined by 
\[
\left( \begin{pmatrix} a & b \\ c & d \end{pmatrix} \cdot \phi \right)(x) = \delta_2(ad-bc)\delta(a-cx)\phi \left(\frac{dx-b}{a-cx} \right).
\]

One can show that $B_A(\delta_1,\delta_2) = \mathrm{Ind}_{\overline{B}}^{G}(\delta_1\chi^{-1} \otimes \delta_2)$ (where $\overline{B}$ is the lower-half Borel subgroup of $G$). Here $\delta_1\chi^{-1} \otimes \delta_2$ is viewed as the character $\begin{pmatrix} a & 0 \\ c & d \end{pmatrix} \mapsto \delta_1\chi^{-1}(a)\delta_2(d)$. For the definition of $ \mathrm{Ind}_{\overline{B}}^{G}(\delta_1\chi^{-1} \otimes \delta_2)$, cf. Remark \ref{rem:lains}. The topology of $B_A(\delta_1, \delta_2)$ is by definition the topology coming from $\mathrm{LA}(G/\overline{B}, A)$, cf. Definition \ref{def:repla}. This makes $B_A(\delta_1, \delta_2)$ into a Hausdorff, complete, locally convex $A$-module, cf. Definition \ref{def:right} and Lemma \ref{lem:magdimo}. 

The strong topological dual of $B_A(\delta_1,\delta_2)$, cf. Definition \ref{def:stdua}, identifies with a space of distributions on $\P^1$ equipped with an action of $G$ defined by 
\[
\int_{\P^1} \phi \begin{pmatrix} a & b \\ c & d \end{pmatrix} \cdot \mu = \delta_1^{-1}\chi(ad-bc)\int_{\P^1} \delta(cx+d)\phi \left(\frac{ax+b}{cx+d} \right) \mu(x). 
\]

\subsection{The $G$-module $\Robba_A(\delta) \boxtimes_\omega \P^1$} \label{RtimesP1}

Suppose $\Delta \in \Phi\Gamma(\Robba_A)$ is of rank 2 and trianguline. In this section, we follow \cite[\S 4.3]{colmez2015}, to construct the $G$-modules $\Robba_A(\delta) \boxtimes_\omega \P^1$ which will be the constituents of the $G$-module $\Delta \boxtimes_\omega \P^1$. Once $\Delta \boxtimes_\omega \P^1$ is constructed, we will see that one of its constituents is the representation $\Pi(\Delta)$, which we are searching for. Since we are only interested in constructing representations of $\mathrm{GL}_2(\qp)$, the constructions from \cite{colmez2015} (where representations of $\mathrm{GL}_2(F)$, for $F / \qp$ a finite extension, are constructed) simplify considerably.

We start by recalling a structure result for arithmetic families of $(\varphi, \Gamma)$-modules.

\begin{prop} [Theorem 3.1.1,\cite{kedlaya2014}] \label{kpxgamma} Let $A$ be a $\qp$-affinoid algebra and let $\Delta \in \Phi \Gamma(\Robba_A)$. There exists $r(\Delta)$ such that, for any $0 < r < r(\Delta)$, $\gamma - 1$ is invertible on $(\Delta^{]0, r]})^{\psi = 0}$, and the $A[\Gamma, (\gamma - 1)^{-1}]$-module structure on  $(\Delta^{]0, r]})^{\psi = 0}$ extends uniquely by continuity to a $\Robba_A^{]0, r]}(\Gamma)$-module structure for which  $(\Delta^{]0, r]})^{\psi = 0}$ is finite projective of rank $d = \mathrm{rank}_{\Robba_A} \; \Delta$. Moreover, if $\Delta$ is free over $\Robba_A$, then $(\Delta^{]0, r]})^{\psi = 0}$ admits a set of $d$ generators over $\Robba_A^{]0, r]}(\Gamma)$.
\end{prop}

\begin{rema} \leavevmode \label{rem:che}
\begin{itemize}
\item The proof of the last statement of Proposition \ref{kpxgamma} can be found in the the proof \cite[Theorem 3.1.1]{kedlaya2014}. In general there exists a finite projective $\Robba_A$-module $N$ such that $\Delta \oplus N$ is free of rank $m$ over $\Robba_A$ and the proof in loc.cit. shows that $(\Delta^{]0, r_n]})^{\psi = 0}$ admits a set of $m$ generators over $\Robba_A(\Gamma)$.
\item Taking direct limits we also get that $\Delta^{\psi = 0}$ is a finite projective module over $\Robba(\Gamma)$ of rank $d$, admitting a set of $m$ generators ($m = d$ if $\Delta$ is free).
\item In the case when $\Delta = \Robba_A$, one can show that $\Delta^{\psi = 0}$ is a free module of rank one over $\Robba_A(\Gamma)$ generated by $(1+T)$, cf. \cite[Proposition 2.14 and Remarque 2.15]{chen2013}. 
\end{itemize}
\end{rema}

If $\Delta = \Robba_A$ we have a short exact sequence of $\Gamma$-modules \[ 0 \to (\Robba^+_A)^{\psi = 0} \to \Robba_A^{\psi = 0} \to (\Robba^-_A)^{\psi = 0} \to 0. \] Recall that $(\Robba^+_A)^{\psi = 0} = \Robba^+_A \boxtimes \zpe \cong \mathscr{D}(\zpe, A)$ via the Amice transform, and that we have an involution $w_*$ on it given by \[ \int_\zpe \phi(x) \cdot w_* \mu = \int_\zpe \phi(x^{-1}) \cdot \mu. \] The involution is $\Gamma$-anti-linear in the sense that we have $w_* \circ \sigma_a = \sigma_a^{-1} \circ w_*$ for all $a \in \zpe$. We denote by $\iota: \Robba_A(\Gamma) \to \Robba_A(\Gamma)$ the involution defined by $\sigma_a \mapsto \sigma_a^{-1}$ on $\Gamma$.

\begin{lemma} \label{lem3.50}
There exists a unique $\Robba_A(\Gamma)$-anti-linear involution $w_*$ with respect to $\iota$ \footnote{i.e satisfying $w_* \circ \lambda = \iota(\lambda) \circ w_*$ for all $\lambda \in \Robba_A(\Gamma)$.} on $\Robba_A \boxtimes \zpe$ extending that on $\Robba_A^+ \boxtimes \zpe$. Moreover, $w_*$ satisfies
\begin{itemize}
\item $w_* = \partial w_* \partial$.
\item $\nabla \circ w_* = -w_* \circ \nabla$.
\item $w_* \circ \mathrm{Res}_{a + p^n \zp} = \mathrm{Res}_{a^{-1} + p^n \zp} \circ w_*$, for all $a \in \zpe$, $n \geq 1$.
\end{itemize} 
\end{lemma}

\begin{proof}

Take a generator $e$ of the free $\Robba_A(\Gamma)$-module $\Robba_A \boxtimes \zpe$ of rank one such that $e \in \Robba_A^+ \boxtimes \zpe$ (e.g $(1 + T)$, cf. Remark \ref{rem:che}). This forces \[ w_*(\lambda \cdot e) = \iota(\lambda) \cdot w_*(e) \] for every $\lambda \in \Robba(\Gamma)$, where $w_*(e) \in \Robba_A^+ \boxtimes \zpe$ is well defined since $e \in \Robba^+_A \boxtimes \zpe$.

For the rest of the properties we can use \cite[Lemme 2.14]{colmez2015} which shows that they hold for $w_*$ acting on $\Robba^+_A \boxtimes \zpe$ (the same proof carries over for any $A$). We only show the first one, the other two being immediate. Let $z = \lambda \cdot e \in \Robba_A \boxtimes \zpe$ for some $\lambda \in \Robba_A(\Gamma)$. We have \[ \partial \circ w_* \circ \partial (\lambda \cdot e) = \chi(\lambda) \; \partial \circ w_*(\lambda \cdot \partial e) =  \chi(\lambda) \; \partial ( \iota(\lambda) \cdot w_* (\partial e)) = \iota(\lambda) \cdot \partial \circ w_* \circ \partial (e) = \iota(\lambda) \cdot w_*(e) = w_*(z). \]
\end{proof}

The following gives a relation between $m_{\delta}$ and $w_*$.
\begin{lemm}\label{lem:preinvres}
If $\delta \colon \zpe \to A^{\times}$ is a continuous character, then
\[
m_{\delta} \circ w_{*} = w_{*} \circ m_{\delta^{-1}}.
\]
\end{lemm}

\begin{proof}
By Lemma \ref{lem3.50}, the identity is true for $\delta = x^k$ for all $k \in \mathbf{Z}$ (this is because $\partial^k = m_{x^k}$). Now the functions $\delta \mapsto m_{\delta} \circ w_{*}$ and $\delta \mapsto w_{*} \circ m_{\delta^{-1}}$ are rigid functions and coincide on $x^k$ for all $k \in \mathbf{Z}$. Thus they coincide for all $\delta$.  
\end{proof}

If $\Delta \in \Phi\Gamma(\Robba_A)$, $\omega \colon \qpe \to A^\times$ (for applications $\omega$ will be $\delta_1\delta_2\chi^{-1}$ for any two continuous characters $\delta_1, \delta_2 \colon \qpe \to A^{\times}$) is a locally analytic character and $\iota$ is an involution on $\Delta \boxtimes \zpe$, we can define a module $\Delta \boxtimes_{\omega, \iota} \P^1$ (cf. \cite[\S 3.1.1]{colmez2015} for details) equipped with an action of a group $\tilde{G}$ generated freely by a group $\tilde{Z}$ isomorphic to the torus $\{ {\matrice a 0 0 a}, a \in \qpe \}$ (acting on $\Delta \boxtimes_{\omega, \iota} \P^1$ via multiplication by $\omega$), a group $\tilde{A}^0 \cong \zpe$ (encoding the action of $\sigma_a$), a group $\tilde{U} \cong p \zp$ (encoding the multiplication by $(1 + T)^b$, $b \in p \zp$) and the elements ${\matrice p 0 0 1}$ (encoding the action of $\varphi$) and $w = {\matrice 0  1 1 0}$. Precisely, the $\tilde{G}$-module $\Delta \boxtimes_{\omega, \iota} \P^1$ is defined as \[ \Delta \boxtimes_{\omega, \iota} \P^1 = \{ (z_1, z_2) \in \Delta \times \Delta \; : \; \mathrm{Res}_\zpe(z_1) = \iota(\Res_{\zpe}(z_2)) \} \] and the action of $\tilde{G}$ on an element $z = (z_1, z_2) \in \Delta \boxtimes_{\omega, \iota} \P^1$ is described by the following formulae: 
\begin{itemize}
\item ${\matrice 0 1  1 0} \cdot z = (z_2, z_1)$.
\item ${\matrice a 0 0 a} \cdot z = (\omega(a) z_1, \omega(a) z_2)$, $a \in \qpe$.
\item ${\matrice a 0 0 1} \cdot z = ( {\matrice a 0 0 1} z_1, \omega(a) {\matrice {a^{-1}} 0 0 1} z_2)$, $a \in \zpe$.
\item If $z' = {\matrice p 0 0 1} z$, then $\mathrm{Res}_{p \zp} z' = {\matrice p 0 0 1} z_1$ and $\mathrm{Res}_{\zp} w z' = \omega(p) \psi(z_2)$.
\item If $b \in p \zp$ and $z' = {\matrice 1 b 0 1}\cdot z$ then $\mathrm{Res}_\zp z' = {\matrice 1 b 0 1} \cdot z_1$ and $\mathrm{Res}_{p \zp} wz' = u_b(\mathrm{Res}_{p \zp}(z_2))$, where \[ u_b = \omega(1 + b) {\matrice 1 {-1} 0 1} \circ \iota \circ {\matrice {(1 + b)^{-2}} {b(1 + b)^{-1}} 0 1} \circ \iota \circ {\matrice 1 {(1 + b)^{-1}} 0 1 } \text{ on } \Delta^+ \boxtimes p\zp. \]
\end{itemize}

\begin{lemma}
The functor $M \mapsto M \boxtimes_{\omega, \iota} \P^1$ is an exact functor from $P^+$-modules living on $\zp$ to $\tilde{G}$-modules living on $\P^1(\qp)$.
\end{lemma}

\begin{proof}
Let $0 \to M' \to M \to M'' \to 0$ be a short exact sequence of $P^+$-modules. We claim that we have an exact sequence $0 \to M' \boxtimes_{\omega, \iota} \P^1 \to M \boxtimes_{\omega, \iota} \P^1 \to M'' \boxtimes_{\omega, \iota} \P^1$. Let's show that the last arrow is surjective (for exactness in the middle and injectivity, the proof is similar). Let $(c, d) \in M'' \boxtimes_{\omega, \iota} \P^1$ and $(a, b) \in M \times M$ be any lifting. The element $\mathrm{Res}_\zpe a - \iota(\mathrm{Res}_\zpe b)$ maps to zero in $M''$ and so there exists an element $x \in M'$ such that $\mathrm{Res}_\zpe a - \iota(\mathrm{Res}_\zpe b) = x$. The element $(a - x, b) \in M \boxtimes_{\omega, \iota} \P^1$ maps then to $(c, d)$.
\end{proof}

For $\delta_1, \delta_2: \qpe \to A^\times$ continuous characters, recall that we have set $\delta = \delta_1 \delta_2^{-1} \chi^{-1}, \omega = \delta_1 \delta_2 \chi^{-1}$. We will soon be working with $\Delta \in \Phi\Gamma(\Robba_A)$ which is an extension of $\Robba_A(\delta_2)$ by $\Robba_A(\delta_1)$. Thus we need to \emph{twist} appropriately the current involution $w_*$, cf. Lemma \ref{lem3.50}, on $\Robba_A \boxtimes \zpe$.  We define an involution $\iota_{\delta_1, \delta_2}$\footnote{The fact that $i_{\delta_1, \delta_2}$ is an involution follows from Lemma \ref{lem:preinvres}.} acting on the module $\Robba_A(\delta_1) \boxtimes \zpe$ by the formula\footnote{By Proposition \ref{cor:stup}, this formula is well defined.} \[ \iota_{\delta_1, \delta_2}(f \otimes \delta_1) = (\delta_1(-1) w_* \circ m_{\delta^{-1}}(z)) \otimes \delta_1. \]

 We get in this way a module $\Robba_A(\delta_1) \boxtimes_{\omega, \iota_{\delta_1, \delta_2}} \P^1$, that we simply note $\Robba_A(\delta_1) \boxtimes_{\omega} \P^1$, equipped with an action of $\tilde{G}$. We show in what follows that this action of $\tilde{G}$ factorises through $G$.

Recall that, for a finite extension $L$ of $\qp$ and $\Delta \in \Phi\Gamma(\Robba)$ of rank 1 or 2 (if it is of rank 2, assume it is also trianguline), $\omega: \qpe \to L^\times$ locally analytic and $\iota$ an involution on $\Delta \boxtimes \zpe$, we have (cf. \cite{colmez2015}) a $G$-module $\Delta \boxtimes_\omega \P^1$. the following lemma shows that, for $\Delta = \Robba_A(\delta_1)$, our construction specializes to that of Colmez.

\begin{lemma}\label{lem:cast}
Let $\mathfrak{m} \subseteq A$ be a maximal ideal of $A$, $L = A/\mathfrak{m}$, $\Delta \in \Phi\Gamma(\Robba_A)$. Assume that $\Delta$ is either of rank 1 or 2 (if it is of rank 2, assume it is also trianguline) and $\omega$, $\iota$ be as above. Then the $\tilde{G}$-module $(\Delta \boxtimes_\omega \P^1) \otimes_A L$ is canonically isomorphic to $(\Delta \otimes_A L) \boxtimes_{\omega \otimes L} \P^1$.
\end{lemma}

\begin{proof}
This is immediate. The uniqueness of both involutions $w_*$ defined in Lemma \ref{lem3.50} above and in \cite[Proposition 2.19]{colmez2015} shows that they both coincide (since they do on $\Robba^+ \boxtimes \P^1$).
\end{proof}

The following result provides a link between the $- \boxtimes_{\omega} \P^1$ construction and principal series, cf. \S \ref{subpinc}

\begin{lemma}\label{lem:princser}
We have\footnote{By $B_A(\delta_1, \delta_2)^*$ we mean $\mathrm{Hom}_{A,\mathrm{cont}}(B_A(\delta_1, \delta_2),A)$ equipped with the strong dual topology, cf. Defintition \ref{def:stdua}, where $(-)^*$ is denoted by $(-)'_b$ there.} 
\begin{itemize} 
\item $\Robba^+_A(\delta_1) \boxtimes_\omega \P^1 \cong B_A(\delta_2, \delta_1)^* \otimes \omega$
\item $\Robba^-_A(\delta_1) \boxtimes_\omega \P^1 \cong B_A(\delta_1, \delta_2)$.
\end{itemize}
Moreover $\Robba^+_A(\delta_1) \boxtimes_\omega \P^1$ and $\Robba^-_A(\delta_1) \boxtimes_\omega \P^1$ are $\qp$-analytic sheaves. 
\end{lemma}

\begin{proof}
This is essentially \cite[Corollaire 4.11]{colmez2015}. The same proof carries over with $A$ in place of $L$ (since one only checks that both actions of $G$ coincide and the coefficient ring plays no role). The last part follows from Lemma \ref{lem:prinbag}. 
\end{proof}

For the rest of this paper we note $\Robba_A(\delta_1,\delta_2)$ to be the $\overline{P}^+$-module \footnote{We warn the reader that the module we call $\Robba_A(\delta_1,\delta_2)$ is not the one noted in the same way in \cite[\S 4.3.2]{colmez2015}. In our notation, $\Robba_A(\delta_1,\delta_2)$ corresponds to the module $\Robba_A(\delta_1,\delta_2, \eta)$ for $\eta = 1$ as defined in \cite[\S 5.6]{colmez2015}. }
\[
\Robba_A(\delta_1,\delta_2):= (\Robba_A(\delta_1) \boxtimes_{\omega} \zp) \otimes \delta_2^{-1}. 
\]
We set $\Robba_A^{+}(\delta_1, \delta_2)$ the sub-$\overline{P}^+$-module of $\Robba_A(\delta_1, \delta_2)$ corresponding to $\Robba_A^+$, and $\Robba_A^{-}(\delta_1, \delta_2)$ to be the quotient of $\Robba_A(\delta_1, \delta_2)$ by $\Robba_A^{+}(\delta_1, \delta_2)$. 

\begin{rema}
As $A^+$-modules, $\Robba_A(\delta_1, \delta_2)$, $\Robba_A^+(\delta_1, \delta_2)$ and $\Robba_A^-(\delta_1, \delta_2)$ are respectively isomorphic to $\Robba_A(\delta_1\delta_2^{-1})$, $\Robba_A^{+}(\delta_1\delta_2^{-1})$ and $\mathrm{LA}(\zp,A) \otimes \delta$. The technical heart of this paper is to compare the $\overline{P}^+$-cohomology of $\Robba_A(\delta_1, \delta_2)$ and the $A^+$-cohomology of $\Robba_A(\delta_1\delta_2^{-1})$, cf. \S \ref{sec:extgrpact}. 
\end{rema}

The following is the main result of this section, which is a relative version of \cite[Proposition 4.12]{colmez2015}.

\begin{prop} \label{actionG}
The action of $\tilde{G}$ on $\Robba_A(\delta_1) \boxtimes_\omega \P^1$ factorises through $G$ and we have an exact sequence of $G$-modules \[ 0 \to B_A(\delta_2, \delta_1)^* \otimes \omega \to \Robba_A(\delta_1) \boxtimes_\omega \P^1 \to B_A(\delta_1, \delta_2) \to 0. \] 
Moreover $\Robba_A(\delta_1) \boxtimes_\omega \P^1$ is a $\qp$-analytic sheaf. 
\end{prop}

\begin{proof}
We reduce the result to the case of a point, cf. \cite[Proposition 4.12]{colmez2015}, using an inductive argument on the index $i \geq 0$ of nilpotence of $A$.

Suppose first that $i = 0$, i.e that $A$ is reduced. Take $z \in \Robba_A(\delta_1) \boxtimes_\omega \P^1$ and $g$ in the kernel of $\tilde{G} \to G$. We need to show that $(g - 1) z = (z_1,z_2) = 0$. Let $\mathfrak{m} \subseteq A$ be any maximal ideal of $A$ and note $L = A / \mathfrak{m}$. Since $(\Robba_A(\delta_1) \boxtimes_\omega \P^1) \otimes_A L = \Robba_L(\delta_1) \boxtimes_\omega \P^1$, then we know by \cite[Proposition 4.12]{colmez2015} that $z_i = 0$ mod $\mathfrak{m}$. If we write $z_i = \sum_{n \in \Z} a_{n,i} T^n$, $i = 1, 2$, this means that $a_{n,i} = 0$ mod $\mathfrak{m}$ and hence, since this holds for every maximal ideal $\mathfrak{m}$ and since $A$ is reduced, we deduce that $a_{n,i} = 0$ for every $n$ and hence $z_i = 0$ as desired.

Suppose now the result is true for every affinoid algebra of index of nilpotence $\leq j$ and let $A$ be an affinoid algebra whose nilradical $N$ satisfies $N^{j + 1} = 0$ and $g$ be in the kernel of $\tilde{G} \to G$. We have the following short exact sequence \[ 0 \to (\Robba_{A / N} \boxtimes_\omega \P^1) \otimes_{A / N} N^j \to \Robba_A \boxtimes_\omega \P^1 \to \Robba_{A / N^j} \boxtimes_\omega \P^1 \to 0. \] By the base case of a reduced affinoid algebra and by the inductive hypothesis, the element $g - 1$ induces a linear endomorphism of the short exact sequence above which vanishes on $(\Robba_{A / N} \boxtimes_\omega \P^1) \otimes_{A / N} N^j$ and $\Robba_{A / N^j} \boxtimes_\omega \P^1$ respectively. Therefore it vanishes on $\Robba_{A} \boxtimes \P^1$, which shows the desired result.
 
%

For the second part, we first observe that, if we call $K_m = {\matrice {1 + p^m \zp} {p^m \zp} {p^m \zp} {1 + p^m \zp}}$, then the decomposition $K_m = {\matrice 1 0 {p^m \zp} {1 + p^m \zp}} {\matrice {1 + p^m \zp} {p^m \zp} 0 1}$ \footnote{This decomposition follows by noting that ${\matrice a b c d} = {\matrice 1 0 {c a^{-1}} {d - bc a^{-1}} } {\matrice a b 0 1}$, for ${\matrice a b c d} \in K_m$.} shows that it is enough to show the existence of an action of the distribution algebra of $\overline{U}^m = {\matrice 1 0 {p^m \zp} 1}$ for some $m \geq 0$.

We claim that, as a consequence of the identity $H^1_{\an}(\overline{U}^1, \Robba_A(\delta_1, \delta_2)) \cong H^1(\overline{U}^1, \Robba_A(\delta_1, \delta_2))$ of Proposition \ref{prop:lazcompcom}, the action of $\overline{U}^1$ on $\Robba_A(\delta)$ is locally analytic and hence extends to a separately continuous action of $\mathscr{D}(\overline{U}^1,A)$ (since $\Robba_A$ is barrelled it will also be jointly continuous, cf. \cite[\S 0.3.11]{Emertonred}). Indeed, calling $M = \Robba_A(\delta_1, \delta_2)$, which is an $LF$-space, we need to show that, for any $m \in M$, the orbit map $o_m \colon \overline{U}^1 \to M$ is locally analytic. By the definition of a locally analytic function, one reduces to the case where $M$ is Banach.

Consider now any continuous 1-coycle $c \colon \overline{U}^1 \to M$ such that $c(\tau) = m$. Then $c$ defines a function (which we continue to denote by $c$)
\[ c \colon \zp \to M : \;\;\; a \mapsto c(\tau^a) = \frac{\tau^a-1}{\tau-1}c(\tau) \]
which is, by Proposition \ref{prop:lazcompcom}, cohomologous to a locally analytic $1$-cocycle. Since 1-coboundaries are trivially locally analytic, it follows that $c$ is locally analytic. By expanding $\tau^a = \sum_{n \geq 0} \binom{a}{n}(\tau-1)^n$ one gets that \[ c(a) = \sum_{n \geq 1} \binom{a}{n} (\tau - 1)^{n-1} m, \] which shows that the Mahler coefficients of the function $c$ are given by $a_n(c) = (\tau - 1)^{n-1} m$. But the Mahler coefficients of the orbit map are nothing but $a_n(o_m) = (\tau - 1)^n m$, thus showing that $o_m$ is a locally analytic function, completing the proof.
\end{proof}

\section{Cohomology of $(\varphi, \Gamma)$-modules} \label{sec:cohphigammch}

In this section we recalculate some results of \cite{chen2013} using $(\varphi, \Gamma)$-cohomology. We calculate higher cohomology groups studied in \cite[\S 5]{colmez2015}, in preparation to extend the results in loc.cit. to the affinoid setting. Let us begin by recalling the definition of analytic cohomology. Let $H$ be a $\qp$-analytic semi-group (e.g. $A^+$, $\overline{P}^+$, $G$). Let $M$ be a complete, Hausdorff locally convex $A$-module (cf. Definition \ref{def:right}) with the structure of a seperately continuous $A$-linear $\mathscr{D}(H,A)$-module (i.e. $M$ is an object of the category $\mathscr{G}_{H,A}$, cf. Definition \ref{def:lacatlur}). We note $\mathrm{LA}^{\bullet}(H,M)$ to be the complex

\[
0 \to \mathrm{LA}^{0}(H,M) \xrightarrow{d_1} \mathrm{LA}^{1}(H,M) \xrightarrow{d_2} \cdots,
\]
where $\mathrm{LA}^{n}(H,M) := \mathrm{LA}(H^n,M)$ and $d_{n+1}$ is the differential

\[
d_{n+1}c(g_0,\ldots,g_n) = g_{0}\cdot c(g_1,\ldots,g_n) + \sum_{i=0}^{n-1} (-1)^{i+1}c(g_0, \ldots,g_{i}g_{i+1}, \ldots g_n) + (-1)^{n+1}c(g_0,\ldots,g_{n-1}).
\] 
Throughout $H^{i}_{\an}(H,M)$ will denote the $i$th cohomology group of this complex. For a detailed introduction to $H^{*}_{\an}$ (although in a slightly different setting) we refer the reader to the paper of Kohlhasse, cf. \cite{kohl2011}. Finally $H^i(H,M)$ will denote continuous (semi-)group cohomology.  

For $n \in  \N$, denote by $U^{n} = \begin{pmatrix}
                                           1 & p^{n}\zp \\
																					 0 & 1
																					 \end{pmatrix}$
\begin{lemma} \label{cohom1}
If $\Delta \in \Phi\Gamma(\mathscr{R}_A)$, then $H^{i}_{\an}(U^{n},\Delta) = 0 \hspace{2mm}\forall n \in \N$, if $i=0,1$. 
\end{lemma}

\begin{proof}
For $i=0$, we note that $H^{0}_{\an}(U^{n},\Delta) = \Delta^{(1+T)^{p^n}=1}$. For $i=1$, we have a map
$$H^{1}_{\an}(U^{n},\Delta) \hookrightarrow H^{1}(U^{n},\Delta),$$
since the continuous 1-coboundaries are in correspondence with the locally analytic 1-coboundaries. Finally $U^n$ is procyclic and so
$$H^{1}(U^{n},\Delta) = \Delta/\left( (1+T)^{p^n}-1 \right) = 0,$$
as $(1+T)^{p^n}-1$ is invertible in $\mathscr{R}_A$. 
\end{proof}

Denote by $A^{0} = \begin{pmatrix}
																                             \zpe & 0\\
																														 0 & 1
																				
																											 \end{pmatrix}$.
and let $\Delta$ be a $(\varphi,\Gamma)$-module over $\mathscr{R}_A$. We construct a natural map
$$\Theta^{\Delta}: \Ext^{1}(\mathscr{R}_A,\Delta) \rightarrow H^{1}_{\an}(A^+,\Delta).$$
Let $\widetilde{\Delta}$ be an extension of $\mathscr{R}_A$ by $\Delta$ and let $e \in \widetilde{\Delta}$ be a lifting of $1 \in \mathscr{R}_A$. Then $g \mapsto (g-1)e$, $g \in A^+$, is an analytic 1-cocycle and induces an element of $H^{1}_{\an}(A^+,\Delta)$ independent of the choice of $e$. Thus we obtain the desired map. 

\begin{prop} \label{prop:ext1h1a}
For any $(\varphi,\Gamma)$-module $\Delta$ over $\mathscr{R}_A$, $\Theta^{\Delta}$ is an isomorphism. 
\end{prop} 

\begin{proof}
For injectivity of $\Theta^{\Delta}$, let $\widetilde{\Delta}$ be an extension of $\mathscr{R}_A$ by $\Delta$ in the category of $(\varphi,\Gamma)$-modules whose image under $\Theta^{\Delta}$ is zero. Let $e \in \widetilde{\Delta}$ be a lifting of $1 \in \mathscr{R}_A$. Then there exists $d \in \Delta$, such that $(g-1)e = (g-1)d$ for all $g \in A^+$. Then $g(e-d) = e-d$ for all $g \in A^+$ and thus $\widetilde{\Delta} = \Delta \oplus \mathscr{R}_A$ as a $(\varphi,\Gamma)$-module. For surjectivty of $\Theta^{\Delta}$, given a 1-cocycle $g \mapsto c(g) \in \Delta$, we can extend the $(\varphi,\Gamma)$-module structure on $\Delta$ to the $\mathscr{R}_A$-module $\widetilde{\Delta} = \Delta \oplus \mathscr{R}_{A}e$, such that $\varphi(e) = e + c(\varphi)$ and $\gamma(e) = e + c(\gamma)$ for $\gamma \in \Gamma$. 
\end{proof}

Next we relate $H^{i}_{\an}(A^+,\Delta)$ to a Lie-algebra cohomology, where calculations can be made explicit. We denote by $\Phi^{+}$ the semi-group $\begin{pmatrix}
                                                  p^\N & 0 \\
																									0 & 1
																									\end{pmatrix}$, where $\varphi = \begin{pmatrix}
																									                              p & 0 \\
																																								0 & 1
																																								\end{pmatrix}$. We have that $A^+ = \Phi^{+} \times A^{0}$ (this decomposition breaks up the $\varphi$-action and the $\Gamma$-action). For $\Delta$ a $(\varphi,\Gamma)$-module over $\mathscr{R}_A$, we denote by $H^{i}_{\Lie}(A^+,\Delta)$ to be the cohomology groups of the complex:
																																												
$$0 \longrightarrow \Delta \xrightarrow{x \mapsto (\nabla x,(\varphi-1)x)} \Delta \oplus \Delta \xrightarrow{(a,b) \mapsto (\varphi-1)a-\nabla b} \Delta \longrightarrow 0.$$
We will be interested in the $A^{0}$-invariants, $H^{i}_{\Lie}(\Delta):=H^{0}(A^{0},H^{i}_{\Lie}(A^+,\Delta))$. By a simple calculation we see that:
$$H^{0}_{\an}(A^+,\Delta) = H^{0}_{\Lie}(\Delta) = \Delta^{\varphi=1,\Gamma=1}.$$

Now let $\widetilde{\Delta}$ be an extension of $\mathscr{R}_A$ by $\Delta$ and let $e \in \widetilde{\Delta}$be a lifting of $1 \in \mathscr{R}_A$. Then $(\nabla_{\widetilde{\Delta}}e,(\varphi-1)e)$ is a 1-cocycle in the above complex which is $\Gamma$-invariant, whose class does not depend on $e$. Thus we obtain a map:

$$\Theta^{\Delta}_{\Lie}: H^{1}_{\an}(A^+,\Delta) \rightarrow H^{1}_{\Lie}(\Delta).$$

\begin{lemm} \label{biglem}
For any $(\varphi,\Gamma)$-module over $\mathscr{R}_A$, $\Theta^{\Delta}_{\Lie}$ is an isomorphism.  
\end{lemm}

\begin{proof}
Copy the proof of \cite[Lemme 5.6]{colmez2015}.
\end{proof}

\subsection{Continuous vs. analytic cohomology}

In this section we show the following comparison between locally analytic cohomology defined by Lazard, cf. \cite[Chapitre V, \S 2.3]{Lazardgrp} and continuous group cohomology.

\begin{prop}\label{prop:lazcompcom}
Let $\delta_1, \delta_2 \colon \qpe \to A^\times$ be continuous characters. If $$M \in \left\{ \mathscr{R}_A^{+}(\delta_1, \delta_2), \mathscr{R}_A^{-}(\delta_1, \delta_2), \mathscr{R}_A(\delta_1, \delta_2) \right\},$$ then the natural application
$$H^{i}_{\an}(\overline{P}^{+},M) \rightarrow H^{i}(\overline{P}^{+},M)$$
is an isomorphism for all $i \geq 0$ (here $H^{i}(\overline{P}^{+},M)$ denotes continuous cohomology).  
\end{prop}

\begin{proof}
From the exact sequence of $\overline{P}^+$-modules
\[
0 \to \Robba_A^{+}(\delta_1,\delta_2) \to \Robba_A(\delta_1,\delta_2) \to \Robba_A^{-}(\delta_1,\delta_2) \to 0,
\]
it suffices to prove the result for $ \Robba_A^{+}(\delta_1,\delta_2)$ and $\Robba_A^{-}(\delta_1,\delta_2)$. So suppose $M \in {\mathscr{R}_A^{+}(\delta_1, \delta_2), \mathscr{R}_A^{-}(\delta_1, \delta_2)}$. 
Indeed for $? \in \{ \an, \emptyset\}$ we have spectral sequences
\begin{equation}\label{eq:spseqone}
H^{i}_{?}(A^+, H_{?}^{j}(\overline{U}^{1}, M)) \implies H_{?}^{i+j}(\overline{P}^{+}, M)
\end{equation}
and for $j \geq 0$ fixed
\begin{equation}\label{eq:spseqtwo}
H^{i}_{?}(A^0, H_{?}^{k}(\Phi^+, H_{?}^{j}(\overline{U}^{1}, M))) \implies H_{?}^{i+k}(A^{+}, H_{?}^{j}(\overline{U}^{1}, M)),
\end{equation}
where $\overline{U}^{1} = \begin{pmatrix} 1 & 0 \\ p\zp & 1 \end{pmatrix}$.
We claim that 
\begin{equation}\label{eq:isgef}
H_{\an}^j(\overline{U}^{1},M) \to H^j(\overline{U}^{1},M).
\end{equation}
is an isomorphism. 
Indeed this is true for $j=0$ and for $j = 1$ it is enough to prove that a continuous 1-cocycle $c \colon \overline{U}^{1} \to M$ satisfying
\[
c(\tau^a) = \frac{\tau^a-1}{\tau-1}c(\tau)
\] 
for all $a \in \zp$ is locally analytic. Indeed this follows from the fact that $M \boxtimes_{\omega} \P^1$ is a $\qp$-analytic sheaf, cf. Lemma \ref{lem:princser} (note that $c(\tau^a) = \sum_{n \geq 1} \binom{a}{n} (\tau-1)^{n-1}c(\tau)$ and $(\tau-1)^{n}c(\tau)$ are the Mahler coefficients of the locally analytic function $\overline{U}^{1} \to M$ given by $g \mapsto g \cdot c(\tau)$, cf. \cite[\S IV.2]{coldosunicomp}). Since $\Phi^+$ is discrete we have an isomorphism
\[
H_{\an}^{k}(\Phi^+, H_{\an}^{j}(\overline{U}^{1}, M)) \to H^{k}(\Phi^+, H^j(\overline{U}^{1}, M)).
\]
The same argument (for proving \eqref{eq:isgef} is an isomorphism) gives an isomorphism
\[
H^{i}_{\an}(A^0, H_{\an}^{k}(\Phi^+, H_{\an}^{j}(\overline{U}^{1}, M))) \to H^{i}(A^0, H^{k}(\Phi^+, H^{j}(\overline{U}^{1}, M))). 
\]
Spectral sequences \eqref{eq:spseqone} and \eqref{eq:spseqtwo} now give the result.

\end{proof}

\begin{rema}
In the setting of Proposition \ref{prop:lazcompcom}, one cannot apply Lazard's classical result \cite[Th\'eor\`eme 2.3.10]{Lazardgrp} because $M$ is not of finite type over $A$. Note also that a similar proof yields isomorphisms $H^{i}_{\an}(A^{+},M) \xrightarrow{\sim} H^{i}(A^{+},M)$ for all $i \geq 0$. 
\end{rema}

\subsection{The cohomology of $\Phi^+$}

Let $\delta: \qpe \rightarrow A^{\times}$ be a continuous character. We next compute explicitly some $\Phi^{+}$-cohomology. Note that as $\Phi^{+}$ is discrete, analytic cohomology coincides with standard (continuous) cohomology. In particular we will be interested in the groups $H^{i}(\Phi^{+},\mathscr{R}_{A}^{-} \otimes \delta)$ and $H^{i}(\Phi^{+},\mathscr{R}_{A}^{+} \otimes \delta)$ viewed as $A^{0}$-modules. Since $\Phi^+$ is infinite cyclic, these cohomology groups vanish for $i \geq 2$.

\subsubsection{The case of $\Robba^-_A$} 

We begin with some notation. If $N \geq 0$, we set $ \mathrm{Pol}_{\leq N}(\zp, A) \subset \mathrm{LA}(\zpe, A)$ to be the free sub $A$-module of rank $N+1$ consisting of polynomial functions of degree at most $N$ with coefficients in $A$. Observing that $ \mathrm{LA}(\zpe, A) \cap \mathrm{Pol}_{\leq N}(\zp, A) = \emptyset$ we set
$$T_N := \mathrm{LA}(\zpe, A) \oplus \mathrm{Pol}_{\leq N}(\zp, A).$$ Here is a lemma describing the kernel and cokernel of $1 - \alpha \varphi$ on the locally analytic functions:

\begin{lemm} \label{sum}
Let $\alpha \in A^\times$. Then 
\begin{enumerate}
\item $1-\alpha\varphi \colon \mathscr{R}_{A}^{-} \rightarrow \mathscr{R}_{A}^{-}$ is injective.
\item If $N \geq 0$ is large enough, then $T_N + (1-\alpha\varphi)\mathscr{R}_A^{-} = \mathscr{R}_A^{-}$ and $T_N \cap (1-\alpha\varphi)\mathscr{R}_A^{-} = (1 - \alpha \varphi) 
\mathrm{Pol}_{\leq N}(\zp, A)$.
\end{enumerate}

\end{lemm}

\begin{proof}
We first prove injectivity. Note that this is already proved in \cite[Lemme 2.9(ii)]{chen2013}. We repeat the argument here. If $\phi \in \mathscr{R}_{A}^{-}$ is in the kernel of $1-\alpha\varphi$, then $\phi = \alpha^{n}\varphi^{n}(\phi) \hspace{1mm} \forall \hspace{1mm} n\in \mathbf{N}$. Recall the action of $\varphi$ on $\mathscr{R}_{A}^{-}$:
$$(\varphi \cdot \phi)(x) = \begin{cases}
                            \phi \left( \frac{x}{p} \right) & \text{if } x \in p\zp \\
														0 & \text{if } x \not\in p\zp
                            \end{cases}$$
Thus $\phi$ is zero on $p^{n}\zpe$ and hence $\phi = 0$, as desired.

We next prove the second assertion. Let $N \geq 0$ be such that $|\alpha^{-1} p^{N+1}| < 1$, so that in particular $1 - \alpha p^{-j} = - \alpha p^{-j} (1 - \alpha^{-1} p^j) \in A^\times$ for all $j > N$. Let $\phi \in \mathrm{LA}(\zp, A)$ and $n$ be such that $\phi$ is analytic on every ball $i + p^{n} \zp$ ($i \in \zpe$), then we can write $ \mathbf{1}_{p^n \zp} \phi = \sum_{j=0}^{+\infty} a_{j}\mathbf{1}_{p^{n} \zp}x^{j}$ and so the function
$$\phi - (1-\alpha\varphi)\left( \sum_{j=N+1}^{+\infty}\frac{a_j}{1-\alpha p^{-j}}\mathbf{1}_{p^{n-1}\zp}x^{j} \right), $$ which is well defined since the elements $1 - \alpha p^{-j} \in A^\times$, can be expressed as the sum of a polynomial of degree $N$ and a locally analytic function vanishing on $p^{n}\zp$. Hence every $\phi \in \mathscr{R}_A^{-}$ is of the form $\phi_1 + P_\phi + (1-\alpha\varphi)\phi_2$, with $P_\phi  \in \mathrm{Pol}_{\leq N}(\zp ,A)$ and  $\phi_1, \phi_2 \in \mathrm{LA}(\zp,A)$ such that $\phi_1$ vanishes in a neighbourhood of $0$. In particular $\phi_1$ is of the form $\sum_{i = 0}^{n-1} \phi_{1,i}$, with $\phi_{1,i} \in \mathrm{LA}(p^i \zpe,A)$. Writing $\varphi^i\psi^i (\phi_{1,i}) = (1-(1-\alpha\varphi))^{i}\psi^{i}(\alpha^{-i}\phi_{1,i})$ and upon expanding $(1-(1-\alpha\varphi))^i$ expresses $\phi_{1,i}$ as a sum of elements in $(1-\alpha\varphi)\mathscr{R}_A^{-}$ and $\psi^{i}(\alpha^{-i}\phi_{1,i}) \in \mathrm{LA}(\zpe, A)$.

We calculate next the intersection $T_N \cap (1-\alpha\varphi)\mathscr{R}_{A}^{-}$. If $(1-\alpha\varphi)\phi = \phi' + P$ for some $\phi \in \mathrm{LA}(\zp,A)$, $\phi' \in \mathrm{LA}(\zpe, A)$ and $P \in \mathrm{Pol}_{\leq _N}(\zp, A)$, then we have $\psi((1-\alpha\varphi)\phi) = \psi (P)$ (as $\mathrm{LA}(\zpe,A) =  \mathrm{LA}(\zp,A)^{\psi=0}$). Thus $(\psi-\alpha) \cdot \phi = \psi(P)$ and hence $\phi(x) = \alpha^{-1}( \phi(px) - P(px))$ for all $x$ (recalling that $(\psi \cdot \phi)(x) = \phi(px)$). Repeating gives $\phi(x) = \alpha^{-n}\phi(p^{n}x) - \alpha^{-n} P(p^nx) - \alpha^{n-1} P(p^{n-1}x) - \hdots - \alpha^{-1} P(px)$, which shows that $\phi$ is analytic on $\zp$. Writing $\phi(x) = \sum_{i \geq 0} a_{i}x^{i}$, $P(x) = \sum_{i = 0}^N b_n x^n$ on $\zp$ with $a_i, b_i \in A$, the equality $\phi(px) - \alpha \phi(x) = P(px)$ gives \[ \sum_{i=0}^{+\infty} (p^n - \alpha) a_{i} x^i = \sum_{i = 0}^N b_n x^n \] on $\zp$. This gives $(p^n - \alpha) a_{i} = 0$ for $i > N$ and thus $a_i = 0$ (since $(p^n - \alpha) = p^ n (1 - \alpha p^{-n}) \in A^\times$), which implies that $\phi \in \mathrm{Pol}_{\leq N}(\zp, A)$ and hence $T_N \cap (1-\alpha\varphi)\mathscr{R}_{A}^{-} \subseteq (1 - \alpha \varphi) \mathrm{Pol}_{\leq N}(\zp, A)$. To prove the reverse inclusion, we note that, if $P(x) = \sum_{i = 0}^N a_i x^i \in \mathrm{Pol}_{\leq N}(\zp, A)$, then
\begin{eqnarray*} 
(1 - \alpha \varphi) P &=& \mathbf{1}_\zpe \cdot P + \mathbf{1}_{p \zp} \cdot \sum_{i = 0}^N (1 - \alpha p^{-i}) a_i x^i \\
&=& \alpha \mathbf{1}_\zpe \cdot \sum_{i = 0}^N p^{-i} a_i x^i + \sum_{i = 0}^N (1 - \alpha p^{-i}) a_i x^i \in T_N
\end{eqnarray*}

\end{proof}

The following three statements are now immediate from the above lemma, the identity $H_{\rm an}^i(\Phi^+, M) = H^i(\Phi^+, M)$ and the description of the continuous cohomology of a cyclic group.

\begin{coro} \label{H0-}
$H^{0}(\Phi^{+},\mathscr{R}_{A}^{-} \otimes \delta) = 0$.
\end{coro}

\begin{coro} \label{H1-} For $N$ large enough, we have a short exact sequence (of $A^0$-modules)
\[ 0 \to (1 - \alpha \varphi) \mathrm{Pol}_{\leq N}(\zp, A) \otimes \delta \to T_N \otimes \delta \to H^1(\Phi^{+},\mathscr{R}_{A}^{-} \otimes \delta) \to 0. \] 
\end{coro}

%

\subsubsection{The case of $\Robba_A^+$.} \label{caseofR+}

As in the preceding section, the calculation of the cohomology of the group $\Phi^+$ acting on $\Robba_A^+$ will reduce to the following lemma

\begin{lemm} \label{sum2}
Let $\alpha \in A^\times$, consider $1-\alpha\varphi \colon \mathscr{R}_{A}^{+} \rightarrow \mathscr{R}_{A}^{+}$ and let $N$ be large enough. Then $$ \ker (1 - \alpha \varphi) = \ker (1 - \alpha \varphi \colon \mathrm{Pol}_{\leq N}(\zp, A)^*), $$ $$ \coker (1 - \alpha \varphi) = \coker (1 - \alpha \varphi \colon \mathrm{Pol}_{\leq N}(\zp, A)^*).$$ In particular:
\begin{enumerate}
\item $\ker (1 - \alpha \varphi)  = \bigoplus_{i = 0}^N \mathrm{Ann}(1 - \alpha p^i) t^i$ and $1 - \alpha \varphi$ is injective if $\alpha$ is such that $(1 - \alpha p^i)$ is not a zero divisor for any $i$.
\item $1 - \alpha \varphi$ is surjective if $1 - \alpha p^i \in A^\times$ for all $i \in \N$.
\end{enumerate} 
\end{lemm}

\begin{proof}
This is essentially \cite[Lemme 2.9 (ii)]{chen2013}, of which the idea of proof comes from \cite[Lemme A.1]{col2008}. We provide a sketch here. Choose $N \geq 0$ an integer large enough so that $| \alpha p^{N+1} | < 1$. Then $1 -\alpha \varphi$ is invertible on $T^{N+1} \mathscr{R}_{A}^{+}$. To conclude, it suffices to remark that there is a $\varphi$-stable decomposition:
$$\mathscr{R}_{A}^{+} = (\bigoplus_{i = 0}^N A t^i )\oplus T^{N+1} \mathscr{R}_{A}^{+} $$
and that $(1-\alpha \varphi)(t^j) = (1-\alpha p^j)t^j$. We get then the desired result for the kernel and cokernel of $1 - \alpha \varphi$ observing that Amice transform identifies $\bigoplus_{0 \leq i \leq N} A t^i$ with the dual of $\mathrm{Pol}_{\leq N}(\zp, A)$.      
\end{proof}

\begin{rema}
Observe that if $\alpha = p^{-i}$ for some $i \in \N$, then the kernel of $1-\alpha\varphi: \mathscr{R}_{A}^{+} \rightarrow \mathscr{R}_{A}^{+}$ contains $A \cdot t^i$, which is identified with the free $A$-module of rank one generated by the distribution sending a function $f$ to $f^{(i)}(0)$ via the Amice transformation.
\end{rema}

\begin{coro} \label{H+} Let $j \in \{0, 1\}$ and $N$ be large enough. We have $$ H^j(\Phi^{+},\mathscr{R}_{A}^{+} \otimes \delta) \cong H^j(\Phi^+, \mathrm{Pol}_{\leq N}(\zp, A)^* \otimes \delta). $$ 

In particular
\begin{itemize}
\item $H^0(\Phi^+, \Robba_A^+ \otimes \delta) \neq 0$ if and only if $1 - \delta(p) p^i$ divides zero for some $i \in \N$.
\item $H^1(\Phi^+, \Robba_A^+ \otimes \delta) \neq 0$ if and only if there exists an $i \in \N$ such that $1 - \delta(p) p^i$ vanishes at some point of $\mathrm{Sp}(A)$.
\end{itemize}
\end{coro}

\subsection{The $A^0$-cohomology}

We next compute some $A^{0}$-cohomology, for which we use the description of analytic cohomology in terms of the action of the Lie algebra. Fix a continuous character $\eta: \zpe \rightarrow A^\times$, viewing it naturally as a character of $A^{0}$.

\subsubsection{The case of $\mathrm{LA}(\zpe, A) \otimes \eta$.}

\begin{lemma} \label{newprop1}
$\nabla + \kappa(\eta)$ is surjective on $\mathrm{LA}(\zpe,A)$ and its kernel is generated by the set of $\phi \eta$, with $\phi$ locally constant.
\end{lemma}

\begin{proof}
We compute for $\phi \in \mathrm{LA}(\zpe,A)$,
$$\nabla \phi(x) = \lim_{a \to 1} \frac{\phi(x/a)-\phi(x)}{a-1} = -x\phi'(x)$$
and thus
$$(\nabla + \kappa(\eta))(\phi \eta) = -x(\phi' \eta + \kappa(\eta)x^{-1}\phi \eta) + \kappa(\eta)\phi \eta = \nabla(\phi) \cdot \eta,$$
where the first equality follows from the relation $\eta'(x) = \eta'(1)\eta(x)x^{-1}$. We see that to show (1), multiplying by $\eta$ if necessary (recall that $\eta$ takes values in $A^\times$ so multiplication by $\eta$ is invertible), we can assume $\kappa(\eta) = 0$. The description of the kernel is clear and surjectivity follows from surjectivity of $\phi \mapsto \phi'$ (as we can easily see integrating a power series).  
\end{proof}

\begin{prop} \label{2.25} \leavevmode
\begin{enumerate} 
\item $H^{0}_{\an}(A^{0},\mathrm{LA}(\zpe,A) \otimes \eta)$ is a free $A$-module of rank 1 generated by $\mathbf{1}_{\zpe}\eta$. \item $H^{1}_{\an}(A^{0},\mathrm{LA}(\zpe,A) \otimes \eta) = 0$.
\end{enumerate}
\end{prop}

\begin{proof}
Note that for $i \in \{0,1 \}$ and $M \in \Phi\Gamma(\mathscr{R}_A)$, $H^{i}_{\an}(A^{0},M)$ is computed by the $A^{0}$-invariants of the cohomology of the complex
$$0 \rightarrow M \xrightarrow{\nabla} M \rightarrow 0.$$
To see this, it suffices to repeat Lemma \ref{biglem}, ignoring the action of $\Phi^{+}$. As $\nabla(\phi \otimes \eta) = ((\nabla + \kappa(\eta))\phi) \otimes \eta$, the two assertions follow from lemma \ref{newprop1}: the first cohomology group is trivial since $(\nabla + \kappa(\eta))$ is surjective, and if $\phi \eta \otimes \eta$ is killed by $\nabla$ and fixed by $A^0$, then $\phi(a x) \eta(x) = \phi(x) \eta(x)$ for every $a \in A^0$ and so $\phi$ is constant on $\zpe$, hence the result.
\end{proof}

\subsection{The $A^+$-cohomology}

In this section we denote $\delta \colon \qpe \to A^\times$ a continuous character, $\alpha = \delta(p)$ and $\beta = \delta(a)$. Let $N$ be large enough and set $\mathrm{Pol}_{\leq N} = \mathrm{Pol}_{\leq N}(\zp, A)$. We next calculate the $A^+$-cohomology of $\mathscr{R}_{A}^{-} \otimes \delta$, $\mathscr{R}_{A}^{+} \otimes \delta$ and hence that of $\Robba_A \otimes \delta$.

\subsubsection{The case of $\mathscr{R}_A^- \otimes \delta$.}

\begin{lemma} \label{H0A+R-}
We have $H^0(A^+, \Robba_A^- \otimes \delta) = 0$.
\end{lemma}

\begin{proof}
This follows from Corollary \ref{H0-}.
\end{proof}

\begin{lemma}
The groups $H^i(A^+, \Robba_A^- \otimes \delta)$, $i = 1, 2$, live in an exact sequence of $A$-modules
\begin{eqnarray*} 
0 &\to& ((1 - \alpha \varphi) \mathrm{Pol}_{\leq N} \otimes \delta)^\Gamma \xrightarrow{f} (T_N \otimes \delta)^\Gamma \to H^1(A^+, \Robba_A^- \otimes \delta) \\
&\to& H^1(A^0, (1 - \alpha \varphi) \mathrm{Pol}_{\leq N} \otimes \delta) \xrightarrow{g} H^1(A^0, T_N \otimes \delta) \to H^2(A^+, \Robba_A^- \otimes \delta) \to 0
\end{eqnarray*}
\end{lemma}

\begin{proof}
Inflation-restriction and Corollary \ref{H0-} give \[ H^1(A^+, \Robba^-_A \otimes \delta) = H^0(A^0, H^1(\Phi^+, \Robba^-_A \otimes \delta)). \] The result follows then by taking the long exact sequence of $A^0$-cohomology associated to the short exact sequence of $A^0$-modules of Corollary \ref{H1-}.
\end{proof}

In order to calculate $H^i(A^+, \Robba^- \otimes \delta)$, we need to examine the cokernel of $f$ and the kernel of $g$. We first begin with a lemma stating some preliminary reductions. 

\begin{lemma} \label{2.29} \leavevmode
\begin{enumerate}
\item $(T_N \otimes \delta)^\Gamma = A \cdot \langle \mathbf{1}_\zpe \delta \rangle \oplus \bigoplus_{i = 0}^N \mathrm{Ann}(1 - \beta a^{-i}) \cdot x^i$.
\item $((1 - \alpha \varphi) \mathrm{Pol}_{\leq N} \otimes \delta)^\Gamma = (1 - \alpha \varphi) \bigoplus_{i = 0}^N \mathrm{Ann}(1 - \beta a^{-i}) \cdot x^i$
\item $H^1(A^0, T_N \otimes \delta) = H^1(A^0,  \mathrm{Pol}_{\leq N} \otimes \delta). $
\end{enumerate}
\end{lemma}

\begin{proof}
Since $T_N = \mathrm{LA}(\zpe, A) \oplus \mathrm{Pol}_{\leq N}$ (as $A^0$-modules), the first point follows from Proposition \ref{2.25}(1) and the fact that $(1-  \beta \gamma) (a_i x^i) = a_i (1 - \beta a^{-i}) x^i$. 

The second point follows from the same calculation, the fact that $\gamma$ commutes with $\varphi$ and Lemma \ref{sum}(1). 

The third point is a consequence of Proposition \ref{2.25}(2).
\end{proof}

\begin{lemma} \label{H2A+R-}
$H^2(A^+, \Robba_A^- \otimes \delta) = \oplus_{i = 0}^N A / (1 - \alpha p^{-i}, 1 - \beta a^{-i})$.
\end{lemma}

\begin{proof}
As calculated in the proof of Lemma \ref{sum}, if $Q(x) = a_0 + a_1 x + \hdots + a_N x^N \in \mathrm{Pol}_{\leq N}$, then $(1 - \alpha \varphi) Q \in T_N = \mathrm{LA}(\zpe, A) \oplus \mathrm{Pol}_{\leq N}$ is given by 
\begin{equation}\label{ff}
(1 - \alpha \varphi)Q = \alpha \mathbf{1}_\zpe \cdot \sum_{i = 0}^N a_i p^{-i} x^i \oplus \sum_{i = 0}^N a_i (1 - \alpha p^{-i}) x^i. 
\end{equation}
Using Lemma \ref{2.29}(3), we need to calculate the cokernel of the following composition:
\[ H^1(A^0, (1 - \alpha \varphi) \mathrm{Pol}_{\leq N} \otimes \delta) \to H^1(A^0, T_N \otimes \delta) \xrightarrow{\sim} H^1(A^0, \mathrm{Pol}_{\leq N} \otimes \delta)). \] By equation \eqref{ff} above, this map sends the class of $(1 - \alpha \varphi)Q$ in  $H^1(A^0, (1 - \alpha \varphi) \mathrm{Pol}_{\leq N} \otimes \delta)$ to the class of $\sum_{i = 0}^N a_i (1 - \alpha p^{-i}) x^i$ in $H^1(A^0, (1 - \mathrm{Pol}_{\leq N} \otimes \delta))$. Since $\mathrm{Pol}_{\leq N} = \oplus_{i = 0}^N A \cdot x^i$ and since the action of $\varphi$ and $\gamma$ commute, we easily see that this cokernel is given by $\oplus_{i = 0}^N A / (1 - \alpha p^{-i}, 1 - \beta a^{-i})$ as claimed.
\end{proof}

\begin{coro} \label{H2A+R-cor}
$H^2(A^+, \Robba_A^- \otimes \delta) = 0$ if and only if $\delta$ is pointwise never of the form $x^i$ for any $i \geq 0$.
\end{coro}

\begin{proof}
This follows immediately from the last lemma, observing that, if $\delta$ is pointwise never of the form $x^i$ if and only $(1 - \alpha p^{-i}, 1 - \beta a^{-i}) = A$ for all $i \in \N$.
\end{proof}

An explicit description of the cokernel of $f$ and the kernel of $g$ seems a difficult task, but we can describe them in a particular case that will be of interest to us.

\begin{prop} \label{prop:LAfree}
Let $\delta \colon \qpe \to A^\times$ be such that $\delta$ is pointwise never of the form $x^i$ for any $i \geq 0$. Then $H^1(A^+, \Robba_A^- \otimes \delta)$ is a free $A$-module of rank $1$.
\end{prop}

\begin{proof}
First note that as $\delta$ is pointwise never of the form $x^i$ for any $i \geq 0$, $(1 - \alpha p^{-i}, 1 - \beta a^{-i}) = A$.
The key observation, which appears in the proof of \cite[Th\'eor\`eme 2.29]{chen2013}, is that, if $a, b \in A$ are such that $(a, b) = A$, then multiplication by $a$ is bijective on $A / bA$ and on $\mathrm{Ann}(b)$. Indeed, if $u, v \in A$ are such that $a u + b v = 1$, then multiplication by $u$ provides an inverse for this map.

We first show that $\mathrm{ker}(g) = 0$. As in the proof of Lemma \ref{H2A+R-}, we consider the following composition \[ H^1(A^0, (1 - \alpha \varphi) \mathrm{Pol}_{\leq N} \otimes \delta) \to H^1(A^0, T_N \otimes \delta) \xrightarrow{\sim} H^1(A^0, \mathrm{Pol}_{\leq N} \otimes \delta)) \] and we show in this case that it is injective. It is easy to see, from the formulas describing the action of $\varphi$ and $\Gamma$, that the problem reduces to showing that, for every $i = 0, \hdots, N$, if $b_i \in A$ is such that $a_i (1 - \alpha p^{-i}) = (1 - \beta a^{-i}) b_i$, then $b_i \in (1 - \alpha p^{-i}) A$, which is a consequence of the fact that multiplication by $1 - \beta a^{-i}$ is bijective on $A / (1 - \alpha p^{-i}) A$.

In a similar way, one can show that $\mathrm{coker}(f)$ is a free $A$-module of rank one using Lemma \ref{2.29}(1) and (2), and the fact that multiplication by $(1 - \alpha p^{-i})$ is bijective on $\mathrm{Ann}(1 - \beta a^{-i})$. This shows that $H^1(A^+, \Robba_A \otimes \delta)$ is a free $A$-module of rank one and completes the proof.


 
\end{proof}

\subsubsection{The case of $\mathscr{R}_A^+ \otimes \delta$.}

\begin{lemma} \label{2.31} Let $N$ be large enough. Then
\begin{enumerate}
\item $H^0(A^0, H^0(\Phi^+,\mathscr{R}_A^+ \otimes \delta)) = \bigoplus_{i = 0}^N \mathrm{Ann}(1 - \alpha p^i, 1 - \beta a^i) \cdot t^i$.
\item $H^1(A^0, H^0(\Phi^+,\mathscr{R}_A^+ \otimes \delta)) = \bigoplus_{i = 0}^N \mathrm{Ann}(1 - \alpha p^i) / (1 - \beta a^i) \cdot t^i$.
\item $H^0(A^0, H^1(\Phi^+, \mathscr{R}_A^+ \otimes \delta)) = \bigoplus_{i = 0}^N \mathrm{Ann}(1 - \beta a^i \colon A / (1 - \alpha p^i)) \cdot t^i$.
\item $H^1(A^0, H^1(\Phi^+, \mathscr{R}_A^+ \otimes \delta)) = \bigoplus_{i = 0}^N A / (1 - \alpha p^i, 1 - \beta a^i) \cdot t^i$.
\end{enumerate}
\end{lemma}

\begin{proof}
This follows easily from Lemma \ref{sum2} and Corollary \ref{H+}, observing that $(1 - \beta \sigma_a)(t^i) = (1 - \beta a^i) t^i$.
\end{proof}

\begin{prop} \label{H1R+}
Let $\delta \colon \qpe \to A^\times$ be such that $\delta$ is pointwise never of the form $x^{-i}$ for any $i \geq 0$. Then $H^i(A^+, \Robba_A^+ \otimes \delta) = 0$ for $i \in \{ 0,1,2 \}$.
\end{prop}

\begin{proof}
First note that as $\delta$ is pointwise never of the form $x^{-i}$ for any $i \geq 0$, $(1 - \alpha p^i, 1 - \beta a^i) = A$. This implies, as in the proof of Proposition \ref{prop:LAfree}, that multiplication by $1 - \beta a^i$ is bijective on $\mathrm{Ann}(1 - \alpha p^i)$ and on $A / (1 - \alpha p^i)$.

The vanishing of $H^0(A^+, \Robba_A^+ \otimes \delta)$ follows from Lemma \ref{2.31}(1) and the fact that multiplication by $1 - \beta a^i$ on $\mathrm{Ann}(1-\alpha p^i)$ is injective. 

The vanishing of $H^2(A^+, \Robba_A^+ \otimes \delta)$ follows from Lemma \ref{2.31}(4) and the fact that multiplication by $1 - \beta a^i$ on $A/(1-\alpha p^i)$ is surjective. 

In what concerns the vanishing of $H^1(A^+, \Robba_A^+ \otimes \delta)$, the inflation-restriction sequence gives
\[ 0 \to H^1(A^0, H^0(\Phi^+,\mathscr{R}_A^+ \otimes \delta)) \to H^1(A^+, \mathscr{R}_A^+ \otimes \delta) \to H^0(A^0, H^1(\Phi^+, \mathscr{R}_A^+ \otimes \delta)) \to 0. \] The result is now an easy consequence of Lemma \ref{2.31}(2) and (3) and the fact that multiplication by $1 - \beta a^i$ is surjective on $\mathrm{Ann}(1 - \alpha p^i)$ and injective on $A / (1 - \alpha p^i)$, respectively.
\end{proof}

\subsubsection{The case of $\Robba_A \otimes \delta$}

The calculation of $H^i(A^+, \Robba_A \otimes \delta)$ is now formal using the (long exact sequence of $A^+$-cohomology associated to the) short exact sequence of $A^+$-modules
\[ 0 \to \Robba_A^+ \otimes \delta \to \Robba_A \to \Robba_A^- \otimes \delta \chi^{-1} \to 0. \] Moreover in the so-called regular case, we are able to compute it explicitly.  \footnote{This result is already proved by similar methods in \cite[Th\'eor\`eme 2.29]{chen2013}.}

\begin{prop} \label{HiRreg}
Let $\delta: \qpe \to A^\times$ be such that $\delta$ is pointwise never of the form $\chi x^i$ nor of the form $x^{-i}$ for any $i \geq 0$. Then $H^0(A^+, \Robba_A \otimes \delta) = H^2(A^+, \Robba_A \otimes \delta)= 0$ and $H^1(A^+, \Robba_A \otimes \delta)$ is a free $A$-module of rank $1$.
\end{prop}

\begin{proof}
The long exact sequence of $A^+$-cohomology associated to the short exact sequence of $A^+$-modules
\[ 0 \to \Robba_A^+ \otimes \delta \to \Robba_A \to \Robba_A^- \otimes \delta \chi^{-1} \to 0 \] and the fact that $H^0(A^+, \Robba_A^- \otimes \delta \chi^{-1}) = H^2(A^+, \Robba_A^+ \otimes \delta) = 0$, cf. Lemma \ref{H0A+R-} and Proposition \ref{H1R+} yields
\[ 0 \to H^1(A^+, \Robba^+_A \otimes \delta) \to H^1(A^+, \Robba_A \otimes \delta) \to H^1(A^+, \Robba^-_A \otimes \delta \chi^{-1}) \to 0, \] and so
\[ H^2(A^+, \Robba_A \otimes \delta) \cong H^2(\Robba^-_A \otimes \delta \chi^{-1}). \] 
The result follows then from Proposition \ref{prop:LAfree} and Proposition \ref{H1R+}.
\end{proof}

\begin{rema} \leavevmode \label{rem:recchenres}
\begin{itemize}
\item The calculations of this section show that, for $M \in \{ \Robba^+_A, \Robba^-_A, \Robba_A \}$, the $A$-modules $H^i(A^+, M \otimes \delta)$ are finite (as also proved in \cite{chen2013}).
\item $H^2(A^+, \Robba_A \otimes \delta) = 0$ if and only if $\delta$ is pointwise never of the form $\chi x^i$, $i \in \N$. Indeed, this is a necessary condition by Corollary \ref{H2A+R-cor}. For the converse first note that if $\delta$ is never of the form $\chi x^i$ nor $x^{-i}$ for any $i \in \N$, then $H^2(A^+, \Robba_A \otimes \delta) = 0$ by Proposition \ref{HiRreg}. On the other hand, if $\delta$ reduces to $x^{-i}$ for some $i \geq 0$ at some point of $\mathrm{Sp}(A)$, we use the following argument to reduce to the case of a point ($A$ a finite extension of $\qp$). The finiteness of the $A$-module $H^2(A^+, \Robba_A \otimes \delta)$, the vanishing of $H^3(A^+, \Robba_A \otimes \delta)$, the fact that $\Robba_A$ is a flat $A$-module (cf. Lemma \ref{flatnessR}) and the $\mathrm{Tor}$-spectral sequence 
\[ \mathrm{Tor}_{-p}(H^q(A^+, \Robba_A \otimes \delta), A / \mathfrak{m}) \Rightarrow H^{p + q}(A^+, \Robba_{A / \mathfrak{m}} \otimes \delta) \]
show that $H^2(A^+, \Robba_A \otimes \delta) \otimes A / \mathfrak{m} = H^2(A^+, \Robba_{A / \mathfrak{m}} \otimes \delta)$ for every maximal ideal $\mathfrak{m} \subseteq A$. Since $H^2(A^+, \Robba_{A / \mathfrak{m}} \otimes \delta) = 0$ (cf. \cite[Th\'eor\`eme 5.16]{colmez2015}), we can conclude that $H^2(A^+, \Robba_A \otimes \delta) = 0$ by Nakayama's lemma.
\end{itemize}
\end{rema}

\section{Relative cohomology} \label{sec:maincofinza}

Over the next few sections we prove an isomorphism between the $\overline{P}^+$-cohomology and the $A^+$-cohomology with coefficients in $\Robba_A(\delta_1, \delta_2)$ assuming $\delta_1\delta_2^{-1}$ is regular. This is a generalization of a result of Colmez, who proves it for the case where $A$ is a finite extension of $\qp$ and we indeed reduce the general result to that case using some arguments on derived categories inspired by \cite{kedlaya2014}.

\subsection{Formalism of derived categories}\label{sec:formdercat}

In this section we fix a noetherian ring $A$. Let $\mathcal{D}^{-}(A)$ denote the derived category of $A$-modules bounded above. We begin by recalling the notion of a pseudo-coherent complex. For a detailed explanation we refer the reader to \cite[\href{http://stacks.math.columbia.edu/tag/064N}{Tag 064N}]{stacks-project} 

\begin{defi} \label{pcoherent} \leavevmode
\begin{enumerate}
\item An object $K^{\bullet}$ of $\mathcal{D}^{-}(A)$ is pseudo-coherent if it is quasi-isomorphic to a bounded above complex of finite free $A$-modules. We denote by $\mathcal{D}_{\mathrm{pc}}^{-}(A) \subseteq \mathcal{D}^{-}(A)$ the full subcategory of pseudo-coherent objects of $\mathcal{D}^{-}(A)$.
\item An $A$-module $M$ is called pseudo-coherent if $M[0] \in \mathcal{D}_{\mathrm{pc}}^{-}(A)$.
\end{enumerate} 
\end{defi}

We have the following simple Lemma detecting when a module is pseudo-coherent.

\begin{lemm} \label{1.2}
An $A$-module $M$ is pseudo-coherent iff there exists an infinite resolution 
$$\cdots \rightarrow A^{\oplus n_1} \rightarrow A^{\oplus n_0} \rightarrow M \rightarrow 0$$
\end{lemm}

\begin{proof}
This is just rephrasing part (2) of Definition \ref{pcoherent}.  
\end{proof}

Since $A$ is noetherian, Lemma \ref{1.2} can be further strenghened to the following.

\begin{lemm} \label{1.3}
An $A$-module $M$ is pseudo-coherent iff it is finite.  
\end{lemm}

\begin{proof}
We first show that a finite $A$-module $M$ is pseudo-coherent. Indeed since $M$ is finite, one may choose a surjection $A^{\oplus n_{0}} \rightarrow M$. Then having constructed an exact complex of finite free $A$-modules of length $t$, we can extend by choosing a surjection
$$A^{\oplus n_{t+1}} \rightarrow \ker(A^{\oplus n_{t}} \rightarrow A^{\oplus n_{t-1}}).$$
Here we have implicitly used that a submodule of a finite $A$-module is finite. Conversely, a pseudo-coherent module is finite by Lemma \ref{1.2}.   
\end{proof}

The following Lemma allows us to use induction-type arguments when trying to prove results concerning pseudo-coherent complexes.

\begin{lemm} \label{1.4}
Let $K^{\bullet} \in \mathcal{D}^{-}(A)$. The following are equivalent
\begin{enumerate}
\item $K^{\bullet} \in \mathcal{D}_{\mathrm{pc}}^{-}(A)$.
\item For every integer $m$, there exists a bounded complex $E^{\bullet}$ (depending on $m$) of finite free $A$-modules and a morphism $\alpha: E^{\bullet} \rightarrow K^{\bullet}$ such that $H^{i}(\alpha)$ is an isomorphism for $i > m$ and $H^{m}(\alpha)$ is surjective.
\end{enumerate}
\end{lemm}

\begin{proof}
Suppose (1) holds. Let $E^{\bullet}$ be a bounded above complex of finite free $A$-modules and let $E^{\bullet} \rightarrow K^{\bullet}$ be a quasi-isomorphism. Consider the naive truncation at place $m$
$$F_{m}^{\bullet}: \cdots \rightarrow 0 \rightarrow E^{m} \rightarrow E^{m+1} \rightarrow \cdots.$$
Then the induced maps $F_{m}^{\bullet} \rightarrow K^{\bullet}$ satisfy condition (2). 

Suppose (2) holds. We are going to construct our bounded above complex $E^{\bullet}$ of finite free $A$-modules (which will be quasi-isomorphic to $K^{\bullet}$) by descending induction. Since $K^{\bullet}$ is bounded above, there is an integer $a$, such that $K^{n} = 0$, $\forall n \geq a$. By descending induction on $n \in \Z$, we are going to construct a complex
$$F_{n}^{\bullet}: \cdots \rightarrow 0 \rightarrow F^{n} \rightarrow F^{n+1} \rightarrow \cdots \rightarrow F^{a-1} \rightarrow 0 \rightarrow \cdots$$
and a morphism $\alpha_{n}: F_{n}^{\bullet} \rightarrow K^{\bullet}$, such that $H^{i}(\alpha_{n})$ is an isomorphism for $i > n$ and a surjection for $i=n$. For the base case $n=a$, we can take $F^{i} = 0$ $\forall i$. Now consider the induction step. Let $C^{\bullet} = \mathrm{cone}(F_{n}^{\bullet} \xrightarrow{\alpha_n} K^{\bullet})$. The long exact sequence of cohomology coming from the triangle
$$F_{n}^{\bullet} \rightarrow K^{\bullet} \rightarrow C^{\bullet} \rightarrow F_{n}^{\bullet}[1]$$
gives $H^{i}(C^{\bullet}) = 0$ for $i \geq n$. It is easy to see that condition (2) is stable by extensions and so in particular $C^{\bullet}$ satisfies condition (2). We claim that $H^{n-1}(C^{\bullet})$ is a finite $A$-module. Indeed choose a bounded complex $D^{\bullet}$ of finite free $A$-modules and a morphism $\beta: D^{\bullet} \rightarrow C^{\bullet}$ inducing isomorphism on cohomology in degrees $\geq n$ and a surjection in degree $n-1$. It suffices to show $H^{n-1}(D^{\bullet})$ is a finite $A$-module. Let $t$ be the largest integer such that $E^{t} \neq 0$. If $t = n-1$, then the result is clear. If $t > n-1$, then $D^{t-1} \rightarrow D^{t}$ is surjective as $H^{t}(D^{\bullet}) = 0$. As $D^{t}$ is free, we see that $D^{t-1} = D' \oplus D^{t}$. It suffices to prove the result for the complex $(D')^{\bullet}$, which is the same as $D^{\bullet}$ except has $D'$ in degree $t-1$ and $0$ in degree $t$. The result follows by induction. Hence $H^{n-1}(C^{\bullet})$ is a finite $A$-module as claimed. 

Choose a finite free $A$-module $F^{n-1}$ and a map $p: F^{n-1} \rightarrow C^{n-1}$ such that the composition $F^{n-1} \rightarrow C^{n-1} \rightarrow C^{n}$ is zero and such that $F^{n-1}$ surjects onto $H^{n-1}(C^{\bullet})$. Since $C^{n-1} = K^{n-1} \oplus F^{n}$ (by definition of the cone), we can write $p = (\alpha^{n-1}, -d^{n-1})$. The vanishing of the composition $F^{n-1} \rightarrow C^{n-1} \rightarrow C^{n}$, implies these maps fit into a morphism of complexes:
$$
\begin{tikzcd}[row sep = large, column sep = large]
\cdots \arrow[r] & 0 \arrow[r] \arrow[d] & F^{n-1} \arrow[d, "\alpha^{n-1}"] \arrow[r, "d^{n-1}"] & F^{n} \arrow[d, "\alpha_{n}^{n}"] \arrow[r] & F^{n+1} \arrow[d, "\alpha_{n}^{n+1}"] \arrow[r] & \cdots \\
\cdots \arrow[r] & K^{n-2} \arrow[r] & K^{n-1} \arrow[r] & K^{n} \arrow[r] & K^{n+1} \arrow[r] & \cdots
\end{tikzcd}
$$

Moreover we obtain a morphism of triangles
$$
\begin{tikzcd}[row sep = large, column sep = large]
(F^n \rightarrow \cdots) \arrow[r] \arrow[d] & (F^{n-1} \rightarrow \cdots) \arrow[r] \arrow[d] & F^{n-1} \arrow[d, "p"] \\
(F^n \rightarrow \cdots) \arrow[r] & K^{\bullet} \arrow[r] & C^{\bullet}
\end{tikzcd}
$$
By the octaeder axiom for triangulated categories, our choice of $p$ implies that the map of complexes 
$$(F^{n-1} \rightarrow \cdots) \rightarrow K^{\bullet}$$
induces an isomorphism in degrees $\geq n$ and a surjection in degree $n-1$.
\end{proof}

\begin{rema} \label{3.5}
The above proof also shows the following useful fact. If a complex in $\mathcal{D}^{-}_{\mathrm{pc}}(A)$ has trivial cohomology in degrees strictly greater than $b$, then it is quasi-isomorphic to a complex $P^{\bullet} \in \mathcal{D}^{-}_{\mathrm{pc}}(A)$ with $P^{i} = 0$ $\forall i \geq b+1$ and each $P^{j}$ is finite free $\forall j \in \Z$. The proof also shows that $\mathcal{D}^{-}_{\mathrm{pc}}(A)$ is stable by extensions.  
\end{rema}

Since $A$ is noetherian, we have the following simple criterion for detecting whether an object in $\mathcal{D}^{-}(A)$ is pseudo-coherent. 

\begin{prop} \label{3.6}
An object $K^{\bullet} \in \mathcal{D}^{-}(A)$ is pseudo-coherent iff $H^{i}(K^{\bullet})$ is a finite $A$-module for all $i$. 
\end{prop}

\begin{proof}
If $K^{\bullet} \in \mathcal{D}^{-}(A)$ is pseudo-coherent then every cohomology $H^{i}(K^{\bullet})$, is a finite $A$-module. For the converse suppose that $H^{i}(K^{\bullet})$ is a finite $A$-module for all $i$. By Lemmas \ref{1.3} and \ref{1.4}, $H^{i}(K^{\bullet})[0]$ satisfies condition (2) of Lemma \ref{1.4}. Let $n$ be the largest integer such that $H^{n}(K^{\bullet})$ is non-zero. We will prove the Proposition by induction on $n$. We have the distinguished triangle
$$\tau_{\leq n-1}K^{\bullet} \rightarrow K^{\bullet} \rightarrow H^{n}(K^{\bullet})[-n].$$
Fix an integer $k$. Now $H^{n}(K^{\bullet})[-n]$ satisfies condition (2) of Lemma \ref{1.4} for $m=k$. Since condition (2) for $m=k$ is stable under extensions, $K^{\bullet}$ satisfies condition (2) for $m=k$ if $\tau_{\leq n-1}K^{\bullet}$ satisfies condition (2) for $m=k$. The result follows by induction.    
\end{proof}

\subsection{The Koszul complex}\label{sec:koscomcoh}

Let $a$ be a generator of $\zpe$ and note $$ \gamma = {\matrice a 0 0 1}, \;\;\; \varphi = {\matrice p 0 0 1}, \;\;\; \tau = {\matrice 1 0 p 1 } $$ which are generators of the group $\overline{P}^+$ satisfying the following relations: $$ \varphi \gamma = \gamma \varphi, $$ $$ \gamma \tau = \tau^{a^{-1}} \gamma, $$ $$ \varphi \tau^p = \tau \varphi, $$ giving a finite presentation of the group generated by those elements (proof: using those relations we can write any other relation as $\varphi^x \tau^y \gamma^z = 1$ which in turn implies $x = y = z = 0$) We thus have the following result. We need a lemma describing the nilpotent nature of $\tau-1$

In this subsection we compute a complex that computes $\overline{P}^+$-cohomology, cf. Lemma \ref{lem:comcohg}. This is an analogue of the Koszul complex in a non-commutative setting. The reader can compare this with the complex constructed in \cite[\S 1.5.1]{Ribeiro2011} which calculates Galois cohomology. In loc.cit. the construction is somewhat simpler because of the non-triviality of the center (of the group under consideration). Let $M$ be a $\overline{P}^+$-module such that the action of $\overline{P}^+$ extends to an action of the Iwasawa algebra $\zp[[\overline{P}^+]]$ and define

\begin{equation} 
\mathscr{C}_{\tau,\varphi,\gamma}: 0 \rightarrow M \xrightarrow{X} M \oplus M \oplus M \xrightarrow{Y} M \oplus M \oplus M \xrightarrow{Z} M \rightarrow 0
\end{equation} 
where
\begin{align*}
X(x) &= ((1-\tau)x,(1-\varphi)x,(\gamma - 1)x) \\
Y(x,y,z) &= ((1-\varphi\delta_p)x+(\tau-1)y, (\gamma\delta_a-1)x + (\tau-1)z, (\gamma-1)y + (\varphi-1)z) \\
Z(x,y,z) &= (\gamma\delta_a-1)x +(\varphi\delta_p-1)y+ (1-\tau)z
\end{align*}

where, $$\delta_p = \frac{1 - \tau^p}{1 - \tau} = 1 + \tau + \hdots + \tau^{p-1}, $$ for $a \in \zpe, b \in \zp$ $$ \frac{\tau^{ba} - 1}{\tau^a  - 1} = \sum_{n \geq 1} {ba \choose n} (\tau^{a} - 1)^{n - 1} \in \zp[[\tau - 1]],$$ $$ \delta_a = \frac{\tau^{a} - 1}{\tau - 1} $$ which is a well defined element since, as $\tau^{p^n} \to 1$ as $n$ tends to $+\infty$, $\tau - 1$ is topologically nilpotent in the Iwasawa algebra $\zp[[ \tau - 1 ]] = \zp[[\overline{U}]] \subseteq \zp[[\overline{P}^+]]$.

The construction of $\mathscr{C}_{\tau,\varphi,\gamma}$ is obtained from taking successive fibers of smaller complexes. Define

\begin{equation} 
\mathscr{C}_{\tau}: 0 \rightarrow M \xrightarrow{D} M \rightarrow 0
\end{equation}

where 
$$D(x):= (\tau-1)x$$
and

\begin{equation} 
\mathscr{C}_{\tau, \varphi}: 0 \rightarrow M \xrightarrow{E} M \oplus M \xrightarrow{F} M \rightarrow 0
\end{equation}

where 
\begin{align*}
E(x) &= ((\tau-1)x,(\varphi-1)x) \\
F(x,y) &= (\varphi\delta_p-1)x + (1-\tau)y
\end{align*}

We now define morphisms between the complexes. We note by $[\varphi-1]: \mathscr{C}_{\tau} \rightarrow \mathscr{C}_{\tau}$ the morphism:

$$
\begin{tikzcd} [row sep = large, column sep = large] 
\mathscr{C}_{\tau}: 0 \arrow[r] &
M \arrow[r] \arrow[d, "\varphi-1"] &
M \arrow[r] \arrow[d, "\varphi\delta_p-1"] &
0 \\
\mathscr{C}_{\tau}: 0 \arrow[r] &
M \arrow[r] &
M \arrow[r] &
0
\end{tikzcd}
$$

and $[\gamma-1]: \mathscr{C}_{\tau, \varphi} \rightarrow \mathscr{C}_{\tau, \varphi}$ the morphism:

$$
\begin{tikzcd} [row sep = large, column sep = large] 
\mathscr{C}_{\tau, \varphi}:0 \arrow[r] &
M \arrow[r] \arrow[d, "\gamma-1"] &
M \oplus M \arrow[r] \arrow[d, "s"] &
M \arrow[r] \arrow[d, "\gamma\delta_a -1"] & 
0 \\
\mathscr{C}_{u^{-}, \varphi}:0 \arrow[r] &
M \arrow[r] &
M \oplus M \arrow[r] &
M \arrow[r] &
0
\end{tikzcd}
$$

where $s(x,y) = ((\gamma\delta_a-1)x,(\gamma-1)y)$

\begin{lemm} \label{triang}
There are distinguised triangles
$$\mathscr{C}_{\tau,\varphi} \rightarrow \mathscr{C}_{\tau} \xrightarrow{[\varphi-1]} \mathscr{C}_{\tau}$$
and
$$\mathscr{C}_{\tau,\varphi,\gamma} \rightarrow \mathscr{C}_{\tau,\varphi} \xrightarrow{[\gamma-1]} \mathscr{C}_{\tau,\varphi}$$
in $\mathcal{D}^{-}(A)$. 
\end{lemm}

\begin{proof}
This is evident from the definition of the cone of a morphism in $\mathcal{D}^{-}(A)$ and the relations $\varphi\delta_{p}\cdot\gamma\delta_a = \gamma\delta_a\cdot\varphi\delta_{p}$, $(\gamma\delta_a-1)(\tau-1) = (\tau-1)(\gamma-1)$ and $(\varphi\delta_p-1)(\tau-1)=(\tau-1)(\varphi-1)$.   
\end{proof}


We now show that $C_{\tau,\varphi,\gamma}$ computes $\overline{P}^+$-cohomology. Recall the complex

$$\mathscr{C}_{\varphi,\gamma}: 0 \rightarrow M \xrightarrow{E'} M \oplus M \xrightarrow{F'} M \rightarrow 0$$
where 
\begin{align*}
E'(x) &= ((1-\varphi)x,(\gamma-1)x) \\
F'(x,y) &= (\gamma-1)x + (\varphi-1)y
\end{align*}

calculates the $A^+$-cohomology of $M$. There is an obvious restriction morphism $\mathscr{C}_{\tau,\varphi,\gamma} \to \mathscr{C}_{\varphi,\gamma}$ whose kernel (as a morphism in the abelian category of chain complexes) is

$$\mathscr{C}_{\varphi, \gamma}^{\mathrm{twist}}: 0 \rightarrow M \xrightarrow{E''} M \oplus M \xrightarrow{F''} M \rightarrow 0$$
where
\begin{align*}
E''(x) &= ((1-\varphi\delta_p)x,(\gamma\delta_a-1)x) \\
F''(x,y) &= ((\gamma\delta_a-1)x + (\varphi\delta_p-1)y
\end{align*}

\begin{lemm} \label{lem:comcohg}
The complex $\mathscr{C}_{\tau,\varphi,\gamma}$ calculates the $\overline{P}^+$-cohomology groups. That is $H^{i}(\mathscr{C}_{\tau,\varphi,\gamma}) = H^i(\overline{P}^+, M)$.
\end{lemm}

\begin{proof}
This is just a reinterpretation of the Hochschild-Serre spectral sequence. We have a distinguished triangle
\begin{equation}\label{eq:fedz}
\mathscr{C}_{\tau,\varphi,\gamma} \to \mathscr{C}_{\varphi,\gamma} \xrightarrow{1-\tau} \mathscr{C}_{\varphi, \gamma}^{\mathrm{twist}}
\end{equation}
in the derived category $\mathcal{D}^{-}(A)$, where the morphism 
\[
1- \tau \colon \mathscr{C}_{\varphi,\gamma} \to \mathscr{C}_{\varphi, \gamma}^{\mathrm{twist}}
\]
is component-wise just $1-\tau$. Let $\overline{U} = \begin{pmatrix} 1 & 0 \\ p\zp & 1 \end{pmatrix}$ so that $\overline{P}^+ = \overline{U} \rtimes A^+$. For a semi-group $G$ we denote by $R^G$ denote the derived functor of $(-)^{G}$. For $M$ a $\overline{P}^+$-module, we claim that\footnote{Here $R^{\overline{U}}$ is viewed as a function from $\mathcal{D}^+(\overline{P}^+-\mathrm{Mod})$ to $\mathcal{D}^{+}(A^+-\mathrm{Mod})$.}
\[
R^{\overline{U}}(M) = (0 \to M \xrightarrow{1 - \tau} M^* \to 0)
\]
where $M^*$ is isomorphic to $M$ as $\overline{U}$-modules, but is equipped with a \emph{twisted} $(\varphi, \gamma)$-action (which we denote by $(\widetilde{\varphi}, \widetilde{\gamma})$):
\[
\widetilde{\varphi} \cdot m := \varphi\delta_p \cdot m \text{ and } \widetilde{\gamma} \cdot m := \gamma\delta_a \cdot m.
\]
First note that $1 - \tau \colon M \to M^*$ is indeed a morphism of $\overline{P}^+$-modules (this follows from the relations $(\gamma\delta_a-1)(\tau-1) = (\tau-1)(\gamma-1)$ and $(\varphi\delta_p-1)(\tau-1)=(\tau-1)(\varphi-1)$). Now $H^1(\overline{U}, M)$ is equipped with a natural $(\varphi,\gamma)$-action (which we denote by $(\varphi',\gamma')$):
\[
\varphi' \cdot c_\tau:= \varphi \cdot c_{\tau^p} \text{ and } \gamma' \cdot c_\tau:= \gamma \cdot c_{\tau^a},
\] 
where $c_\tau$ is the value of the 1-cocycle $c$ with $[c] \in H^1(\overline{U}, M)$, at $\tau$. To prove the claim it suffices to show that $\varphi \cdot c_{\tau^p} = \varphi \delta_p \cdot c_{\tau}$ and $\gamma \cdot c_{\tau^a} = \gamma\delta_a \cdot c_{\tau}$. However these follow from the fact that $c$ is a 1-cocycle. Thus by the Hochschild-Serre spectral sequence we have
\[
R^{\overline{P}^+}(M) = R^{A^+}(0 \to M \xrightarrow{1 - \tau} M^* \to 0).
\]   
Therefore applying $R^{A^+}$ to the distinguished triangle
\[
(0 \to M \xrightarrow{1 - \tau} M^* \to 0) \to M \xrightarrow{1 - \tau} M^*
\]
gives the distinguished triangle
\begin{equation}\label{eq:festr}
R^{\overline{P}^+}(M) \to R^{A^+}(M) \xrightarrow{1 - \tau} R^{A^+}(M^*)
\end{equation}
and it is easy to see that $R^{A^+}(M) = \mathscr{C}_{\varphi,\gamma}$ and $R^{A^+}(M^*) = \mathscr{C}_{\varphi, \gamma}^{\mathrm{twist}}$. The result now follows from comparing the triangles \eqref{eq:fedz} and \eqref{eq:festr}.  
\end{proof}

\subsection{Finiteness of cohomology}
In this subsection we show that the cohomology groups 
$$H^{i}(\overline{P}^+, \mathscr{R}_A(\delta_1,\delta_2))$$
are finite-type $A$-modules. The idea is to reduce the problem to finiteness of $A^+$-cohomology, cf. \cite{chen2013}, \cite{kedlaya2014} and finiteness of \emph{twisted} $A^+$-cohomology. 

The first thing to note is that the complexes $\mathscr{C}_{\tau, \varphi, \gamma}$ are well defined for $M \in \{ \Robba^+_A(\delta), \Robba_A(\delta), \Robba_A^-(\delta) \}$, which is a consequence of the following lemma.

\begin{lemma}
Let $M \in \{ \Robba^-_A(\delta_1, \delta_2), \Robba_A(\delta_1, \delta_2), \Robba_A^+(\delta_1, \delta_2) \}$. The action of $\overline{P}^+$ extends by continuity to an action of the distribution algebra $\mathscr{D}(\overline{P}^+, A)$. In particular, $M$ is equipped with an action of the Iwasawa algebra $\zp[[\overline{P}^+]]$.
\end{lemma}

\begin{proof}
For the proof of this lemma, we use some facts of \S \ref{RtimesP1} (which is independent of the present section). For $M \in \{ \Robba_A^+(\delta_1, \delta_2), \Robba^-_A(\delta_1, \delta_2) \}$, the result is a consequence of the isomorphisms $\Robba^+_A(\delta_1) \boxtimes_\omega \P^1 \cong B_A(\delta_2, \delta_1)^* \otimes \omega$ and $\Robba^-_A(\delta_1) \boxtimes_\omega \P^1 \cong B_A(\delta_1, \delta_2)$ of Lemma \ref{lem:princser}, the fact that the locally analytic principal series are equipped with an action of the distribution algebra $\mathscr{D}(G, A)$  and the fact that, since $\overline{P}^+$ stabilizes $\zp$, then $\Robba_A^{(\pm)}(\delta_1, \delta_2) = (\Robba_A^{(\pm)}(\delta_1) \boxtimes_\omega \zp) \otimes \delta_2^{-1}$ inherits an action of the distribution algebra $\mathscr{D}(\overline{P}^+, A)$, and in particular an action of the Iwasawa algebra $\zp[[\overline{P}^+]]$. For $M = \Robba_A(\delta_1, \delta_2)$, the result follows by the same arguments noting that, since $\Robba(\delta_1) \boxtimes_\omega \P^1$ is an extension of $\Robba_A^-(\delta) \boxtimes_\omega \P^1$ by $\Robba_A^+(\delta) \boxtimes_\omega \P^1$ in the category of separately continuous $\mathscr{D}(G, A)$-modules, it is also equipped with an action of $\mathscr{D}(G, A)$.
\end{proof}

The main theorem of this subsection is the following. 

\begin{theo}
If $M = \Robba_A(\delta_1, \delta_2)$ then $\mathscr{C}_{\tau, \varphi, \gamma} \in \mathcal{D}^{-}_{\mathrm{pc}}(A)$. In particular, the $A$-modules $H^{i}(\overline{P}^{+},\mathscr{R}(\delta_1, \delta_2))$ are finite. 
\end{theo}

\begin{proof}
Recall that we have the distinguished triangle
$$\mathscr{C}_{\tau,\varphi,\gamma} \to \mathscr{C}_{\varphi,\gamma} \to \mathscr{C}_{\varphi, \gamma}^{\mathrm{twist}}$$
in the derived category $\mathcal{D}^{-}(A)$. By Theorem 4.4.2, \cite{kedlaya2014}, $\mathscr{C}_{\varphi,\gamma} \in \mathcal{D}^{-}_{\mathrm{pc}}(A)$. Thus by Lemma \ref{lem:comcohg}, to prove the result it is enough to show $\mathscr{C}_{\varphi, \gamma}^{\mathrm{twist}} \in \mathcal{D}^{-}_{\mathrm{pc}}(A)$. This now follows from Lemma \ref{finitecohom}  
\end{proof}

\begin{lemm} \label{finitecohom}
For $M = \mathscr{R}(\delta_1, \delta_2)$, the $A$-modules $H^{i}(\mathscr{C}_{\varphi, \gamma}^{\mathrm{twist}})$ are finite.
\end{lemm}

\begin{proof}
To prove this lemma, we still proceed by a d\'evissage argument. We define a complex 
$$
\begin{tikzcd} [row sep = large, column sep = large] 
\mathscr{C}_{\varphi \delta_p}: 0 \arrow[r] &
M \arrow[r, "1 - \varphi \delta_p"] &
M \arrow[r] & 
0
\end{tikzcd}
$$
and we observe that we have a distinguished triangle  \[ \mathscr{C}_{\varphi, \gamma}^{\mathrm{twist}} \to \mathscr{C}_{\varphi \delta_p} \xrightarrow{1 - \gamma \delta_a} \mathscr{C}_{\varphi \delta_p}. \]

Moreover, by taking long exact sequences associated to the short exact sequence $0 \to \Robba^+_A(\delta_1, \delta_2) \to \Robba_A(\delta_1, \delta_2) \to \Robba^-_A(\delta_1, \delta_2) \to 0$, it is enough to show finiteness for $\Robba^+_A(\delta_1, \delta_2)$ and $\Robba^-_A(\delta_1, \delta_2)$. The lemma follows from \ref{cols}.
\end{proof}

\begin{lemm} \label{cols}
For $M \in \left\{ \mathscr{R}^{+}(\delta_1, \delta_2), \mathscr{R}^{-}(\delta_1, \delta_2) \right\}$, the $A$-modules $H^{i}(\mathscr{C}_{\varphi, \gamma}^{\mathrm{twist}})$ are finite.
\end{lemm} 

\begin{proof}
The case of $\mathscr{R}^{+}(\delta_1, \delta_2)$ follows directly from lemma \ref{varphidelta1} below, which shows that the cohomology of the complex $\mathscr{C}_{\varphi \delta_p}$ is already of finite type.

For $\mathscr{R}^{-}(\delta_1, \delta_2)$, the long exact sequence associated to the triangle $\mathscr{C}_{\varphi, \gamma}^{\mathrm{twist}} \to \mathscr{C}_{\varphi \delta_p} \xrightarrow{1 - \gamma \delta_a} \mathscr{C}_{\varphi \delta_p}$ yields 
\begin{align*}
0 &\to H^0(\mathscr{C}_{\varphi, \gamma}^{\mathrm{twist}}) \to H^0(\mathscr{C}_{\varphi \delta_p}) \xrightarrow{1 - \gamma \delta_a} H^0(\mathscr{C}_{\varphi \delta_p}) \to H^1(\mathscr{C}_{\varphi, \gamma}^{\mathrm{twist}}) \to H^1(\mathscr{C}_{\varphi \delta_p}) \xrightarrow{1 - \gamma \delta_a} H^0(\mathscr{C}_{\varphi \delta_p})\\ 
&\to H^2(\mathscr{C}_{\varphi, \gamma}^{\mathrm{twist}}) \to 0, 
\end{align*}
and the result follows then from lemmas \ref{varphidelta2}  and \ref{coldbean}
\end{proof}  

\begin{lemm} \label{varphidelta1}
The operator $1-\varphi\delta_p: \mathscr{R}_{A}^{+}(\delta_1, \delta_2) \rightarrow \mathscr{R}_{A}^{+}(\delta_1, \delta_2)$ has finite kernel and cokernel.
\end{lemm}

\begin{proof}
The proof of this lemma is an adaptation of lemma \ref{sum2}. Let $N$ be big enough such that $|\delta(p) p^N| < 1$. We show that $1 - \varphi \delta_p: T^N \mathscr{R}_{A}^{+}(\delta_1, \delta_2) \rightarrow T^N \mathscr{R}_{A}^{+}(\delta_1, \delta_2)$ is bijective. For that, we construct an inverse of this operator by proving that $\sum_{k \geq 0} (\varphi \delta_p)^k$ converges. Observe that $(\varphi \delta_p)^k = \varphi^k \delta_{p^k}$ and that the operator $\delta_{p^k} = 1 + \tau + \hdots + \tau^{p^k - 1} = p^k + (\tau - 1) + \hdots + (\tau^{p^k - 1} - 1)$ is bounded (independently of $k$) by a constant $C$.

By identifying $\mathscr{R}_{A}^{+}(\delta_1, \delta_2)$ with the space of analytic functions on the open unit ball equipped with the Fr\'echet topology given by the family of norms $(|\cdot|_{[0, r]})_{0 < r < 1}$ and the action of $\varphi$ twisted by $\delta(p)$, we have (cf. lemme 2.9.(ii), \cite{chen2013}) $|\varphi^k(T^N)|_{[0, r]} \leq C_r p^{-Nk}$ for some constant $C_r >0$ and hence, for $f \in T^N \Robba^+(\delta_1, \delta_2)$ and any $0 < r < 1$, \[ | (\varphi \delta_p)^k(f)|_{[0, r]} = | \varphi^k \delta_{p^k}(f)|_{[0, r]} \leq C C_r |f|_{[0, r]} \left( \frac{\lambda}{p^{N}} \right)^k, \] which shows that the expression $\sum_{k \geq 0} (\varphi \delta_p)^k$ converges. We deduce that the kernel and cokernel are, respectively, a submodule and a quotient of $\mathrm{Pol}_{\leq N}(\zp, A)$. This concludes the proof.
\end{proof}

\begin{lemm} \label{varphidelta2} \leavevmode
\begin{itemize} 
\item The operator $1 - \varphi \delta_p: \Robba^-_A(\delta_1, \delta_2) \to  \Robba^-_A(\delta_1, \delta_2)$ is injective.
\item If $N \geq 0$ is big enough, then $\mathscr{R}_A^{-}(\delta_1, \delta_2) = (1 - \varphi \delta_p)\mathscr{R}_A^{-}(\delta_1, \delta_2) + \mathrm{LA}(\zpe, A) + \mathrm{Pol}_{\leq N}(\zp, A)$.
\end{itemize}
\end{lemm}

\begin{proof}
For the first point, exactly as in the proof of lemma \ref{sum}, if $(1 - \varphi \delta_p) f = 0$ then $\varphi \delta_p f = f$ and, applying this and the identity $(\varphi \delta_p)^n = \varphi^n \delta_{p^n}$ successively \footnote{We denote $\delta_{p^n} = \frac{1-\tau^{p^n}}{1-\tau}$.}, we get $\varphi^n \delta_{p^n} f = f$ and so $f$ is supported on $p^n \zp$ for all $n \geq 0$ and hence vanishes everywhere. 

We now prove the second assertion. By a direct calculation solving a differential equation locally, we can show that every $\phi \in \mathscr{R}_A^{-}(\delta_1,\delta_2)$ is of the form $\phi_1 + P_\phi + (1-\alpha\varphi)\phi_2$, with $P_\phi  \in \mathrm{Pol}_{\leq N}(\zp ,A), \phi_2 \in \Robba_A^-$ and $\phi_1$ who is zero in a neighbourhood of $0$, and thus of the form $\sum_{i = 0}^{n-1} \phi_{1,i}$, with $\phi_{1,i} \in \mathrm{LA}(p^i \zpe,A)$. Writing $\varphi^i \psi^i (\phi_{1,i}) = \varphi^i \delta_{p^i} \cdot \delta_{p^i}^{-1} \psi^i = (\varphi \delta_p)^i (\delta_p^{-1} \psi)^i = (1 - (1 - \varphi \delta_p))^{i} (\delta_p^{-1} \psi)^{i}(\alpha^{-i}\phi_{1,i})$ and upon expanding $(1-(1-\varphi \delta_p))^i$ expresses $\phi_{1,i}$ as a sum of elements in $(1-\varphi \delta_p) \mathscr{R}_A^{-}(\delta_1, \delta_2)$ and $\psi^{i}(\alpha^{-i}\phi_{1,i}) \in \mathrm{LA}(\zpe, A)$.
\end{proof}

\begin{lemm} \label{coldbean}
The operator $1-\gamma\delta_a: \mathrm{LA}(\zpe,A) \to \mathrm{LA}(\zpe,A)$ has finite kernel and cokernel (as $A$-modules)
\end{lemm}

\begin{proof}
For the sake of brevity write $M = \mathrm{LA}(\zpe,A)$. We have a morphism of complexes (in the abelian category of chain complexes)
$$
\begin{tikzcd} [row sep = large, column sep = large] 
0 \arrow[r] &
M \arrow[r, "1-\gamma"] \arrow[d, "1-\tau"] &
M \arrow[r] \arrow[d, "1 - \tau"] &
0 \\
0 \arrow[r] &
M \arrow[r, "1-\gamma\delta_a"] &
M \arrow[r] &
0
\end{tikzcd}
$$
Note that the cokernel of this morphism of complexes vanishes by Lemma \ref{cool} (the same proof carries over with $L$ replaced by $A$). Thus by Proposition \ref{2.25}, it suffices to show that $M^{\tau=1}$ is a finite $A$-module.  Take $f \in M^{\tau=1}$. Then by definition of the action of $\tau$ on $\mathscr{R}^{-}_{A}(\delta_1,\delta_2)$ we have
\[
f(x) = \delta(1-px)f\left(\frac{x}{1-px}\right).
\]
Repeating this procedure we see that the value of $f(x)$ determines the value of $f\left(\frac{x}{1-kpx} \right)$ for all $k\in \Z$. Now $1-p\Z$ is dense in $1-p\zp$ and so $(1-p\Z)^{-1}$ is dense in $1-p\zp$. By continuity of $f$, this implies that the values $f(1)$, $f(2)$, $\ldots$, $f(p-1)$ determine $f$ completely. This proves the result.   

\end{proof}

\section{The $\overline{P}^+$-cohomology} \label{sec:maincofinz}

In this section we fix a finite extension $L$ of $\qp$, two continuous characters $\delta_1,\delta_2: \zpe \rightarrow L^\times$ and we consider the modules $\Robba_L^+(\delta_1, \delta_2), \Robba_L(\delta_1, \delta_2)$ and $\Robba_L^-(\delta_1, \delta_2)$ (for brevity we will omit the subscript $L$). We systematically calculate all $\overline{P}^+$-cohomology groups of these modules, which will be essential in comparing them to their $A^+$-cohomology. This section is inspired by combining two observations. The first is that if $M$ is equipped with a continuous action of $\overline{P}^{+}$ such that this action induces an action of the Lie algebra of $\overline{P}^+$, then we can simplify cohomological calculations by passing to the Lie algebra. The second is that we have a good enough understanding of the infinitesimal action of the Lie algebra on a $(\varphi, \Gamma)$-module so as to be able to make explicit computations (cf. \cite{Berpaddif}, \cite{dos2012}). For the commodity of the reader, the main results of this section can be summarized as follows:

\begin{prop} \label{cohomfinal} \leavevmode
\begin{itemize}
\item Let $M_+ = \Robba^+(\delta_1, \delta_2)$.
\begin{enumerate}
\item If $\delta_1 \delta_2^{-1} \notin \{ x^{-i}, i \in \N \}$, then $H^j(\overline{P}^+, M_+) = 0$ for all $j$.
\item  If $\delta_1 \delta_2^{-1} = \mathbf{1}_\qpe$, then $\dim_L H^j(\overline{P}^+, M_+) = 1, 1, 1, 0$ for $j = 0, 1, 2, 3$.
\item  If $\delta_1 \delta_2^{-1} = x^{-i}$, $i \geq 1$, then $\dim_L H^j(\overline{P}^+, M_+) = 1, 3, 3 , 1$ for $j = 0, 1, 2, 3$.
\end{enumerate}
\item Let $M_- = \Robba^-(\delta_1, \delta_2)$.
\begin{enumerate}
\item  If $\delta_1 \delta_2^{-1} \notin \{ \chi x^i, i \in \N \}$,  then $\dim_L H^j(\overline{P}^+, M_-) = 0, 1, 1, 0$ for $j = 0, 1, 2, 3$.
\item If $\delta_1 \delta_2^{-1} = \chi x^i$, $i \in \N$, then $\dim_L H^j(\overline{P}^+, M_-) = 0, 2, 2, 1$ for $j = 0, 1, 2, 3$.
\end{enumerate}
\item Let $M = \Robba(\delta_1, \delta_2)$. 
\begin{enumerate}
\item If $\delta_1 \delta_2^{-1} \notin \{ x^{-i}, i \in \N \} \cup \{ \chi x^i, i \in \N \}$, then $\dim_L H^j(\overline{P}^+, M) = 0, 1, 1, 0$, for $j = 0, 1, 2, 3$.
\item If $\delta_1 \delta_2^{-1} = \mathbf{1}_\qpe$, then $\dim_L H^j(\overline{P}^+, M) = 1, 2, 2, 0$, for $j = 0, 1, 2, 3$.
\item If $\delta_1 \delta_2^{-1} = x^{-i}, i \geq 1$, then $\dim_L H^j(\overline{P}^+, M) = 1, 3, 2, 0$, for $j = 0, 1, 2, 3$.
\item If $\delta_1 \delta_2^{-1} =\chi x^i, i \in \N$, then $\dim_L H^j(\overline{P}^+, M) = 0, 2, 2, 1$, for $j = 0, 1, 2, 3$.
\end{enumerate}
\end{itemize}
\end{prop}

\begin{rema} 
Observe that the result about $H^1(\overline{P}^+, \Robba^+(\delta_1, \delta_2))$ when $\delta_1 \delta_2^{-1} = x^{-i}$, $i \geq 1$, is in contradiction with \cite[Lemme 5.21]{colmez2015}. There seems to be a mistake in loc.cit., where the twisted action of $A^+$ on $H^1(\overline{U}, M)$ is not taken into account. This changes slightly the results of \cite{colmez2015}, getting unicity of the correspondence only for the non-pathological case $\delta_1 \delta_2^{-1} \notin \{ x^{-i}, i \geq 1 \}$ (indeed, the restriction $H^1(\overline{P}^+, \Robba(\delta_1, \delta_2)) \to H^1(A^+, \Robba(\delta_1, \delta_2))$ turns out to be only surjective, but not injective). The authors plan to study this supplementary extensions in more detail in the near future.
\end{rema}

As an immediate corollary from the results of this proposition, we can deduce the following.

\begin{prop} \label{cohomfinal2}
The restriction morphism \[ H^1(\overline{P}^+, \Robba(\delta_1, \delta_2)) \to H^1(A^+, \Robba(\delta_1, \delta_2)) \] is a surjection. Moreover, if $\delta_1 \delta_2^{-1} \notin \{ x^{-i}, i \geq 1\}$, then it is an isomorphism.
\end{prop}

\subsection{The Lie algebra complex} \label{sublieal}

We note

$$a^{+} =
\begin{pmatrix}
1 & 0 \\
0 & 0
\end{pmatrix}, \hspace{1mm}
a^{-} =
\begin{pmatrix}
0 & 0 \\
0 & 1
\end{pmatrix}, \hspace{1mm}
u^{+} =
\begin{pmatrix}
0 & 1 \\
0 & 0
\end{pmatrix}, \hspace{1mm}
u^{-} = 
\begin{pmatrix}
0 & 0 \\
1 & 0
\end{pmatrix}
$$
the usual generators of the Lie algebra $\mathfrak{gl}_2$ of $\mathrm{GL}_2$. We note that $[a^{+},u^{-}] = -u^{-}$ and $p\varphi u^{-} = u^{-}\varphi$.

Denote by $H^{i}_{\Lie}(\overline{P}^{+},M)$ the cohomology groups of the complex
\begin{equation} \label{exactp}
\mathscr{C}_{u^{-},\varphi,a^{+}}: 0 \rightarrow M \xrightarrow{X'} M \oplus M \oplus M \xrightarrow{Y'} M \oplus M \oplus M \xrightarrow{Z'} M \rightarrow 0
\end{equation} 
where
\[ X'(x) = ((\varphi-1)x,a^{+}x,u^{-}x) \]
\[ Y'(x,y,z) = (a^{+}x - (\varphi-1)y, u^{-}y - (a^{+} +1)z, (p\varphi-1)z - u^{-}x) \]
\[ Z'(x,y,z) = u^{-}x + (p\varphi -1)y + (a^{+} + 1)z \]


Let $\tilde{P}:= \begin{pmatrix} \zpe & 0 \\ p \zp & 1 \end{pmatrix}$. Note that $\tilde{P}$ is a $p$-pro-subgroup of $\overline{P}^{+}$. 

\begin{lemm} \label{lielem}
If $M \in \left\{ \mathscr{R}^{+}(\delta_1, \delta_2), \mathscr{R}^{-}(\delta_1, \delta_2), \mathscr{R}(\delta_1, \delta_2) \right\}$, the natural application
$$H^{i}(\overline{P}^{+},M) \rightarrow H^{0}(\tilde{P}, H^{i}_{\Lie}(\overline{P}^{+},M))$$
is an isomorphism. 
\end{lemm}

\begin{proof}
The same proof as Lemma \ref{lem:comcohg} shows that there is a spectral sequence
\[
H^{i}_{\Lie}(A^+,H^{j}_{\Lie}(\overline{U},M)) \Rightarrow H^{i+j}_{\Lie}(\overline{P}^{+},M)
\]
where $H^{i}_{\Lie}(\overline{U},M)$ is defined to be the cohomology of the complex
\[
0 \to M \xrightarrow{u^{-}} M \to 0.
\]
For the definition of $H^{j}_{\Lie}(A^{+},-)$, cf. \cite[\S 5.2]{colmez2015}. The result now follows from \cite[Corollary 21]{georgana} by taking $\tilde{P}$-invariants on both sides. 
\end{proof}

\begin{rema} \label{actioncohom} For future calculations, we need to explicit the action of $\tilde{P}$ on the different Lie algebra cohomology groups. Recall that this group acts naturally on the module and by its adjoint action on the Lie algebra. Take for instance $(x, y, z) \in M^{\oplus 3}$ a 1-cocycle on the Lie algebra complex $\mathscr{C}_{u^-, \varphi, a^+}$ representing some cohomology class. An easy calculation shows that, if $\sigma_a \in A^0$, then, as cohomology classes \[\sigma_a \cdot (x, y, z) = (\sigma_a x, \sigma_a y, a \sigma_a z). \] If we want to calculate the action of $\tau$, say, on 1-coycles, in the same way, we get \[ \tau (x, y, z) = (\tau x + \tau\varphi \frac{\tau^{1 - p} - 1}{\log(\tau)} z , \tau \, y - p \tau \, z, \tau z).\] The formula for the first coordinate is obtained by using the fact that Lie algebra cohomology is calculated by `differentiating locally analytic cocycles at the identity' (cf. \cite{georgana}), and it can be taken as a formal formula (since there might be some convergence problems) but it will be enough for us (in general, one should replace $\tau$ by $\tau^n$ for some $n$ big enough).
\end{rema}

\subsection{Deconstructing cohomology} In order to compute cohomology we build the complex $\mathscr{C}_{u^{-}, \varphi, a^{+}}$ from smaller complexes. Define

\[ \mathscr{C}_{u^{-}}: 0 \rightarrow M \xrightarrow{D} M \rightarrow 0, \]
where 
$$D(x):= u^{-}x$$
and

\[ \mathscr{C}_{u^{-}, \varphi}: 0 \rightarrow M \xrightarrow{E} M \oplus M \xrightarrow{F} M \rightarrow 0, \]
where 
\begin{align*}
E(x) &= (u^{-}x,(\varphi-1)x) \\
F(x,y) &= (p\varphi-1)x - u^{-}y.
\end{align*}

We now define morphisms between the complexes. We note by $[\varphi-1]: \mathscr{C}_{u^{-}} \rightarrow \mathscr{C}_{u^{-}}$ the morphism:

$$
\begin{tikzcd} [row sep = large, column sep = large] 
\mathscr{C}_{u^{-}}: 0 \arrow[r] &
M \arrow[r] \arrow[d, "\varphi-1"] &
M \arrow[r] \arrow[d, "p\varphi-1"] &
0 \\
\mathscr{C}_{u^{-}}: 0 \arrow[r] &
M \arrow[r] &
M \arrow[r] &
0
\end{tikzcd}
$$

and $[a^{+}]: \mathscr{C}_{u^{-}, \varphi} \rightarrow \mathscr{C}_{u^{-}, \varphi}$ the morphism:

$$
\begin{tikzcd} [row sep = large, column sep = large] 
\mathscr{C}_{u^{-}, \varphi}:0 \arrow[r] &
M \arrow[r] \arrow[d, "a^{+}"] &
M \oplus M \arrow[r] \arrow[d, "s"] &
M \arrow[r] \arrow[d, "a^{+}+1"] & 
0 \\
\mathscr{C}_{u^{-}, \varphi}:0 \arrow[r] &
M \arrow[r] &
M \oplus M \arrow[r] &
M \arrow[r] &
0
\end{tikzcd}
$$

where $s(x,y) = ((a^{+}+1)x,a^{+}y)$

\begin{lemm} \label{triang}
We have the following distinguished triangles in $\mathcal{D}^{-}(L)$:
$$\mathscr{C}_{u^{-},\varphi} \rightarrow \mathscr{C}_{u^{-}} \xrightarrow{[\varphi-1]} \mathscr{C}_{u^{-}}, $$
$$\mathscr{C}_{u^{-},\varphi,a^{+}} \rightarrow \mathscr{C}_{u^{-},\varphi} \xrightarrow{[a^{+}]} \mathscr{C}_{u^{-},\varphi}. $$ 
\end{lemm}

\begin{proof}
This is evident from the definition of the cone of a morphism in $\mathcal{D}^{-}(L)$.  
\end{proof}

The following lemma will be the cornerstone of our cohomology calculations.

\begin{lemm} \label{iso1} \label{3.1} \leavevmode 
\begin{enumerate}
\item $H^0(\mathscr{C}_{u^-, \varphi, a^+}) = H^0([a^+] \colon H^0(\mathscr{C}_{u^-, \varphi}))$.
\item We have the following exact sequences in cohomology:
\begin{equation} \label{linee1}
0 \to H^1([a^+] \colon H^0(\mathscr{C}_{u^-, \varphi})) \to H^1(\mathscr{C}_{u^-, \varphi, a^+}) \to H^0([a^+] \colon H^1(\mathscr{C}_{u^-, \varphi})) \to 0,
\end{equation}
\begin{equation} \label{linee2}
0 \to H^1([\varphi - 1] \colon H^0(\mathscr{C}_{u^-})) \to H^1(\mathscr{C}_{u^-, \varphi}) \to H^0([\varphi - 1] \colon H^1(\mathscr{C}_{u^-})) \to 0.
\end{equation}
\item We have the following exact sequences in cohomology:
\begin{equation} \label{line1} 
0 \rightarrow H^1([a^{+}]: H^1(\mathscr{C}_{u^{-},\varphi})) \to H^2(\mathscr{C}_{u^{-},\varphi,a^{+}}) \to H^0([a^{+}]: H^2(\mathscr{C}_{u^{-},\varphi})) \rightarrow 0,
\end{equation}
\begin{equation} \label{line3} H^{2}(\mathscr{C}_{u^{-},\varphi}) \cong H^1([\varphi-1]: H^1(\mathscr{C}_{u^-})).
\end{equation}
\item $H^3(\mathscr{C}_{u^-, \varphi, a^+}) = H^1([a^+] \colon H^2(\mathscr{C}_{u^-, \varphi}))$.
\end{enumerate}
\end{lemm}

\begin{proof}
This follows from taking long exact sequences in cohomology from the triangles in Lemma \ref{triang}.
\end{proof}

The module from which we are taking cohomology should be clear from context, and whenever it is not specified, it means that the result holds for any such module. Notations should be clear, for instance, we have $H^0([a^+] \colon H^0(\mathscr{C}_{u^-, \varphi})) = M^{u^- = 0, \varphi = 1, a^+ = 0}$, $H^1([a^+] \colon H^0(\mathscr{C}_{u^-, \varphi})) = \coker([a^+] \colon H^0(\mathscr{C}_{u^-, \varphi})) = \coker(a^+ \colon M^{u^- = 0, \varphi = 1} \to M^{u^- = 0, \varphi = 1})$, et cetera desunt.

\begin{rema} \label{actioncohom1}
From the definition of the action of the group on the different terms of the Hochschild-Serre spectral sequence, or on the Lie algebra cohomology, as the composition of the natural action on the module with inner automorphisms on the group or on the Lie algebra, we can calculate the explicit action of $\tilde{P}$ on each constituent component of the exact sequences appearing in Lemma \ref{iso1} and Lemma \ref{3.1}. For instance, the action of $A^0$ on the third term of Equation \eqref{linee2}, on the terms of Equation \eqref{line3} and on those of Lemma \ref{iso1}(4) is twisted by $\chi$ (this comes from the identity $\sigma_a^{-1} u^- \sigma = a u^-$), while it acts as usual on the other terms (since $A^0$ commutes with itself and with $\varphi$). We can check that $\tau$ acts as usual on each separate term (but the sequences do not split as a sequence of $\overline{U}$-modules).
\end{rema}

\subsection{The Lie algebra cohomology of $\Robba^-(\delta_1, \delta_2)$} \label{Liecohom1}

The following conglomerate of technical lemmas on the action of the Lie algebra on $\Robba^-(\delta_1, \delta_2)$ will culminate in the main Proposition \ref{prop:H2LieR^-} of this section, calculating the $\overline{P}^+$-cohomology on this module, following the strategy suggested by Lemma \ref{3.1}. 

\begin{lemm}
Call $M = \Robba^-(\delta_1, \delta_2)$ and let $f \in M$. Under the identification (as modules) $M = \mathrm{LA}(\zp, L)$, the infinitesimal actions of $a^{+}$, $u^{-}$ and $\varphi$ are given by \footnote{Observe that, in the formula for $\varphi$ below, $f\left(\frac{x}{p}\right)$ is taken to be zero whenever $z \in \zpe$, so the precise formula should be $(\varphi f)(x) = \mathbf{1}_{p \zp}(x) \delta(p) f\left(\frac{x}{p}\right)$.}
\[ (a^{+}f)(x) = \kappa(\delta)f(x) - xf'(x), \]
\[ (u^{-}f)(x) = \kappa(\delta)xf(x) - x^{2}f'(x), \]
\[ (\varphi f)(x) = \delta(p)f\left(\frac{x}{p}\right).\]
\end{lemm}

\begin{proof}
First note tat, for $\begin{pmatrix} a & 0 \\ b & 1 \end{pmatrix} \in \overline{P}^{+}$ and $f \in \mathscr{R}^{-}(\delta_1, \delta_2)$, the action of $\overline{P}^{+}$ on $\mathscr{R}^{-}(\delta_1, \delta_2)$ is given by 
$$\bigg( \begin{pmatrix} a & 0 \\ b & 1 \end{pmatrix} \cdot f \bigg) (x)= \delta(a-bx)f \bigg( \frac{x}{a-bx} \bigg).$$
The action of $\varphi$ is now evident and that of $a^+$ follows from a direct calculation.


Viewing $\mathscr{R}^{-}(\delta_1, \delta_2)$ as the module $\mathscr{R}^{-}$ equipped with action of $\overline{P}^{+}$, we have, by \cite[Th\'eor\`eme 1.1]{dos2012},
$$u^{-} = -t^{-1}\nabla(\nabla+\kappa(\delta_2\delta_{1}^{-1})),$$
where here $\nabla = t\frac{d}{dt}$. Recall that, by the dictionary of functional analysis, multiplication by $t$ and $\frac{d}{dt}$ become, respectively, the operations of derivation and multiplication by $x$ on $\mathrm{LA}(\zp, L)$. In particular we have $\nabla(f) = f + xf'$. The description of the action of $u^-$ follows now from a direct computation.

\end{proof}

By an inoffensive abuse of language, we will talk in the sequel about the action of the elements $a^+$ and $u^-$ on $\mathrm{LA}(\zp, L)$ (resp. $\mathrm{LA}(\zpe, L)$), by which we mean their action on $\Robba^-(\delta_1, \delta_2)$ under the identification $\Robba^-(\delta_1, \delta_2) = \mathrm{LA}(\zp, L)$ (resp. $\Robba^-(\delta_1, \delta_2) \boxtimes \zpe = \mathrm{LA}(\zpe, L)$).

\subsubsection{Calculation of $H^0(\mathscr{C}_{u^-, \varphi, a^+})$:}

\begin{prop} \label{Lie0-}
Let $M = \Robba^-(\delta_1, \delta_2)$. We have $H^0(\mathscr{C}_{u^-, \varphi, a^+}) = 0$.
\end{prop}

\begin{proof}
This follows immediately from the injectivity of $1 - \delta(p) \varphi$ on $\mathrm{LA}(\zp, L)$.
\end{proof}

\subsubsection{Calculation of $H^2(\mathscr{C}_{u^-, \varphi})$:}

\begin{lemm} \label{cool}
The operator $u^{-}$ restricted to $\mathrm{LA}(\zpe, L)$ is surjective on $\mathrm{LA}(\zpe, L)$. 
\end{lemm}

\begin{proof}
This is an easy exercise on power series that we leave to the reader.
\end{proof}
\begin{lemm} \label{cool1} \label{banalized} 
If $M = \mathscr{R}^{-}(\delta_1, \delta_2)$ then:
\begin{enumerate}
\item If $\delta(p) \not\in \left\{ p^{i} \lvert \text{ } i\geq -1 \right\}$, or if $\delta(p) = p^i$ for some $i \geq 0$ and $\kappa(\delta) \not= i$, then $H^{2}(\mathscr{C}_{u^{-},\varphi}) = 0$.
\item Otherwise \footnote{i.e if $\delta(p) = p^{-1}$, or if $\delta(p) = p^i$ for some $i \geq 0$ and $\kappa(\delta) = i$.} $H^{2}(\mathscr{C}_{u^{-},\varphi})$ is of dimension $1$ naturally generated by $[x^{i+1}]$.
\end{enumerate}
\end{lemm}

\begin{proof}
Suppose first  $\delta(p) \not\in \left\{ p^{i} \lvert \text{ } i\geq -1 \right\}$. Note that, by \cite[Lemme 5.9]{colmez2015},
$$M = \mathrm{LA}(\zpe) \oplus (p\varphi-1)M.$$

By Lemma \ref{3.1},
$$H^{2}(\mathscr{C}_{u^{-},\varphi}) = \coker([\varphi-1]: M/u^{-}M).$$
Thus it suffices to show that the map $p\varphi-1: M/u^{-}M \to M/u^{-}M$ is surjective. This follows from the fact that
$$M \xrightarrow{p\varphi-1} M \to M/u^-M$$
is surjective (where the second map is the natural quotient map) since $u^-$ is surjective on $\mathrm{LA}(\zpe,L)$ by Lemma \ref{cool}.

Suppose now  $\delta(p)\in \left\{ p^{i} \lvert \text{ } i\geq -1 \right\}$. In this case, by \cite[Lemme 5.9]{colmez2015}, we have
\begin{equation} \label{eq:splgeh}
M = \left(\mathrm{LA}(\zpe) + (p\varphi-1)M \right) \oplus L \cdot x^{i+1},
\end{equation}
where $\mathrm{LA}(\zpe) \cap (p\varphi-1)M = L \cdot \mathbf{1}_{\zpe}x^{i+1}$. 

Suppose first $i \geq 0$. If $\kappa(\delta) \not=i$ then $(\kappa(\delta)-i)^{-1}u^{-}x^{i} = x^{i+1}$. Thus in this case the map $M \xrightarrow{p\varphi-1} M \to M/u^-M$ is surjective and the result follows. On the other hand if $\kappa(\delta) = i$ then $x^{i+1}$ is not in the image of $u^{-}$ and in this case $H^{2}(\mathscr{C}_{u^{-},\varphi})=L \cdot x^{i+1}$.

Finally consider the case $i=-1$. In this case, as $\mathbf{1}_{\zp}$ is never in the image of $u^{-}$ the result follows.
\end{proof}

\subsubsection{Calculation of $H^3(\mathscr{C}_{u^-, \varphi, a^+})$:}

At this stage, we can already deduce the following.

\begin{prop} \label{Lie3-} Let $M = \Robba^-(\delta_1, \delta_2)$.
\begin{itemize}
\item If $\delta(p) = p^i$, $i \geq -1$, and $\kappa(\delta) = i$, then $H^3(\mathscr{C}_{u^-, \varphi, a^+})$ is of dimension $1$ naturally generated by $[x^{i + 1}]$.
\item Otherwise $H^3(\mathscr{C}_{u^-, \varphi, a^+}) = 0$.
\end{itemize}
\end{prop}

\begin{proof}
This follows from Lemma \ref{iso1}(4) and Lemma \ref{banalized}, by observing that the action of $[a^+]$ on $H^2(\mathscr{C}_{u^-, \varphi}) = M / (u^-, p \varphi - 1)$ is given by $a^+ + 1$, and using the formula $(a^+ + 1)[x^{i + 1}] = (\kappa(\delta) - i) [x^{i + 1}]$.
\end{proof}

\subsubsection{Calculation of $H^1(\mathscr{C}_{u^-, \varphi})$:}

The following lemma describes the kernel of $u^-$ acting on $\Robba^-(\delta_1, \delta_2)$ in the appropriate way so as to calculate (cf. Corollary \ref{cokernonr}) the left term of Equation \eqref{linee2} of Lemma \ref{3.1}. Define \footnote{Observe that, upon developing $(x / i)^{\kappa(\delta)} = \sum_{k \geq 0} {\kappa(\delta) \choose k} (x/i - 1)^n$ and observing that $v_p({\kappa(\delta) \choose k}) \geq k (\min(\kappa(\delta), 0) - \frac{1}{p - 1})$, we see that $(x / i)^{\kappa(\delta)}$ is a well defined analytic function on $i + p^n \zp$ for $n$ big enough. }
$$X_{\kappa(\delta)}:= \bigg\{ f \in \mathrm{LA(\zpe)} \lvert \hspace{1mm} f(x) = \sum_{i \in (\Z/p^{n}\Z)^{\times}} c_{i}\left(\frac{x}{i}\right)^{\kappa(\delta)}\mathbf{1}_{i+p^{n}\zp} \text{ for some } n > 0 \bigg\}. $$

\begin{lemm} \label{3.2} \label{nonregcase}
If $M = \mathscr{R}^{-}(\delta_1, \delta_2) = \mathrm{LA}(\zp, L)$. Then
\begin{enumerate}
\item If $\delta(p) \not\in \left\{ p^{i} \lvert \text{ } i\geq 0 \right\}$, or if $\delta(p) = p^i$ for some $i \geq 0$ and $\kappa(\delta) \not= i$, then
$$ M^{u^{-} = 0} = X_{\kappa(\delta)} \oplus (1-\varphi) M^{u^{-} = 0}. $$
\item Otherwise
\[ M^{u^{-} = 0} = \big ( X_{\kappa(\delta)} + (1 - \varphi)M^{u^- = 0} \big ) \oplus L\cdot x^i, \]
Furthermore $X_{\kappa(\delta)} \cap  (1 - \varphi)M^{u^- = 0}$ is the line $L \cdot \mathbf{1}_{\zpe}x^i$.
\end{enumerate}
\end{lemm}

\begin{proof}
(1): First suppose that $\delta(p) \neq p^{-1}$. Take $f \in M^{u^{-} = 0}$. Since $\delta(p) \not\in \left\{ p^{i} \lvert \text{ } i\geq -1 \right\}$, by \cite[Lemme 5.9]{colmez2015}, we can uniquely write $f = f_{1} + (1-\varphi)f_{2}$ where $f_{1}$ is supported on $\zpe$ and $f_2 \in M$. Thus $0 = u^{-}f = u^{-}f_1 + u^{-}(1-\varphi)f_{2} = u^{-}f_1 + (1-p\varphi)u^{-}f_2$. We deduce, again using \cite[Lemme 5.9]{colmez2015}, that $u^{-}f_1 = u^{-}f_2 = 0$. Solving the differential equation $u^{-}f_1 = 0$ gives precisely $f_1 \in X_{\kappa(\delta)}$. 

Suppose now $\delta(p) = p^{-1}$. Repeating the same procedure as above, since $p \delta(p) = 1$, in this case \cite[Lemme 5.9]{colmez2015} gives   $u^{-}f_{1}=b\mathbf{1}_{\zpe}$ and $(1-p\varphi)u^{-}f_2 = -b\mathbf{1}_{\zpe}$ for some $b \in L$. This implies $u^{-}f_2 = -b \mathbf{1}_\zp$. This equation has no solution unless $b = 0$ (as can be easily seen upon expanding in power series around zero), in which case we obtain again $u^- f_1 = u^- f_2 = 0$.

Suppose finally that $\delta(p) = p^{i}$ for some $i\geq 0$ and $\kappa(\delta) \neq i$. By \cite[Lemme 5.9]{colmez2015}, we can write $f = f_{1} + (1-\varphi)f_{2} + ax^i$, where $f_{1}$ is supported on $\zpe$, $f_2 \in M$ and $a \in L$. Thus $0 = u^{-}f = u^{-}f_1 + u^{-}(1-\varphi)f_{2} + au^{-}x^i = u^{-}f_1 + (1-p\varphi)u^{-}f_2 + a(\kappa(\delta)-i)x^{i+1}$. The latter implies
\begin{align}
\label{eqq1} 0 &= u^{-}f_1 + (1-p\varphi)u^{-}f_2 \\
\label{eqq2} 0 &= a(\kappa(\delta)-i)x^{i+1}
\end{align} 
Again by \cite[Lemme 5.9]{colmez2015} (this time with $\alpha = p \delta(p) = p^{i + 1}$), the first equation implies $u^{-}f_1 = b\mathbf{1}_{\zpe}x^{i+1}$ and $(1-p\varphi)u^{-}f_2 = -b\mathbf{1}_{\zpe}x^{i+1}$ for some $b \in L$. This implies $u^{-}f_2 = -b x^{i+1}$. But all the solutions to the equation $u^{-}f_2 = -b x^{i+1}$ are of the form $ f_2 = - b \frac{x^i}{\kappa(\delta) - i} + f_2'$ where $f_2' \in M^{u^- = 0}$ is any element. Observe finally that the set \[ \{ b \cdot \big ( f_1 - (1 - \varphi) \frac{x^i}{\kappa(\delta) - i} \big ) \lvert \hspace{1mm} b \in L, \hspace{1mm} f_1 \in \mathrm{LA}(\zpe, L), \hspace{1mm} u^- f_1 = \mathbf{1}_\zpe x^{i + 1} \big \} \] is exactly $X_{\kappa(\delta)}$. On the other hand, since $\kappa(\delta) \not= i$, equation \eqref{eqq2} forces $a=0$. This shows that $M^{u^- = 0}$ is the sum of $X_{\kappa(\delta)}$ and $(1 - \varphi) M^{u^- = 0}$.

We now show that it is indeed a direct sum: if $\delta(p)$ is not equal to $p^i$ for some $i \geq 0$ then this is immediate. Suppose then that $\delta(p) = p^i$, $i \geq 0$ and $\kappa(\delta) \neq i$ and suppose that $f_1 = (1 - \varphi) g$ for some $g \in M^{u^- = 0}, f_1 \in X_{\kappa(\delta)}$. Then by the same lemma \cite[Lemme 5.9]{colmez2015} we get
\begin{equation} \label{eqq3}
f_1 = b' \mathbf{1}_\zpe x^i = (1 - \varphi) g
\end{equation} for some $b' \in L$. Then $0 = u^- f_1 = b' (\kappa(\delta) - 1)\mathbf{1}_\zpe x^{i + 1}$ so that $b' =0$ and hence $g = 0$ and $f_1 = 0$ as well.


(2): This follows from the same arguments as last two paragraphs, noting the following differences. On the one hand, if $\kappa(\delta) = i$, then the value $a$ in equation \eqref{eqq2} is free to take any value in $L$. On the other hand, if $\kappa(\delta) = i$, then there is no solution to the equation $u^{-}f_2 = -b x^{i+1}$ unless $b = 0$ (expand around zero). In that case equation \eqref{eqq1} gives $u^- f_1 = u^- f_2 = 0$. Finally, note that $L \cdot \mathbf{1}_{\zpe}x^i \in (1-\varphi)M^{u^{-} = 0}$, whence the result.
\end{proof}

Observing that $\mathrm{coker}([\varphi - 1] \colon H^0(\mathscr{C}_{u^-}) = \mathrm{coker}(\varphi - 1 \colon M^{u^- = 0})$, we get the following:

\begin{coro} \label{cokernonr}
If $M = \mathscr{R}^{-}(\delta_1, \delta_2)$ then
\begin{enumerate}
\item If $\delta(p) \not\in \left\{ p^{i} \lvert \text{ } i\geq 0 \right\}$, or if $\delta(p) = p^i$ for some $i \geq 0$ and $\kappa(\delta) \not= i$, then
\[ H^1([\varphi - 1] \colon H^0(\mathscr{C}_{u^-})) = X_{\kappa(\delta)}. \]
\item Otherwise
\[ H^1([\varphi - 1] \colon H^0(\mathscr{C}_{u^-})) = \big ( X_{\kappa(\delta)} / L \cdot \mathbf{1}_{\zpe}x^i \big )  \oplus L\cdot [x^i]. \]
\end{enumerate}
\end{coro}

\begin{proof}
This is an immediate consequence of Lemma \ref{nonregcase}.
\end{proof}

We now proceed to calculate the right side term of Equation \eqref{linee2} of Lemma \ref{3.1}. Note that the action of $[\varphi - 1]$ on $H^1(\mathscr{C}_{u^-}) = M / u^- M$ is given by $p \varphi - 1$.

\begin{lemm} \label{3.3}  \label{injnonreg}
If $M = \mathscr{R}^{-}(\delta_1, \delta_2)$ then
\begin{enumerate}
\item If $\delta(p) \not\in \left\{ p^{i} \lvert \text{ } i\geq -1 \right\}$, or if $\delta(p) = p^i$, $i \geq 0$, and $\kappa(\delta) \neq i$, then $H^0([\varphi - 1] \colon H^1(\mathscr{C}_{u^-})) = 0$.
\item Otherwise $H^0([\varphi-1] \colon H^1(\mathscr{C}_{u^-})) = L \cdot [x^{i+1}]$.
\end{enumerate}
\end{lemm}

\begin{proof}
(1): Suppose first that $\delta(p) \not\in \left\{ p^{i} \lvert \text{ } i\geq -1 \right\}$. Take $f \in M$ and suppose that $(1-p\varphi)f = u^{-}h$ for some $h \in M$. We write $h = h_1 + (1-\varphi)h_{2}$ with $h_1$ supported on $\zpe$. Then $(1-p\varphi)f = u^{-}h = u^{-}h_1 + (1-p\varphi)u^{-}h_2$. By uniqueness and by injectivity of $(1 - p \varphi)$ we obtain $f = u^{-}h_2$.

Now suppose $\delta(p) = p^{i}$ for some $i \geq 0$ and $\kappa(\delta) \not= i$. Take $f \in M$ and suppose that $(1-p\varphi)f = u^{-}h$ for some $h \in M$. We write $h = h_1 + (1-\varphi)h_{2} + ax^i$ with $h_1$ supported on $\zpe$ and $a \in L$. Then $(1-p\varphi)f = u^{-}h = u^{-}h_1 + (1-p\varphi)u^{-}h_2 + a(\kappa(\delta)-i)x^{i+1}$. This forces $a=0$ and $(1-p\varphi)(f-u^{-}h_2) = b\mathbf{1}_{\zpe}x^{i+1}$ for some $b \in L$. Hence $f-u^{-}h_2 = b x^{i+1}$ and thus $f = u^{-}\left(h_{2}+ \frac{b}{\kappa(\delta)-i}x^i\right)$.

(2): Suppose first $\delta(p) = p^{-1}$. Take $f \in M$ and suppose that $(1-p\varphi)f = u^{-}h$ for some $h \in M$. We write $h = h_1 + (1-\varphi)h_{2}$ with $h_1$ supported on $\zpe$. Then $(1-p\varphi)f = u^{-}h = u^{-}h_1 + (1-p\varphi)u^{-}h_2$. Then $u^{-}h_{1} = b\mathbf{1}_{\zpe} = (1-p\varphi)(f-u^{-}h_2)$. The last equality implies that $f-u^{-}h_{2} = b \mathbf{1}_\zp$ and hence $f =  b \mathbf{1}_\zp$ modulo $u^-$. In this case, if $b \not= 0$, then $b \mathbf{1}_\zp$ is not in the image of $u^{-}$. Hence 
$$\ker(p\varphi-1: M/u^{-}M) = L \cdot [\mathbf{1}_{\zp}].$$

Finally if $\delta(p) = p^{i}$ for some $i \geq 0$ and $\kappa(\delta) = i$, then the question is whether $x^{i+1}$ is in the image of $u^{-}$. Indeed taking $f(x) \in \mathrm{LA}(\zp,L)$ and expanding $u^{-}f$ in a small ball around $x=0$, we see that the coefficient of $x^{i+1}$ is $0$. Hence $x^{i+1}$ is not in the image of $u^{-}$, as we have already pointed out. This completes the proof. 
\end{proof}

We can at this stage easily deduce the following corollary, which gives a complete description of $H^1(\mathscr{C}_{u^-, \varphi})$:

\begin{prop} \label{Hline2}
Let $M = \mathscr{R}^{-}(\delta_1, \delta_2)$. Then
\begin{enumerate}
\item If $\delta(p) \not\in \left\{ p^{i} \lvert \text{ } i\geq -1 \right\}$, or if $\delta(p) = p^i$, $i \geq 0$, and $\kappa(\delta) \neq i$, then $H^{1}(\mathscr{C}_{u^{-},\varphi}) = X_{\kappa(\delta)}$.
\item If $\delta(p) = p^{-1}$, then $H^{1}(\mathscr{C}_{u^{-},\varphi})$ lives in an exact sequence
\[ 0 \to X_{\kappa(\delta)} \to H^{1}(\mathscr{C}_{u^{-},\varphi}) \to L \cdot [\mathbf{1}_\zp] \to 0. \]
\item If $\delta(p) = p^i$, $i \geq 0$, and $\kappa(\delta) = i$ then $H^{1}(\mathscr{C}_{u^{-},\varphi})$ lives in an exact sequence
\[ 0 \to \big ( X_{\kappa(\delta)} / L \cdot \mathbf{1}_\zpe x^i \big ) \oplus L \cdot [x^i] \to H^{1}(\mathscr{C}_{u^{-},\varphi}) \to L \cdot [x^{i + 1}] \to 0. \]
\end{enumerate}
\end{prop}

\begin{proof}
This is an immediate consequence of Lemmas \ref{3.1}, \ref{injnonreg} and Corollary \ref{cokernonr}.
\end{proof}

\subsubsection{Calculation of $H^1(\mathscr{C}_{u^-, \varphi, \gamma})$.}

We have already explicitly calculated the second exact sequence \eqref{linee2} of Lemma \ref{iso1} in Proposition \ref{Hline2}. The left hand side term of Equation \eqref{linee1} is easy to deal with:


\begin{lemm}\label{lem:triv}
If $M = \Robba^-(\delta_1, \delta_2)$ then $H^0(\mathscr{C}_{u^-, \varphi}) = 0$.
\end{lemm}

\begin{proof}
This follows immediately from the injectivity of $\delta(p) \varphi - 1$ on $\mathrm{LA}(\zp, L)$, cf. Lemma \ref{sum}(1).
\end{proof}

We calculate the kernel of $[a^+]$ on $\ker([\varphi-1] \colon H^1(\mathscr{C}_{u^-}))$:

\begin{lemm} \label{lem:kera}
If $M = \mathscr{R}^{-}(\delta_1, \delta_2)$ then
\begin{enumerate}
\item If $\delta(p) \not\in \left\{ p^{i} \lvert \text{ } i\geq - 1 \right\}$, or if $\delta(p) = p^{i}$ for some $i \geq -1$ and $\kappa(\delta) \not= i$, then $[a^+]$ on $H^0([\varphi-1] \colon H^1(\mathscr{C}_{u^-}))$ is injective.
\item Otherwise the kernel of $[a^+]$ on $H^0([\varphi-1] \colon H^1(\mathscr{C}_{u^-}))$ is (naturally isomorphic to) $L \cdot [x^{i+1}]$.
\end{enumerate}
\end{lemm}

\begin{proof}
We use Lemma \ref{injnonreg}. If $\delta(p) \not\in \left\{ p^{i} \lvert \text{ } i\geq 0 \right\}$ or if $\delta(p) = p^i$ for some $i \geq 0$ and $\kappa(\delta) \not= i$, then $\ker(p\varphi-1 \colon M/u^{-}M) = 0$ and the result is obvious. If $\kappa(\delta) = i$ then
\[
\ker(p\varphi-1 \colon M/u^{-}M) = L \cdot [x^{i+1}].
\]
In this case $(a^+ +1)[x^{i+1}] = 0$.

Suppose now $\delta(p) = p^{-1}$. In this case
$$\ker(p\varphi-1 \colon M/u^{-}M) = L \cdot [\mathbf{1}_{\zp}]$$
and, since $(a^{+}+1)\mathbf{1}_{\zp} = (\kappa(\delta)+1)\mathbf{1}_{\zp}$, the result now follows depending on whether $\kappa(\delta) = -1$ or not.
\end{proof}

The last needed ingredient is the kernel of $[a^+]$ on $H^1([\varphi - 1] \colon H^0(\mathscr{C}_{u^-}))$.

\begin{lemm} \label{lem:kera2}
If $M = \mathscr{R}^{-}(\delta_1, \delta_2)$ then
\begin{enumerate}
\item If $\delta(p) \not\in \left\{ p^{i} \lvert \text{ } i\geq 0 \right\}$, or if $\delta(p) = p^i$ for some $i \geq 0$ and $\kappa(\delta) \not= i$, then
\[ \ker([a^+] \colon H^1([\varphi - 1] \colon H^0(\mathscr{C}_{u^-}))) = X_{\kappa(\delta)}. \]
\item Otherwise
\[ \ker([a^+] \colon H^1([\varphi - 1] \colon H^0(\mathscr{C}_{u^-}))) = \big ( X_{\kappa(\delta)} / L \cdot \mathbf{1}_{\zpe}x^i \big )  \oplus L\cdot [x^i]. \]
\end{enumerate}
\end{lemm}

\begin{proof}
This is an immediate consequence of Lemma \ref{cokernonr}, noting that $a^+ X_{\kappa(\delta)} = 0$ and that $a^+ x^i = 0$ whenever $\kappa(\delta) = i$.
\end{proof}

Now we can calculate the first Lie algebra cohomology group with values in $\Robba^-(\delta_1, \delta_2)$.

\begin{prop} \label{prop:H1Lie}
Let $M = \mathscr{R}^{-}(\delta_1, \delta_2)$. Then 
\begin{enumerate}
\item If $\delta(p) \not\in \left\{ p^{i} \lvert \text{ } i\geq -1 \right\}$, or if $\delta(p) = p^i$ for some $i \geq -1$ and $\kappa(\delta) \not= i$,  then 
\[ H^1(\mathscr{C}_{u^-, \varphi, a^+}) = X_{\kappa(\delta)}. \]
\item If $\delta(p) = p^{i}$ for some $i\geq 0$ and $\kappa(\delta) = i$, then $H^1(\mathscr{C}_{u^-, \varphi, a^+})$ lives in an exact sequence
\[ 0 \to X_{\kappa(\delta)} / L \cdot \mathbf{1}_\zpe x^i \oplus L \cdot [x^i] \to H^1(\mathscr{C}_{u^-, \varphi, a^+}) \to L \cdot [x^{i + 1}] \to 0 \]
\item If $\delta(p) = p^{-1}$ and $\kappa(\delta) = -1$ then $H^1(\mathscr{C}_{u^-, \varphi, a^+})$ lives in an exact sequence
\[ 0 \to X_{\kappa(\delta)} \to H^1(\mathscr{C}_{u^-, \varphi, a^+}) \to L \cdot [\mathbf{1}_\zp] \to 0. \]
\end{enumerate}
\end{prop}

\begin{proof}
By Lemma \ref{lem:triv} and the short exact sequence \eqref{linee1} of Lemma \ref{iso1} we have
\[ H^1(\mathscr{C}_{u^-, \varphi, a^+}) \cong \ker([a^+] \colon H^1(\mathscr{C}_{u^-, \varphi})). \]

Since the short exact sequence \eqref{linee2} splits as a sequence of $a^+$-modules, we have
\[ 0 \to \ker(a^+ \colon M^{u^- = 0} / (\varphi - 1)) \to \ker([a^+] \colon H^1(\mathscr{C}_{u^-, \varphi})) \to \ker(a^+ + 1 \colon (M / u^- M)^{p \varphi = 1}) \to 0. \]

Now (1) (resp. (2), resp. (3)) follows from (1) (resp. (2), resp. (1)) of Lemma \ref{lem:kera2} and (1) (resp. (2), resp. (2)) of Lemma \ref{lem:kera}.
\end{proof}

\subsubsection{Calculation of $H^2(\mathscr{C}_{u^-, \varphi, \gamma})$:}

We start by calculating the left side term of equation \eqref{line1} of Lemma \ref{3.1}:

\begin{lemm} \label{blackdot} \label{3.16}
If $M = \mathscr{R}^{-}(\delta_1, \delta_2)$ then 
\begin{enumerate}
\item If $\delta(p) \not\in \left\{ p^{i} \lvert \text{ } i\geq -1 \right\}$, or if $\delta(p) = p^i$ for some $i \geq -1$ and $\kappa(\delta) \not= i$,  then \[ \coker([a^{+}]: H^{1}(\mathscr{C}_{u^{-},\varphi})) = X_{\kappa(\delta)}. \]
\item If $\delta(p) = p^{i}$ for some $i\geq 0$ and $\kappa(\delta) = i$, we have a short exact sequence
\[ 0 \to X_{\kappa(\delta)}/L \cdot \mathbf{1}_{\zpe}x^i \oplus L\cdot [x^i] \to \coker([a^{+}]: H^{1}(\mathscr{C}_{u^{-},\varphi})) \to L \cdot [x^{i+1}] \to 0. \] 
\item If $\delta(p) = p^{-1}$ and $\kappa(\delta) = -1$ then we have a short exact sequence
\[ 0 \to X_{\kappa(\delta)} \to \coker([a^{+}]: H^{1}(\mathscr{C}_{u^{-},\varphi})) \to L \cdot [\mathbf{1}_{\zp}] \to 0. \] 

\end{enumerate}
\end{lemm}

\begin{proof} \leavevmode
\begin{enumerate}
\item  Suppose first that $\delta(p) \not\in \left\{ p^{i} \lvert \text{ } i\geq -1 \right\}$, or that $\delta(p) = p^{i}$ for some $i\geq 0$ and $\kappa(\delta) \neq i$. By Proposition \ref{Hline2}(1), we have $H^{1}(\mathscr{C}_{u^{-},\varphi}) = X_{\kappa(\delta)}$. The result follows then by noting that the action of $[a^+]$ on this space is given by $a^+$, and that $a^+ X_{\kappa(\delta)} = 0$ (since $x a^+ f = u^- f$).

We now deal with the case where $\delta(p) = p^{-1}$ and $\kappa(\delta) \neq -1$. In this case $H^{1}(\mathscr{C}_{u^{-},\varphi})$ is described by (2) of Proposition \ref{Hline2} and the result follows from Lemma \ref{wrongfoot}(1) below.

\item If $\delta(p) = p^{i}$ for some $i\geq 0$ and $\kappa(\delta) = i$, then $H^{1}(\mathscr{C}_{u^{-},\varphi})$ is described by (3) of Proposition \ref{Hline2}. Note that $[a^+]$ acts as $a^+$ on the left term of the exact sequence, and as $a^+ + 1$ on the right hand side term. The result follows then by Lemma \ref{wrongfoot}(2) and by noting that $a^+ x^i = 0$.

\item This case follows in the same way, using Proposition \ref{Hline2}(2) and Lemma \ref{wrongfoot}(2).
\end{enumerate}
\end{proof}

The next lemma describes the cokernel of $[a^{+}]$ on $H^0([\varphi-1] \colon H^1(\mathscr{C}_{u^-}))$.

\begin{lemm} \label{wrongfoot}
If $M = \mathscr{R}^{-}(\delta_1, \delta_2)$ and  $\delta(p) = p^{i}$ for some $i \geq -1$ then
\begin{enumerate}
\item If $\kappa(\delta) \not= i$ then $\coker([a^{+}] \colon H^0([\varphi-1] \colon H^1(\mathscr{C}_{u^-}))) = 0$.
\item If $\kappa(\delta) = i$ then $\coker([a^+] \colon H^0([\varphi-1] \colon H^1(\mathscr{C}_{u^-})))$ is (naturally isomorphic to) $L \cdot [x^{i+1}]$.
\end{enumerate}
\end{lemm}

\begin{proof}
We use Lemma \ref{injnonreg}. Suppose first $i \geq 0$. If $\kappa(\delta) \not=i$ then $\ker(p\varphi-1: M/u^{-}M)=0$ and the result is obvious. If $\kappa(\delta) = i$ then
$$\ker(p\varphi-1: M/u^{-}M) = L \cdot [\mathbf{1}_{\zp}x^{i+1}].$$
In this case $(a^{+}+1)x^{i+1} = 0$. 

Suppose now $i = -1$. In this case
$$\ker(p\varphi-1: M/u^{-}M) = L \cdot [\mathbf{1}_{\zp}]$$
and so $(a^{+}+1)\mathbf{1}_{\zp} = (\kappa(\delta)+1)\mathbf{1}_{\zp}$. The result now follows depending on whether $\kappa(\delta) = -1$ or not.  
\end{proof}

We finally calculate the right hand side term $\ker([a^+] \colon H^2(\mathscr{C}_{u^-, \varphi}))$ of equation \eqref{line1} of Proposition \ref{3.1}.

\begin{lemm} \label{bladot}
If $M = \mathscr{R}^{-}(\delta_1, \delta_2)$. Then
\begin{enumerate}
\item If $\delta(p) \not\in \left\{ p^{i} \lvert \text{ } i\geq -1 \right\}$, or if $\delta(p) = p^{i}$ for some $i\geq -1$ and $\kappa(\delta) \not=i$, then $\ker([a^{+}] \colon H^2(\mathscr{C}_{u^-,\varphi}))=0$.
\item Otherwise $\ker([a^+] \colon H^2(\mathscr{C}_{u^{-},\varphi})) = L \cdot [x^{i+1}]$. 
\end{enumerate}
\end{lemm}

\begin{proof} \leavevmode
\begin{enumerate}
\item If $\delta(p) \neq p^{-1}$ the result follows since, by Lemma \ref{banalized}, we know that $H^{2}(\mathscr{C}_{u^{-},\varphi}) = 0$. If $\delta(p) = p^{-1}$ (and hence $\kappa(\delta) \neq - 1$) then $H^{2}(\mathscr{C}_{u^{-},\varphi})=L \cdot \mathbf{1}_{\zp}$ and, since $(a^{+}+1)\mathbf{1}_{\zp} = (\kappa(\delta)+1)\mathbf{1}_{\zp}$, it is injective.
\item Note first that $[a^{+}]$ acts on $H^{2}(\mathscr{C}_{u^{-},\varphi})$ as $a^{+}+1$. By (2) of Lemma \ref{banalized}, $H^{2}(\mathscr{C}_{u^{-},\varphi})=L \cdot x^{i+1}$ and the result follows since $(a^{+}+1)x^{i+1} = 0$.
\end{enumerate}
\end{proof}

We are now ready to compute $H^{2}(\mathscr{C}_{u^-, \varphi, a^+})$:

\begin{prop} \label{prop:H2LieR^-}
If $M = \mathscr{R}^{-}(\delta_1, \delta_2)$ then 
\begin{enumerate}
\item If $\delta(p) \not\in \left\{ p^{i} \lvert \text{ } i\geq -1 \right\}$, or if $\delta(p) = p^i$ for some $i \geq -1$ and $\kappa(\delta) \not= i$,  then 
\[ H^2(\mathscr{C}_{u^-, \varphi, a^+}) = X_{\kappa(\delta)}. \]
\item If $\delta(p) = p^{i}$ for some $i\geq 0$ and $\kappa(\delta) = i$, then $H^2(\mathscr{C}_{u^-, \varphi, a^+})$ lives in an exact sequence
\[ 0 \to Y_i \to H^2(\mathscr{C}_{u^-, \varphi, a^+}) \to L \cdot [x^{i + 1}] \to 0, \]
where  
\[ 0 \to X_{\kappa(\delta)}/L \cdot \mathbf{1}_{\zpe}x^i \oplus L\cdot [x^i] \to Y_i \to L \cdot [x^{i+1}] \to 0. \]
\item If $\delta(p) = p^{-1}$ and $\kappa(\delta) = -1$ then $H^2(\mathscr{C}_{u^-, \varphi, a^+})$ lives in an exact sequence
\[ 0 \to Y_{-1} \to H^2(\mathscr{C}_{u^-, \varphi, a^+}) \to L \cdot [\mathbf{1}_\zp] \to 0, \]
where
\[ 0 \to X_{\kappa(\delta)} \to Y_{-1} \to L \cdot [\mathbf{1}_{\zp}] \to 0. \]
\end{enumerate}
\end{prop}

\begin{proof}
(1) (resp. (2), resp. (3)) follows from (1) (resp. (2), resp (3)) of Lemma \ref{blackdot} and (1) (resp. (2), resp. (2)) of Lemma \ref{bladot}.
\end{proof}

\subsection{The $\overline{P}^+$-cohomology of $\Robba^-(\delta_1, \delta_2)$}

We can now just calculate the $\tilde{P}$-invariants of the Lie algebra cohomology to calculate the $\overline{P}^+$-cohomology of $\Robba^-(\delta_1, \delta_2)$.

\subsubsection{Calculation of $H^0(\overline{P}^+, \Robba^-(\delta_1, \delta_2))$:}

\begin{lemm} \label{cohom0-}
Let $M = \Robba^-(\delta_1, \delta_2)$. Then $H^0(\overline{P}^+, M) = 0$.
\end{lemm}

\begin{proof}
Obvious from Proposition \ref{Lie0-}.
\end{proof}

\subsubsection{Calculation of $H^1(\overline{P}^+, \Robba^-(\delta_1, \delta_2))$:}

\begin{lemm} \label{cohom1-}\label{cor:incr} \label{iso2} \label{lem:finge}
If $M = \mathscr{R}^{-}(\delta_1, \delta_2)$ then
\begin{enumerate}
\item If $\delta(x) \not= x^i$ for any $i \geq 0$ then $H^{1}(\overline{P}^{+}, M)$ is of dimension 1 and generated by $\mathbf{1}_{\zpe}\delta \otimes \delta$.
\item If $\delta(x) = x^i$ for some $i \geq 0$ then $H^{1}(\overline{P}^{+}, M)$ is of dimension 2.
\end{enumerate}
\end{lemm} 

\begin{proof} Start observing that the action of $\tilde{P}$ on the Lie algebra cohomology is explicitly given in Remark \ref{actioncohom} (cf. also Remark \ref{actioncohom1}): the action of $\tau$ on each of the extremities of the exact sequences of Lemma \ref{iso1} is the usual one, while the action of $A^0$ is given by the usual one, except for the term $\ker([\varphi - 1] \colon H^1(\mathscr{C}_{u^-}))$, on which its action is given by the usual action twisted by $\chi$.

- Suppose we are under the hypothesis of Proposition \ref{prop:H1Lie}(1). Then \[ H^1(\mathscr{C}_{u^-, \varphi, a^+}) = X_{\kappa(\delta)}. \] Note $\gamma = \sigma_a \in A^0$ a topological generator. Let $f \in H^0(A^0, X_{\kappa(\delta)})$ and write
$$f(x) = \sum_{i \in (\Z/p^{n}\Z)^{\times}} c_{i}\left(\frac{x}{i}\right)^{\kappa(\delta)}\mathbf{1}_{i+p^{n}\zp}, \;\;\; n \geq 0.$$
Since $\gamma f = f$, we have
\begin{align*}
\sum_{i \in (\Z/p^{n}\Z)^{\times}} c_{i}\left(\frac{x}{i}\right)^{\kappa(\delta)}\mathbf{1}_{i+p^{n}\zp} &= \delta(a)\sum_{i \in (\Z/p^{n}\Z)^{\times}} c_{i}\left(\frac{x}{ia}\right)^{\kappa(\delta)}\mathbf{1}_{ia+p^{n}\zp} \\
&= \delta(a)\sum_{i \in (\Z/p^{n}\Z)^{\times}} c_{ia^{-1}}\left(\frac{x}{i}\right)^{\kappa(\delta)}\mathbf{1}_{i+p^{n}\zp}. 
\end{align*}
Thus $\delta(a)c_{ia^{-1}} = c_{i}$ which implies $c_{a} = c_{1}\delta(a)$, for any $a \in \zpe$. This implies 
$$f(x) = c_{1}\delta(x)\mathbf{1}_{\zpe}.$$
Since $\delta \mathbf{1}_\zpe$ is fixed by $\tau$, the result follows from Lemma \ref{lielem}.

- We now place ourselves under the hypothesis of Proposition \ref{prop:H1Lie}(2). We have
\begin{equation} \label{eq:lala} 0 \to X_{\kappa(\delta)} / L \cdot \mathbf{1}_\zpe x^i \oplus L \cdot [x^i] \to H^1(\mathscr{C}_{u^-, \varphi, a^+}) \to L \cdot [x^{i + 1}] \to 0. \end{equation}
To calculate the $A^0$-invariants of $X := X_{\kappa(\delta)}/L \cdot \mathbf{1}_{\zpe}x^i$, we just consider the short exact sequence of $\Gamma$-modules
\[ 0 \to L \cdot \mathbf{1}_{\zpe}x^i \to X_{\kappa(\delta)} \to X \to 0 \] and take the associated long exact sequence. One easily sees that
\begin{itemize} 
\item if $\delta(x) \neq x^i$, then $H^0(A^0, L \cdot \mathbf{1}_{\zpe}x^i) = H^1(A^0, L \cdot \mathbf{1}_{\zpe}x^i) = 0$ so that $H^0(A^0, X) = H^0(A^0, X_{\kappa(\delta)}) = L \cdot \mathbf{1}_\zpe \delta$. 
\item If $\delta(x) = x^i$, then $A^0$ fixes $\mathbf{1}_{\zpe}x^i$, so we get a long exact sequence
\[ 0 \to L \cdot \mathbf{1}_\zpe \delta \to L \cdot \mathbf{1}_\zpe \delta \to H^0(A^0, X) \to L \cdot \mathbf{1}_\zpe \delta \xrightarrow{\alpha} H^1(A^0, X_{\kappa(\delta)}) \to H^1(A^0, X) \to 0. \]
Now $H^1_{\Lie}(A^0, X_{\kappa(\delta)}) = X_{\kappa(\delta)}$ and so $H^1(A^0, X_{\kappa(\delta)}) = H^0(A^0, X_{\kappa(\delta)}) = L \cdot 1_{\zpe} \delta$. It follows that $\alpha$ is an isomorphism and so $H^0(A^0, X) = 0$. Note also that this implies $H^1(A^0, X) = 0$. 

\end{itemize}

Thus, if $\delta(x) \not= x^i$, $H^1(\mathscr{C}_{u^-, \varphi, a^+})^{\tilde{P}} = L \cdot \mathbf{1}_{\zpe}\delta$. Suppose now $\delta(x) = x^i$. Since the action of $A^0$ on each of the terms of Equation \eqref{eq:lala} is locally constant, taking invariants is exact and we obtain
\[ 0 \to L \cdot [x^i] \to H^1(\mathscr{C}_{u^-, \varphi, a^+})^{A^0} \to L \cdot [x^{i+1}] \to 0. \]
taking $\overline{U}$-invariants of this exact sequence (note that $\tau([x^{i+1}]) = [\frac{x^{i+1}}{1-px}] = [x^{i+1}] + p [x^{i+2}] + p [x^{i+3}] + \hdots = [x^{i+1}] \mod u^-$ if $\delta(x) = x^i$), we obtain the desired result. 

- Finally, suppose that hypothesis of Proposition \ref{prop:H1Lie}(3) hold. We have
\[ 0 \to X_{\kappa(\delta)} \to H^1(\mathscr{C}_{u^-, \varphi, a^+}) \to L \cdot \mathbf{1}_\zp \to 0, \] and the result follows similarly.
\end{proof}

\subsubsection{Calculation of $H^2(\overline{P}^+, \Robba^-(\delta_1, \delta_2))$:}

\begin{lemm} \label{cohom2-} \label{ana} \label{benja}
If $M = \mathscr{R}^{-}(\delta_1, \delta_2)$ then 
\begin{enumerate}
\item If $\delta \neq x^i$ for any $i \geq 0$, then $H^{2}(\overline{P}^{+}, M)$ is of dimension 1 and generated by $\mathbf{1}_{\zpe}\delta \otimes \delta$. 
\item If $i \geq 0$ and $\delta(x) = x^i$ then $H^{2}(\overline{P}^{+}, M)$ is of dimension 3.
\end{enumerate}
\end{lemm}

\begin{proof}
- Suppose first that we are in the case of Proposition \ref{prop:H2LieR^-}(1). Then $H^{2}(\mathscr{C}_{u^{-},\varphi,a^{+}}) = X_{\kappa(\delta)},$ and the result follows as in Lemma \ref{cohom1-} above.

- We now deal with the case of Proposition \ref{prop:H2LieR^-}(2). We have
\begin{equation} \label{eaea} 0 \to Y_i \to H^2(\mathscr{C}_{u^-, \varphi, a^+}) \to L \cdot [x^{i + 1}] \to 0,
\end{equation} 
\[ 0 \to X_{\kappa(\delta)}/L \cdot \mathbf{1}_{\zpe}x^i \oplus L\cdot x^i \to Y_i \to L \cdot [x^{i+1}] \to 0. \]
Again, the action of $\tilde{P}$ on each component is described in Remark \ref{actioncohom1}. As before, since the action of $\tilde{P}$ is locally constant, taking $\tilde{P}$-invariants of the short exact sequence \eqref{eaea} gives
\[
0 \to (Y_i)^{\tilde{P}} \to H^2(\mathscr{C}_{u^-, \varphi, a^+})^{\tilde{P}} \to (L \cdot [x^{i+1}])^{\tilde{P}} \to 0.
\]
Now, $[x^{i+1}]$ is invariant under the action of $A^0$ and $\tau$ if and only if $\delta(x) = x^i$ for some $i \geq 0$, and the $\tilde{P}$-invariants of $Y_i$ were calculated in Lemma \ref{cohom1-}. This allows us to conclude.

- The case of Proposition \ref{prop:H2LieR^-}(3) is treated similarly.
\end{proof}

\subsubsection{Calculation of $H^3(\overline{P}^+, \Robba^-(\delta_1, \delta_2))$:}

\begin{lemm} \label{cohom3-} Let $M = \Robba^-(\delta_1, \delta_2)$.
\begin{itemize}
\item If $\delta = x^i$, $i \geq 0$, then $H^3(\overline{P}^+, \Robba^-(\delta_1, \delta_2))$ is of dimension $1$ naturally generated by $[x^{i + 1}]$.
\item Otherwise $H^3(\overline{P}^+, \Robba^-(\delta_1, \delta_2)) = 0$.
\end{itemize}
\end{lemm}

\begin{proof}
This follows by taking $\tilde{P}$-invariants to the results of Proposition \ref{Lie3-}, by observing that the action of $\tau$ is the natural one, and that of $A^0$ is twisted by $\chi$.
\end{proof}

\subsection{The $\overline{P}^+$-cohomology of $\Robba^+(\delta_1, \delta_2)$: a first reduction}

In this section we calculate all $\overline{P}^+$-cohomology groups of $\mathscr{R}^{+}(\delta_1, \delta_2)$ as described in Proposition \ref{cohomfinal}.

We first start with a lemma that allows us to reduce, as we have already done before for the $A^+$-cohomology (cf. \S \ref{caseofR+}), the calculation of $H^i(\overline{P}^+, \Robba^+(\delta_1, \delta_2))$ to that of $H^i(\overline{P}^+, \mathrm{Pol}_{\leq N}(\zp, L)^*(\delta_1, \delta_2))$ for $N \geq 0$ big enough, where $\mathrm{Pol}_{\leq N}(\zp, L)^*(\delta_1, \delta_2)$ denotes the sub-module of $\Robba^+(\delta_1, \delta_2)$ corresponding to $\mathrm{Pol}_{\leq N}(\zp, L)^*$ under the identification (as $L$-vector spaces) $\Robba^+(\delta_1, \delta_2) = \Robba^+$. We also recall that, under the Amice transform, we have an identification $\mathrm{Pol}_{\leq N}(\zp, L)^*(\delta_1, \delta_2) = \oplus_{i = 0}^N L \cdot t^i$. This module is stable under the action of $\overline{P}^+$ by \cite[Lemme 5.20]{colmez2015} \footnote{The first $\overline{P}^+$-cohomology group of this module is calculated in \cite[Lemme 5.21]{colmez2015} but, as we mentioned earlier, there are some small mistakes there, whence the incompatibility with our results.}, and the action of $\overline{P}^+$ is explicitly given by
\[ \sigma_a(t^j) = \delta_1 \delta_2^{-1}(a) a^j t^j, \;\;\; \varphi(t^j) = \delta_1 \delta_2^{-1}(p) p^j t^j, \;\;\; \tau(t^j) = \sum_{h = 0}^j {{\kappa - h} \choose {j - h}} p^{j - h} \, t^h, \] where we have set $\kappa = - \kappa(\delta_1 \delta_2^{-1}) - 1$. Observe that, if $\kappa \in \{ 0, 1, \hdots, N - 1 \}$ and $j = \kappa + 1$, then $\tau t^j = t^j$.

\begin{lemm} \label{compacohom}
We have, for every $i$ and for $N$ big enough, \[ H^i(\overline{P}^+, \Robba^+(\delta_1, \delta_2)) = H^i(\overline{P}^+, \mathrm{Pol}_{\leq N}(\zp, L)^*(\delta_1, \delta_2)).\]
\end{lemm}

\begin{proof}
This follows from the Hochschild-Serre spectral sequence and the same arguments of Lemma \ref{sum2}. Observe that it suffices to take $N$ such that $|\delta(p)| < p^N$.
\end{proof}

\subsection{The Lie algebra cohomology of $\mathrm{Pol}_{\leq N}(\zp, L)^*(\delta_1, \delta_2)$}

From now on until the end of this section, call $M = \mathrm{Pol}_{\leq N}(\zp, L)^*(\delta_1, \delta_2)$, which we identify with $ \oplus_{i = 0}^N L \cdot t^i$ equipped with the corresponding action of $\overline{P}^+$. Let us now calculate the Lie algebra action on the module $M$.

\begin{lemm} \label{hatbag}
For $f \in M$, the infinitesimal actions of $a^{+}$ and $u^{-}$ and the action of $\varphi$ are given by
\[ (a^{+}f)(t) = (\kappa(\delta_1)-\kappa(\delta_2))f(t) + tf'(t),\]
\[ (u^{-}f)(t) = (\kappa(\delta_2)-\kappa(\delta_1)-1)f'(t) - tf''(t), \]
\[ (\varphi f)(t) = \delta_1\delta_2^{-1}(p)f(pt). \]
\end{lemm}

\begin{proof}
First note that as a $A^+$-module, $\mathscr{R}^{+}(\delta_1, \delta_2)$ is identified with $\mathscr{R}^{+}(\delta_1\delta_2^{-1})$. The action of $\varphi$ is now evident. For the action of $a^{+}$ note that $(\sigma_af)(t) = \delta_{1}\delta_{2}^{-1}(a)f(at)$ for any $a \in \zpe$. The derivative of the function $g(a):= \delta_{1}\delta_{2}^{-1}(a)f(at)$ evaluated at $a=1$ is precisely $(\kappa(\delta_1)-\kappa(\delta_2))f(t) + tf'(t)$.

Finally viewing $\mathscr{R}^{-}(\delta_1, \delta_2)$ as the module $\mathscr{R}^{-}$ equipped with action of $\overline{P}^{+}$, we have by, \cite[Th\'eor\`eme 1.1]{dos2012},
$$u^{-} = -t^{-1}\nabla(\nabla-\kappa(\delta_2\delta_{1}^{-1})),$$
where here $\nabla = t\frac{d}{dt}$. Thus
\begin{align*}
(u^{-}f)(t) &= -\frac{d}{dt}(\nabla-\kappa(\delta_2\delta_{1}^{-1}))(f)(t) \\
&= -\frac{d}{dt}(tf'(t)-\kappa(\delta_2\delta_1^{-1})f(t)) \\
&= (\kappa(\delta_2)-\kappa(\delta_1)-1)f'(t) - tf''(t) 
\end{align*}  
\end{proof}
 
 \subsubsection{Calculation of $H^0(\mathscr{C}_{u^-, \varphi, a^+})$:}

Call $\kappa = \kappa(\delta_1 \delta_2^{-1})$.

\begin{lemm} \label{Lie0+} \leavevmode
\begin{enumerate} 
\item If $\delta_1 \delta_2^{-1}(p) = p^{-i}$, $i \geq 0$, and $\kappa = -i$, then $H^0(\mathscr{C}_{u^-, \varphi, a^+}) = L \cdot t^i$.
\item Otherwise $H^0(\mathscr{C}_{u^-, \varphi, a^+}) = 0$.
\end{enumerate}
\end{lemm}

\begin{proof}
The formula $u^-t^j = j (- \kappa - j) t^{j - 1}$ shows that $M^{u^- = 0} = L \cdot t^0$ if $\kappa \notin \{ -i, N \geq i \geq 1 \}$, and $M^{u^- = 0} = L \cdot t^0 \oplus L \cdot t^i$ if $\kappa = - i, N \geq i \geq 1$. 

Suppose $\kappa = - i, i \geq 1$. In this case, $t^0$ is not killed by $a^+$ and,since we have $a^+ t^j = (\kappa + j) t^j$ and $\varphi(t^j) = \delta_1 \delta_2^{-1}(p) p^j t^j$, we see that the term $t^i$ is in the kernel of $a^+$ and $\varphi - 1$ if and only if $\delta_1 \delta_2^{-1}(p) = p^{-i}$.

On the other cases, the term $t^0$ is in the kernel of $a^+$ and $\varphi - 1$ if and only if $\kappa = 0$ and $\delta_1 \delta_2^{-1}(p) = 1$. This completes the proof.
\end{proof}
 
\subsubsection{Calculation of $H^1(\mathscr{C}_{u^-, \varphi, a^+})$:}

In the following we note $\kappa = \kappa(\delta_1 \delta_2^{-1})$. Observe also, for the following statements, that $\delta_1 \delta_2^{-1}(p) = \delta(p)$ (since $\chi(p) = 1$).

\begin{lemm} \label{H1u-phi} \leavevmode
\begin{enumerate}
\item If $\delta(p) \notin \left\{p^{-i}, i \geq 0 \right\}$, or if $\delta(p) = p^{-i}$ for some $N \geq i \geq 1$ and $\kappa \neq -i$, then $H^1(\mathscr{C}_{u^-, \varphi}) = 0$.
\item If $\delta(p) = 1$ then $H^1(\mathscr{C}_{u^-, \varphi}) = L \cdot [t^0]$.
\item If $\delta(p) = p^{-i}$ for some $i \geq 1$ and $\kappa = - i$ then $H^1(\mathscr{C}_{u^-, \varphi})$ lives in an exact sequence
\[ 0 \to L\cdot [t^i] \to H^1(\mathscr{C}_{u^-, \varphi}) \to L \cdot [t^{i-1}] \to 0. \]
\end{enumerate}
\end{lemm}

\begin{proof}
We use equation \eqref{linee1} of Lemma \ref{iso1} to calculate this group. First observe that, if $\kappa \notin \{ -i, N \geq i \geq 1 \}$, then $H^1(\mathscr{C}_{u^-}) = L \cdot [t^N]$ and $H^1(\mathscr{C}_{u^-}) = L \cdot [t^{i - 1}] \oplus L \cdot [t^N]$ if $\kappa = -i$, $N \geq i \geq 1$. The result now follows by considering the kernel of $[\varphi - 1]$ on $H^1(\mathscr{C}_{u^-})$ and the cokernel of $[\varphi - 1]$ on $H^0(\mathscr{C}_{u^-})$. Observe that $N$ being chosen big enough, the term $[t^N]$ ends up always being killed.
\end{proof}

We now calculate the extremities of the short exact sequence of \eqref{linee1} of Lemma \ref{iso1}.

\begin{lemma} \label{LHSH1} \leavevmode
\begin{enumerate}
\item If $\delta(p) \notin \{ p^{-i}, i \geq 0 \}$, or if $\delta(p) = p^{-i}$, $i \geq 0$, and $\kappa \neq -i$, then $H^1([a^+] \colon H^0(\mathscr{C}_{u^-, \varphi})) = 0$.
\item Otherwise $H^1([a^+] \colon H^0(\mathscr{C}_{u^-, \varphi})) = L \cdot [t^i]$.
\end{enumerate}
\end{lemma}

\begin{proof} Observe first he formula
\[ a^{+}t^i = (\kappa + i) t^i. \] 
\begin{enumerate}
\item In the first case we have $H^0(\mathscr{C}_{u^-, \varphi})$ so in particular $H^1([a^+]: H^0(\mathscr{C}_{u^-, \varphi})) = 0$.
\item If $\delta(p) = p^{-i}$ for some $i \geq 0$ and $\kappa = -i$, then $H^0(\mathscr{C}_{u^-, \varphi}) = L \cdot t^i$ and $a^+ t^i = 0$ giving the result.
\end{enumerate}
\end{proof}

\begin{lemm} \label{RHSH1} \leavevmode
\begin{enumerate}
\item If $\delta(p) \notin \left\{p^{-i} \text{ } | \text{ } i \geq 0 \right\}$, or if $\delta(p) = p^{-i}$, $i \geq 0$, and $\kappa \neq -i$, then $H^0([a^+] \colon H^1(\mathscr{C}_{u^-, \varphi})) = 0$.
\item If $\delta(p) = 1$ and $\kappa = 0$, then $H^0([a^+] \colon H^1(\mathscr{C}_{u^-, \varphi})) = L \cdot [t^0]$.
\item Otherwise we have an exact sequence
\[ 0 \to L\cdot [t^i] \to H^0([a^+] \colon H^1(\mathscr{C}_{u^-, \varphi})) \to L \cdot [t^{i-1}] \to 0.\]
\end{enumerate}
\end{lemm}

\begin{proof} \leavevmode
\begin{enumerate}
\item If $\delta(p) \notin \left\{p^{-i} \text{ } | \text{ } i \geq 0 \right\}$, or if $\delta(p) = p^{-i}$ for some $i \geq 1$ and $\kappa \neq -i$, the group $H^1(\mathscr{C}_{u^-, \varphi})$ is already zero by Lemma \ref{H1u-phi}(i). If $\delta(p) = 1$ and $\kappa \neq 0$, then, by Lemma \ref{H1u-phi}(ii), $H^1(\mathscr{C}_{u^-, \varphi}) = L \cdot [t^0]$ and $a^+$ acts via multiplication by $\kappa$, so it is injective.

\item This follows easily from Lemma \ref{H1u-phi}(ii).

\item If $\delta(p) = p^{-i}$ for some $i \geq 1$ and $\kappa = -i$, then by Lemma \ref{H1u-phi}(3) we have an exact sequence
\[ 0 \to L\cdot [t^i] \to H^1(\mathscr{C}_{u^-, \varphi}) \to L \cdot [t^{i-1}] \to 0, \] where $[a^+]$ acts as $a^+$ on the LHS term and as $a^+ + 1$ on the RHS term. We deduce the result since $a^+ t^i = (a^+ + 1) t^{i - 1} = 0$.
\end{enumerate}
\end{proof}

We can now put everything together to calculate the Lie algebra cohomology.

\begin{prop} \label{H1LiePR+} Let $M = \mathrm{Pol}_{\leq N}(\zp, L)^*(\delta_1, \delta_2)$. Then
\begin{enumerate}
\item If $\delta(p) \notin \{p^{-i}, i \geq 0 \}$, or if $\delta(p) = p^{-i}$, $i \geq 0$, and $\kappa \neq -i$, then $H^1(\mathscr{C}_{u^-, \varphi, a^+}) = 0$.
\item If $\delta(p) = 1$ and $\kappa = 0$, then $H^1(\mathscr{C}_{u^-, \varphi, a^+})$ lives in a short exact sequence \[ 0 \to L \cdot [t^0] \to H^1(\mathscr{C}_{u^-, \varphi, a^+}) \to L \cdot [t^0] \to 0.\]
\item Otherwise we have short exact sequences
\[ 0 \to L \cdot [t^i] \to H^1(\mathscr{C}_{u^-, \varphi, a^+}) \to Z_i \to 0, \]
\[ 0 \to L \cdot [t^i] \to Z_i \to L \cdot [t^{i - 1}] \to 0. \]
\end{enumerate}
\end{prop}

\begin{proof}
This is an immediate consequence of Lemma \ref{LHSH1} and Lemma \ref{RHSH1}.
\end{proof}

\subsubsection{Calculation of $H^2(\mathscr{C}_{u^-, \varphi, a^+})$:}

Using the same methods based on Lemma \ref{3.1}, we calculate $H^2(\mathscr{C}_{u^-, \varphi, a^+})$ for the module $M = \mathrm{Pol}_{\leq N}(\zp, L)^*(\delta_1, \delta_2)$. The following series of lemmas systematically calculate each term appearing on Lemma \ref{3.1}. Since the calculations are in the same spirit, we leave the easy proofs to the reader. Call $\kappa = \kappa(\delta_1 \delta_2^{-1})$ as before.

The following two lemmas calculate $H^2(\mathscr{C}_{u^-, \varphi})$ and the kernel of $[a^+]$ acting on it.


\begin{lemm} \label{Lie2tmp} \leavevmode
\begin{enumerate}
\item If $\kappa = - i$, $i \geq 1$, and $\delta(p) = p^{-i}$, then $H^2(\mathscr{C}_{u^+, \varphi}) = L \cdot [t^{i-1}]$.
\item Otherwise $H^2(\mathscr{C}_{u^+, \varphi}) = 0$.
\end{enumerate}
\end{lemm}

\begin{lemm} \leavevmode
\begin{enumerate}
\item If $\kappa = -i$, $i \geq 1$, and $\delta(p) = p^{-i}$, then $H^0([a^+] \colon H^2(\mathscr{C}_{u^-, \varphi})) = L \cdot [t^{i-1}]$.
\item Otherwise $H^0([a^+] \colon H^2(\mathscr{C}_{u^-, \varphi})) = 0$.
\end{enumerate}
\end{lemm}

We calculate now the cokernel of $[a^+]$ on $H^1(\mathscr{C}_{u^-, \varphi})$.

\begin{lemm} \leavevmode
\begin{enumerate}
\item If $\delta(p) \notin \left\{p^{-i}, i \geq 0 \right\}$, or if $\delta(p) = p^{-i}$ for some $N \geq i \geq 0$ and $\kappa \neq -i$, then $H^1([a^+] \colon H^1(\mathscr{C}_{u^-, \varphi})) = 0$.
\item If $\delta(p) = 1$ and $\kappa = 0$, then $H^1([a^+] \colon H^1(\mathscr{C}_{u^-, \varphi})) = L \cdot [t^0]$.
\item If $\delta(p) = p^{-i}$ for some $i \geq 1$ and $\kappa = - i$ then $H^1([a^+] \colon H^1(\mathscr{C}_{u^-, \varphi}))$ lives in an exact sequence
\[ 0 \to L\cdot [t^i] \to H^1([a^+] \colon H^1(\mathscr{C}_{u^-, \varphi})) \to L \cdot [t^{i-1}] \to 0. \]
\end{enumerate}
\end{lemm}


We conclude

\begin{prop} \label{Lie2+} \leavevmode
\begin{enumerate}
\item If $\delta(p) \notin \left\{p^{-i}, i \geq 0 \right\}$, or if $\delta(p) = p^{-i}$ for some $N \geq i \geq 0$ and $\kappa \neq -i$, then $H^2(\mathscr{C}_{u^-, \varphi, a^+}) = 0$.
\item If $\delta(p) = 1$ and $\kappa = 0$, then $H^2(\mathscr{C}_{u^-, \varphi, a^+}) = L \cdot [t^0]$.
\item If $\delta(p) = p^{-i}$ for some $i \geq 1$ and $\kappa = - i$ then $H^2(\mathscr{C}_{u^-, \varphi, a^+})$ lives in an exact sequence
\[ 0 \to Z \to H^2(\mathscr{C}_{u^-, \varphi, a^+}) \to L \cdot [t^{i - 1}], \]
\[ 0 \to L\cdot [t^i] \to Z \to L \cdot [t^{i-1}] \to 0. \]
\end{enumerate}
\end{prop}


\subsubsection{Calculation of $H^3(\mathscr{C}_{u^-, \varphi, a^+})$:}

\begin{prop} \label{Lie3+} \leavevmode
\begin{enumerate}
\item If $\kappa = - i$, $i \geq 1$, and $\delta(p) = p^{-i}$, then $H^3(\mathscr{C}_{u^+, \varphi, a^+}) = L \cdot [t^{i-1}]$.
\item Otherwise $H^3(\mathscr{C}_{u^+, \varphi, a^+}) = 0$.
\end{enumerate}
\end{prop}

\begin{proof}
Immediate from Lemma \ref{Lie2tmp}.
\end{proof}

\subsection{The $\overline{P}^+$-cohomology of $\Robba^+(\delta_1, \delta_2)$}

We take $\tilde{P}$-invariants to calculate group cohomology from the Lie algebra cohomology.

\subsubsection{Calculation of $H^0(\overline{P}^+, \Robba^+(\delta_1, \delta_2))$ }

\begin{lemm} \label{cohom0+} \leavevmode
\begin{enumerate} 
\item If $\delta_1 \delta_2^{-1} = x^{-i}$, $i \geq 0$, then $H^0(\mathscr{C}_{u^-, \varphi, a^+}) = L \cdot t^i$.
\item Otherwise $H^0(\mathscr{C}_{u^-, \varphi, a^+}) = 0$.
\end{enumerate}
\end{lemm}

\begin{proof}
This follows by taking $\tilde{P}$-invariants on Lemma \ref{Lie0+} and observing that $\tau$ fixes $t^i$.
\end{proof}

\subsubsection{Calculation of $H^1(\overline{P}^+, \Robba^+(\delta_1, \delta_2))$ }

\begin{lemm} \label{cohom1+} \label{lem:karat}
If $M = \Robba^+(\delta_1, \delta_2)$ then 
\begin{enumerate}
\item If $\delta_1\delta_{2}^{-1} \notin \left\{x^{-i} \text{ } | \text{ } i\geq 0 \right\}$ then $H^{1}(\overline{P}^{+}, M) = 0$.
\item If $\delta_1\delta_{2}^{-1} = x^{-i}$ for some $i \geq 0$ then $H^{1}(\overline{P}^{+}, M)$ is of dimension $2$.
\end{enumerate}
\end{lemm}

\begin{proof} \leavevmode
To calculate the fixed points by $\tilde{P}$ on the Lie algebra cohomology after all the identifications we have made, recall that we have rendered these actions explicit in Remark \ref{actioncohom}.
\begin{enumerate}
\item Suppose first that we are in the situation of Proposition \ref{H1LiePR+}(1). In this case $H^1(\mathscr{C}_{u^-, \varphi, a^+}) = 0$ and the result is immediate.

Now assume we are in the case of Proposition \ref{H1LiePR+}(2), then we have
\[ 0 \to L \cdot [t^0] \to H^1(\mathscr{C}_{u^-, \varphi, a^+}) \to L \cdot [t^0] \to 0.\] Note also that, by Remark \ref{actioncohom}, $A^0$ acts on each term on the extremities via multiplication by $\delta_1 \delta_2^{-1}$ (which is not trivial by hypothesis). We conclude by taking $A^0$ invariants of this sequence.

Finally, suppose that the hypothesis of Proposition \ref{H1LiePR+}(3) are satisfied. Then we have exact sequences
\[ 0 \to L \cdot [t^i] \to H^1(\mathscr{C}_{u^-, \varphi, a^+}) \to Z_i \to 0, \]
\[ 0 \to L \cdot [t^i] \to Z_i \to L \cdot [t^{i - 1}] \to 0. \]  Again, by Remark \ref{actioncohom}, $A^0$ acts on the first term of the two SES's as multiplication by $\delta_1 \delta_2^{-1} |_\zpe \neq \chi^{-i}$ and on the second term of the second SES via multiplication by $\chi \delta_1 \delta_2^{-1} |_\zpe \neq \chi^{-i + 1}$. The result also follows by taking $A^0$-invariants.

\item If $\delta_1 \delta_2^{-1} = \mathbf{1}_\qpe$, then we are in the situation of Proposition \ref{H1LiePR+}(2), and the result follows since everything is $\tilde{P}$-invariant.

If $\delta_1 \delta_2^{-1} = x^{-i}$ for some $i \geq 1$, we are in the situation of Proposition \ref{H1LiePR+}(3). Note that, by Remark \ref{actioncohom}, $\tau$ acts on each of the extremities of the long exact sequences of $\ref{H1LiePR+}(3)$ just as $\tau$. In this case, the term $t^{i - 1}$ of the second short exact sequence is not fixed by the action of $\tau$, so taking $\tilde{P}$-invariants in the short exact sequences gives the result.
\end{enumerate}
\end{proof}

\subsubsection{Calculation of $H^2(\overline{P}^+, \Robba^+(\delta_1, \delta_2))$ }

\begin{lemm} \label{cohom2+} \leavevmode
\begin{enumerate}
\item If $\delta_1 \delta_2^{-1} \notin \{x^{-i}, i \geq 0 \}$, then $H^2(\overline{P}^+, M) = 0$.
\item If $\delta_1 \delta_2^{-1} = \mathbf{1}_\qpe$ then $H^2(\overline{P}^+, M)$ is of dimension one naturally isomorphic to $L \cdot [t^0]$.
\item If $\delta_1 \delta_2^{-1} = x^{-i}$, $i \geq 1$, then $H^2(\overline{P}^+, M)$ is of dimension $3$.
\end{enumerate}
\end{lemm}

\begin{proof}
This follows by taking $\overline{P}^+$-invariants in Lemma \ref{Lie2+}.
\end{proof}

\subsubsection{Calculation of $H^3(\overline{P}^+, \Robba^+(\delta_1, \delta_2))$ }

\begin{lemm} \label{cohom3+} \leavevmode
\begin{enumerate}
\item If $\delta_1 \delta_2^{-1} = x^{-i}$, $i \geq 1$, then $H^3(\overline{P}^+, M)$ is of dimension $1$ naturally generated by $[t^{i-1}]$.
\item Otherwise $H^3(\overline{P}^+, M) = 0$.
\end{enumerate}
\end{lemm}

\begin{proof}
Immediate from Lemma \ref{Lie3+}.
\end{proof}

\subsection{The $\overline{P}^+$-cohomology of $\Robba(\delta_1, \delta_2)$}

Call, for $* \in \{ +, -, \emptyset\}$, $M_*$ the module $\Robba^*(\delta_1, \delta_2)$. We use the short exact sequence
\begin{equation} \label{shorteaea} 0 \to M_+ \to M \to M_- \to 0 \end{equation}
in order to calculate the $\overline{P}^+$-cohomology of $M$. We restate one of the propositions announced at the beginning of this section.

\begin{prop} Let $M = \Robba(\delta_1, \delta_2)$. Then
\begin{enumerate}
\item If $\delta_1 \delta_2^{-1} \notin \{ x^{-i}, i \in \N \} \cup \{ \chi x^i, i \in \N \}$, then $\dim_L H^j(\overline{P}^+, M) = 0, 1, 1, 0$, for $j = 0, 1, 2, 3$.
\item If $\delta_1 \delta_2^{-1} = \mathbf{1}_\qpe$, then $\dim_L H^j(\overline{P}^+, M) = 1, 2, 2, 0$, for $j = 0, 1, 2, 3$.
\item If $\delta_1 \delta_2^{-1} = x^{-i}, i \in \N$, then $\dim_L H^j(\overline{P}^+, M) = 1, 3, 2, 0$, for $j = 0, 1, 2, 3$.
\item If $\delta_1 \delta_2^{-1} =\chi x^i, i \in \N$, then $\dim_L H^j(\overline{P}^+, M) = 0, 2, 2, 1$, for $j = 0, 1, 2, 3$.
\end{enumerate}
\end{prop}

\begin{proof}
For a $\overline{P}^+$-module $N$ we note, for simplicity $H^i(N) = H^i(\overline{P}^+, N)$. Consider the long exact sequence on cohomology associated to the sequence of Equation \ref{shorteaea}
\begin{align*} 0 &\to H^0(M_+) \to H^0(M) \to H^0(M_-) \to H^1(M_+) \to H^1(M) \to H^1(M_-) \\
& \to H^2(M_+) \to H^2(M) \to H^2(M_-) \to H^3(M_+) \to H^3(M) \to H^3(M_-) \to 0. \end{align*} Recall that we have already calculated (cf. Lemmas \ref{cohom0-}, \ref{cohom1-}, \ref{cohom2-}, \ref{cohom3-}, \ref{cohom0+}, , \ref{cohom1+}, \ref{cohom2+}, \ref{cohom3+})all of the $H^j(M_\pm)$.

If $\delta_1 \delta_2^{-1} \notin  \{ x^{-i}, i \in \N \}$, then $H^j(M_+)$ for all $j$, which implies that $H^j(M) = H^j(M_-)$.

If $\delta_1 \delta_2^{-1} = x^{-i}$, $i \in \N$, then $H^0(M_-) = H^3(M_-) = 0$ and the map $H^1(M) \to H^1(M_-)$ is the zero map (cf. \cite[Corollaire 5.23(ii)]{colmez2015}, which is independent of \cite[Lemme 5.21]{colmez2015}). We hence have $H^1(M) = H^1(M_+)$, and an exact sequence
\[ 0 \to H^1(M_-) \to H^2(M_+) \to H^2(M) \to H^2(M_-) \to H^3(M_+) \to H^3(M) \to 0. \] By the same arguments as the ones in the proof of \cite[Th\'eor\`eme 5.16]{colmez2015}, we can show that in this case the map $H^3(M_+) \to H^3(M)$ is the zero map, hence we deduce $H^3(M) = 0$ and $dim_L H^2(M) = 2$, which completes the proof.
\end{proof}

\section{A relative cohomology isomorphism} \label{sec:extgrpact}

\begin{defi}\label{def:qausreg}
We call a character $\delta: \qpe \to A^{\times}$ regular if pointwise (meaning the reduction for each maximal ideal $\mathfrak{m} \subset A$) it is never of the form $\chi x^{i}$ or $x^{-i}$ for some $i \geq 0$. 
\end{defi}

\begin{rema}\label{rem:qausreg}
By \cite[Corollaire 2.11]{chen2013}, if $\delta$ is regular then $H^{2}(A^+, \mathscr{R}_{A}(\delta)) = 0$. Moreover in the setting of a point ($A$ is a finite extension of $\qp$), $\delta_1\delta_2^{-1}: \qpe \to L^\times$ is regular implies the pair $(\delta_1,\delta_2)$ is generic in the sense of \cite{colmez2015}. 
\end{rema}

The following is a relative version of \cite[Proposition 5.18]{colmez2015}. 

\begin{prop} \label{3.25}
Suppose $A$ is reduced. Let $\delta_{1}$, $\delta_{2}: \qpe \rightarrow A^{\times}$ such that $\delta_{1}\delta_{2}^{-1}$ is regular. Then the restriction morphism from $\overline{P}^{+}$ to $A^+$, induces a surjection:
$$H^{1}(\overline{P}^{+}, \mathscr{R}_{A}(\delta_1, \delta_2)) \rightarrow H^{1}(A^+, \mathscr{R}_{A}(\delta_{1}\delta_{2}^{-1})).$$
\end{prop}

\begin{proof}
We work at the derived level. For the sake of brevity let $C^{\bullet}_{\overline{P}^+}$ denote the Koszul complex $\mathscr{C}_{\tau, \varphi, \gamma}$ of \S \ref{sec:koscomcoh}. Similarly let $C^{\bullet}_{A^+}$ denote the complex $\mathscr{C}_{\varphi, \gamma}$. We have a canonical morphism:
$$C^{\bullet}_{\overline{P}^{+}}(\mathscr{R}_{A}(\delta_1, \delta_2)) \rightarrow C^{\bullet}_{A^+}(\mathscr{R}_{A}(\delta_{1}\delta_{2}^{-1}))$$
in $\mathcal{D}^{-}_{\mathrm{pc}}(A)$. Let 
$$C^{\bullet}:= \mathrm{Cone}(C^{\bullet}_{\overline{P}^{+}}(\mathscr{R}_{A}(\delta_1, \delta_2)) \rightarrow C^{\bullet}_{A^+}(\mathscr{R}_{A}(\delta_{1}\delta_{2}^{-1})))$$
and note that $C^{\bullet} \in \mathcal{D}^{-}_{\mathrm{pc}}(A)$ by Remark \ref{3.5}. The distinguished triangle
$$C^{\bullet}_{\overline{P}^{+}}(\mathscr{R}_{A}(\delta_1, \delta_2)) \rightarrow C^{\bullet}_{A^+}(\mathscr{R}_{A}(\delta_{1}\delta_{2}^{-1})) \rightarrow C^{\bullet}$$
induces a long exact sequence in cohomology
$$\cdots \rightarrow H^{1}(C^{\bullet}_{\overline{P}^{+}}(\mathscr{R}_{A}(\delta_1, \delta_2))) \rightarrow H^{1}(C^{\bullet}_{A^+}(\mathscr{R}_{A}(\delta_{1}\delta_{2}^{-1}))) \rightarrow H^{1}(C^{\bullet}) \rightarrow \cdots$$
Moreover since $\mathscr{R}_{A}(\delta_1, \delta_2)$ is a flat $A$-module and $\mathscr{R}_{A}(\delta_1, \delta_2) \otimes_{A} A/\mathfrak{m} \cong \mathscr{R}_{A/\mathfrak{m}}(\delta_1, \delta_2)$ for any maximal ideal $\mathfrak{m} \subset A$, it follows that
$$C^{\bullet}_{\overline{P}^{+}}(\mathscr{R}_{A}(\delta_1, \delta_2)) \otimes^{\mathbf{L}} A/\mathfrak{m} \cong C^{\bullet}_{\overline{P}^{+}}(\mathscr{R}_{A/\mathfrak{m}}(\delta_1, \delta_2)).$$
Similarly we have
$$C^{\bullet}_{A^+}(\mathscr{R}_{A}(\delta_{1}\delta_{2}^{-1})) \otimes^{\mathbf{L}} A/\mathfrak{m} \cong C^{\bullet}_{A^+}(\mathscr{R}_{A/\mathfrak{m}}(\delta_{1}\delta_{2}^{-1})).$$
Hence the morphism $A \rightarrow A/\mathfrak{m}$ induces a morphism of distinguished triangles
which by the functoriality of the truncation operators gives a morphism of long exact sequences
$$
\begin{tikzcd} [row sep = large, column sep = small]
\cdots \arrow[r] &
H^{1}(C^{\bullet}_{\overline{P}^{+}}(\mathscr{R}_{A}(\delta_1, \delta_2))) \arrow[r] \arrow[d] &
H^{1}(C^{\bullet}_{A^+}(\mathscr{R}_{A}(\delta_{1}\delta_{2}^{-1}))) \arrow[r, "\gamma"] \arrow[d] &
H^{1}(C^{\bullet}) \arrow[d] \arrow[r, "\gamma_1"] &
\cdots \\
\cdots \arrow[r] &
H^{1}(C^{\bullet}_{\overline{P}^{+}}(\mathscr{R}_{A/\mathfrak{m}}(\delta_1, \delta_2))) \arrow[r, "\alpha"] &
H^{1}(C^{\bullet}_{A^+}(\mathscr{R}_{A/\mathfrak{m}}(\delta_{1}\delta_{2}^{-1}))) \arrow[r, "\beta"] &
H^{1}(C^{\bullet} \otimes^{\mathbf{L}} A/\mathfrak{m}) \arrow[r, "\beta_1"] &
\cdots \end{tikzcd}
$$

By \cite[Proposition 5.18]{colmez2015} (see also Proposition \ref{cohomfinal}), $\alpha$ is an isomorphism and so $\beta$ is the zero morphism. We claim that $\gamma$ is the zero morphism as well. To do this, we take advantage of the spectral sequence:
$$\Tor_{-p}(H^{q}(C^{\bullet}), A/\mathfrak{m}) \implies H^{p+q}(C^{\bullet} \otimes^{\mathbf{L}} A/\mathfrak{m}),$$
whose 2nd page takes the form

$$
\begin{tikzcd} [row sep = large, column sep = large] 
0 \arrow[rrd] & 
0 &  
0 \\
\Tor_{2}(H^{2}(C^{\bullet}),A/\mathfrak{m}) \arrow[rrd] &
\Tor_{1}(H^{2}(C^{\bullet}),A/\mathfrak{m}) &
H^{2}(C^{\bullet}) \otimes_{A} A/\mathfrak{m} \\
\Tor_{2}(H^{1}(C^{\bullet}),A/\mathfrak{m}) \arrow[rrd] &
\Tor_{1}(H^{1}(C^{\bullet}),A/\mathfrak{m}) &
H^{1}(C^{\bullet}) \otimes_{A} A/\mathfrak{m} \\
\Tor_{2}(H^{0}(C^{\bullet}),A/\mathfrak{m}) &
\Tor_{1}(H^{0}(C^{\bullet}),A/\mathfrak{m}) &
H^{0}(C^{\bullet}) \otimes_{A} A/\mathfrak{m}
\end{tikzcd}
$$

The long exact sequence in cohomology
$$\cdots \rightarrow H^{3}(C^{\bullet}_{\overline{P}^{+}}(\mathscr{R}_{A}(\delta_1, \delta_2))) \rightarrow H^{3}(C^{\bullet}_{A^+}(\mathscr{R}_{A}(\delta_{1}\delta_{2}^{-1}))) = 0 \rightarrow H^{3}(C^{\bullet}) \rightarrow 0 \rightarrow \cdots$$
implies that $H^{3}(C^{\bullet}) = 0$ hence explaining the top row. Moreover since $\delta_{1}\delta_{2}^{-1}$ is regular, $H^{2}(C^{\bullet}_{A^+}(\mathscr{R}_{A}(\delta_{1}\delta_{2}^{-1}))) = 0$ by \cite[Corollaire 2.11]{chen2013} (see also Remark \ref{rem:recchenres}). By Proposition \ref{cohomfinal}, $H^{3}(C^{\bullet}_{\overline{P}^{+}}(\mathscr{R}_{A}(\delta_1, \delta_2))) = 0$ and thus from the long exact sequence
$$\cdots \rightarrow 0 \rightarrow H^{2}(C^{\bullet}) \rightarrow  H^{3}(C^{\bullet}_{\overline{P}^{+}}(\mathscr{R}_{A}(\delta_1, \delta_2))) = 0 \rightarrow 0 \rightarrow \cdots$$
we deduce that $H^{2}(C^{\bullet}) = 0$. Hence the spectral sequence degenerates at the 2nd page in degree $1$ cohomology and so $H^{1}(C^{\bullet} \otimes^{\mathbf{L}} A/\mathfrak{m}) = H^{1}(C^{\bullet}) \otimes_{A} A/\mathfrak{m}$. Similarly the spectral sequence
$$\Tor_{-p}(H^{q}(C^{\bullet}_{\overline{P}^{+}}(\mathscr{R}_{A}(\delta_{1}\delta_{2}^{-1}))), A/\mathfrak{m}) \implies H^{p+q}(C^{\bullet}_{\overline{P}^{+}}(\mathscr{R}_{A/\mathfrak{m}}(\delta_{1}\delta_{2}^{-1})))$$
implies $H^{2}(C^{\bullet}_{\overline{P}^{+}}(\mathscr{R}_{A/\mathfrak{m}}(\delta_{1}\delta_{2}^{-1}))) = H^{2}(C^{\bullet}_{\overline{P}^{+}}(\mathscr{R}_{A}(\delta_{1}\delta_{2}^{-1}))) \otimes_{A} A/\mathfrak{m}$. Thus in the diagram
$$
\begin{tikzcd} [row sep = large, column sep = small]
\cdots \arrow[r] &
H^{1}(C^{\bullet}_{A^+}(\mathscr{R}_{A}(\delta_{1}\delta_{2}^{-1}))) \arrow[r, "\gamma"] \arrow[d] &
H^{1}(C^{\bullet}) \arrow[d] \arrow[r, "\gamma_1"] &
H^{2}(C^{\bullet}_{\overline{P}^{+}}(\mathscr{R}_{A}(\delta_1, \delta_2))) \arrow[r] \arrow[d] &
0 \\
\cdots \arrow[r] &
H^{1}(C^{\bullet}_{A^+}(\mathscr{R}_{A/\mathfrak{m}}(\delta_{1}\delta_{2}^{-1}))) \arrow[r, "\beta"] &
H^{1}(C^{\bullet} \otimes^{\mathbf{L}} A/\mathfrak{m}) \arrow[r, "\beta_1"] &
H^{2}(C^{\bullet}_{\overline{P}^{+}}(\mathscr{R}_{A/\mathfrak{m}}(\delta_1, \delta_2))) \arrow[r] &
0 
\end{tikzcd}
$$
we have $\beta_{1} = \gamma_{1} \otimes A/\mathfrak{m}$. By Proposition \ref{cohomfinal2}, $\beta_1$ is an isomorphism of dimension 1 vector spaces over $A/\mathfrak{m}$ for every $\mathfrak{m} \subset A$. Thus $\gamma_{1}$ is a surjective morphism of locally free $A$-modules, cf. \cite[Lemma 2.1.8(1)]{kedlaya2014}, which are locally of dimension 1. Hence it is an isomorphism and so $\gamma$ is the zero morphism, as desired. This completes the proof.   
\end{proof}

\begin{prop} \label{3.26}
Suppose $A$ is reduced. Let $\delta_{1}$, $\delta_{2}: \qpe \rightarrow A^{\times}$ such that $\delta_{1}\delta_{2}^{-1}$ is regular. Then the restriction morphism from $\overline{P}^{+}$ to $A^+$, induces an injection:
$$H^{1}(\overline{P}^{+}, \mathscr{R}_{A}(\delta_1, \delta_2)) \rightarrow H^{1}(A^+, \mathscr{R}_{A}(\delta_{1}\delta_{2}^{-1})).$$
\end{prop}

\begin{proof}
Keeping the notation used in the proof of Proposition \ref{3.25}, since $H^{1}(C^{\bullet})$ is locally free, $\Tor_{1}(H^{1}(C^{\bullet}),A/\mathfrak{m}) = 0$. But the spectral sequence 
$$\Tor_{-p}(H^{q}(C^{\bullet}), A/\mathfrak{m}) \implies H^{p+q}(C^{\bullet} \otimes^{\mathbf{L}} A/\mathfrak{m})$$
abuts to 0 in degree 0 as $\alpha$ is an isomorphism. This implies that $H^{0}(C^{\bullet}) \otimes A/\mathfrak{m} = 0$ for every maximal ideal $m \subset A$. By Nakayama's Lemma it follows that $H^{0}(C^{\bullet})$ is 0 and we deduce the result.  
\end{proof}

\begin{theo} \label{3.27}
Suppose $A$ is reduced. Let $\delta_{1}$, $\delta_{2}: \qpe \rightarrow A^{\times}$ such that $\delta_{1}\delta_{2}^{-1}$ is regular. Then the restriction morphism from $\overline{P}^{+}$ to $A^+$, induces an isomorphism:
$$H^{1}(\overline{P}^{+}, \mathscr{R}_{A}(\delta_1, \delta_2)) \rightarrow H^{1}(A^+, \mathscr{R}_{A}(\delta_{1}\delta_{2}^{-1})).$$
\end{theo}

\begin{proof}
This is a consequence of Propositions \ref{3.25} and \ref{3.26}.
\end{proof}

We are now ready to handle the case when $A$ is non-reduced. We begin by proving a slightly enhanced version of Theorem \ref{3.27}.  

\begin{prop} \label{3.28}
Suppose $A$ is reduced and $M$ a finite $A$-module (equipped with trivial $\overline{P}^{+}$-action). Let $\delta_{1}$, $\delta_{2}: \qpe \rightarrow A^{\times}$ such that $\delta_{1}\delta_{2}^{-1}$ is regular. Then the restriction morphism from $\overline{P}^{+}$ to $A^+$, induces an isomorphism:
$$H^{1}(\overline{P}^{+}, \mathscr{R}_{A}(\delta_1, \delta_2) \otimes_{A} M) \rightarrow H^{1}(A^+, \mathscr{R}_{A}(\delta_{1}\delta_{2}^{-1}) \otimes_{A} M).$$
\end{prop}

\begin{proof}
To prove surjectivity, we follow the proof of Proposition \ref{3.25}. In fact the only thing that needs to be checked is that 
$$H^{1}(\overline{P}^{+}, \mathscr{R}_{A/\mathfrak{m}}(\delta_1, \delta_2) \otimes_{A} M) \rightarrow H^{1}(A^+, \mathscr{R}_{A/\mathfrak{m}}(\delta_{1}\delta_{2}^{-1}) \otimes_{A} M)$$
is an isomorphism.
Denote by $M':= M \otimes_A A/\mathfrak{m}$, so that the above is equivalent to showing
\begin{equation} \label{cold}
H^{1}(\overline{P}^{+}, \mathscr{R}_{A/\mathfrak{m}}(\delta_1, \delta_2) \otimes_{A/\mathfrak{m}} M') \rightarrow H^{1}(A^+, \mathscr{R}_{A/\mathfrak{m}}(\delta_{1}\delta_{2}^{-1}) \otimes_{A/\mathfrak{m}} M'),
\end{equation} 
is an isomorphism. Now $M'$ is flat over $A/\mathfrak{m}$ and so by the Tor-spectral sequence 
$$ H^{1}(\overline{P}^{+}, \mathscr{R}_{A/\mathfrak{m}}(\delta_1, \delta_2) \otimes_{A/\mathfrak{m}} M') \cong H^{1}(\overline{P}^{+}, \mathscr{R}_{A/\mathfrak{m}}(\delta_1, \delta_2)) \otimes_{A/\mathfrak{m}} M'$$
and similarly
$$ H^{1}(A^+, \mathscr{R}_{A/\mathfrak{m}}(\delta_1, \delta_2) \otimes_{A/\mathfrak{m}} M') \cong H^{1}(A^+, \mathscr{R}_{A/\mathfrak{m}}(\delta_1, \delta_2)) \otimes_{A/\mathfrak{m}} M'.$$
Thus the morphism \ref{cold} is an isomorphism by \cite[Proposition 5.18]{colmez2015}. The rest of the proof of Proposition \ref{3.25} goes through with $C^{\bullet}$ replaced by $C^{\bullet} \otimes^{\mathbf{L}} M$. For injectivity the proof of Proposition \ref{3.26} remains unchanged except with $C^{\bullet}$ replaced by $C^{\bullet} \otimes^{\mathbf{L}} M$. This completes the proof.   
\end{proof}

We now need a lemma that guarantees the connection morphisms are $0$ in a certain long exact sequence.

\begin{lemm} \label{3.29}
Let $A$ be an $\qp$-affinoid algebra and $I \subset A$ an ideal of $A$. Let $\delta_{1}$, $\delta_{2}: \qpe \rightarrow A^{\times}$ such that $\delta_{1}\delta_{2}^{-1}$ is regular. The short exact sequence
\[ 0 \to I \to A \to A / I \to 0 \]
induces an injective morphism
$$H^2(\overline{P}^+,\mathscr{R}_{A}(\delta_1,\delta_2) \otimes_A I) \to H^2(\overline{P}^+,\mathscr{R}_{A}(\delta_1,\delta_2)). $$
\end{lemm} 

\begin{proof}
By Lemma \ref{lielem}, it suffices to show that the morphism
$$H_{\Lie}^{2}(\overline{P}^+, \mathscr{R}_{A}(\delta_1,\delta_2) \otimes_A I) \rightarrow H_{\Lie}^{2}(\overline{P}^{+}, \mathscr{R}_{A}(\delta_1,\delta_2))$$
coming from the long exact sequence in Lie algebra cohomology of $\overline{P}^{+}$ is injective. Recall the Lie algebra complex from \S \ref{sublieal}:
\[
\mathscr{C}_{u^{-},\varphi,a^{+}}: 0 \rightarrow M \xrightarrow{A} M \oplus M \oplus M \xrightarrow{B} M \oplus M \oplus M \xrightarrow{C} M \rightarrow 0
\]
where
\begin{align*}
A(x) &= ((\varphi-1)x,a^{+}x,u^{-}x) \\
B(x,y,z) &= (a^{+}x - (\varphi-1)y, u^{-}y - (a^{+} +1)z, (p\varphi-1)z - u^{-}x) \\
C(x,y,z) &= u^{-}x + (p\varphi -1)y + (a^{+} + 1)z. 
\end{align*}
It suffices to show that if $(x,y,z) \in \Robba_A^{\oplus 3}$ such that $B(x,y,z) \in \Robba_A^{\oplus 3} \otimes_A I$, then $(x,y,z) \in \Robba_A^{\oplus 3} \otimes_A I$. Since
$$H_{\Lie}^{2}(A^+, \mathscr{R}_{A}(\delta_1,\delta_2) \otimes_A I) \rightarrow H_{\Lie}^{2}(A^{+}, \mathscr{R}_{A}(\delta_1,\delta_2))$$
is injective (this follows from the fact that $H^2_{\Lie}(A^+, \mathscr{R}_{L}(\delta_1,\delta_2))=0$ for every finite extension $L$ of $\qp$), it follows that $(x,y) \in \Robba_A^{\oplus 2} \otimes_A I$. Thus the problem is the following: $f \in \Robba_A$ such that
\[
(a^+ +1)f \in \Robba_A \otimes_A I\text{ and } (p\varphi-1)f \in \Robba_A \otimes_A I
\] 
and one must show that $f \in \Robba_A \otimes_A I$. 


Call $F = p \delta(p) \varphi - 1$ and denote by $\Robba_I := \Robba_A \otimes_A I$. We will show that, if $f \in \Robba_A$ is such that $F(f) \in \Robba_I$, then $f \in \Robba_I$. We first observe that this statement is true for $\Robba_A^\pm$. Indeed: 
\begin{itemize}
\item If $\phi \in \mathrm{LA}(\zp, A)$ is such that $F(\phi) \in \mathrm{LA}(\zp, I)$, then call $\overline{\phi} \in \mathrm{LA}(\zp, A/I)$ the reduction of $\phi$ modulo $I$. We have then $F(\overline{\phi}) = 0$, but $F$ is injective on $\mathrm{LA}(\zp, A/I)$, so $\overline{\phi} = 0$, which translates into $\phi \in \mathrm{LA}(\zp, I)$.
\item The claim that, for $\mu \in \Robba_A^+$, then $F(\mu) \in \Robba_I^+$ implies $\mu \in \Robba_I^+$ follows from a direct calculation by looking at the coefficients of the power series expression of $\mu$.
\end{itemize}

Consider now the following commutative diagram given by the residue map $f = \sum_{n \in \Z} a_n T^n \mapsto [x \mapsto \phi_f(x) = \text{res}_0((1 + T)^{-x} f(T) \frac{dT}{1 + T}) = \sum_{n \geq 0} a_{-n-1} {-x - 1 \choose n}]$:
\[
\begin{tikzcd} [row sep = large, column sep = small]
0 \arrow[r] &
\Robba_A^+ \arrow[r] \arrow["F", d] &
\Robba_A \arrow[r] \arrow[d, "F"] &
\mathrm{LA}(\zp, A) \arrow[d, "F"] \arrow[r] &
0  \\
0 \arrow[r]&
\Robba_A^+ \arrow[r] &
\Robba_A \arrow[r] &
\mathrm{LA}(\zp, A) \arrow[r] &
0
\end{tikzcd}
\]
Let $f = \sum_{n \in \Z} a_n T^n \in \Robba_A$ be such that $F(f) \in \Robba_I$. Then we have $F(\phi_f) \in \mathrm{LA}(\zp, I)$ and this implies by the first point of the last paragraph that $\phi_f \in \mathrm{LA}(\zp, I)$ and hence $a_n \in I$ for all $n < 0$. Hence, if we write $f = f^- + f^+$, where $f^- = \sum_{n < 0} a_n T^n$ and $f^+ = \sum_{n \geq 0} a_n T^n$, we get that $F(f^-) \in \Robba_I$ (because $f^- \in \Robba_I$!) and hence $F(f^+) = F(f) - F(f^-) \in \Robba_I$, which implies that $f^+ \in \Robba_I$ and allows us to conclude that $f \in \Robba_I$. This completes the proof.
%

\end{proof}

\begin{rema}
The statement of Lemma \ref{3.29} is not surprising as $\overline{P}^{+}$ acts trivially on the coefficient algebra $A$ in $\mathscr{R}_{A}(\delta_1, \delta_2)$.  
\end{rema} 

\begin{theo} \label{3.31}
Let $\delta_{1}$, $\delta_{2}: \qpe \rightarrow A^{\times}$ such that $\delta_{1}\delta_{2}^{-1}$ is regular. Then the restriction morphism from $\overline{P}^{+}$ to $A^+$, induces an isomorphism:
$$H^{1}(\overline{P}^{+}, \mathscr{R}_{A}(\delta_1, \delta_2)) \rightarrow H^{1}(A^+, \mathscr{R}_{A}(\delta_{1}\delta_{2}^{-1})).$$
\end{theo}

\begin{proof}
Since $A$ is in particular noetherian its nilradical is nilpotent. Thus it is natural to proceed via induction on the index of nilpotence $i \geq 0$. The base case $i = 0$ (meaning $A$ is reduced) is Theorem \ref{3.27}. Suppose by induction the result is true for index $i$ and suppose now $N^{i+1}=0$. For the sake of brevity denote $X_{N^i}:=\mathscr{R}_{A}(\delta_1,\delta_2) \otimes_A N^i$, $X_{A}:=\mathscr{R}_{A}(\delta_1,\delta_2)$ and $X_{A/N^i}:=\mathscr{R}_{A}(\delta_1,\delta_2) \otimes_A A / N^i$. The short exact sequence
$$0 \to X_{N^i} \to X_{A} \to X_{A/N^i} \rightarrow 0$$
gives a commutative diagram
$$
\begin{tikzcd} [row sep = large, column sep = small]
H^{0}(\overline{P}^+, X_{A/N^i}) \arrow[r, "\beta'"] \arrow[d] &
H^{1}(\overline{P}^+, X_{N^i}) \arrow[r] \arrow[d, "\alpha_1"] &
H^{1}(\overline{P}^+, X_{A}) \arrow[d, "\alpha_2"] \arrow[r] &
H^{1}(\overline{P}^+, X_{A/N^i}) \arrow[r, "\rho"] \arrow[d, "\alpha_3"] &
H^{2}(\overline{P}^+, X_{N^i})  \\
H^{0}(A^+, X_{A/N^i}) \arrow[r, "\beta"]&
H^{1}(A^+, X_{N^i}) \arrow[r] &
H^{1}(A^+, X_{A}) \arrow[r] &
H^{1}(A^+, X_{A/N^i}) \arrow[r] &
0. 
\end{tikzcd}
$$
The two rows come from long exact sequences in cohomology and commutativity comes from functoriality of the restriction morphism $H^{i}(\overline{P}^+, -) \rightarrow H^{i}(A^+, -)$. By Lemma \ref{3.29}, the connecting morphism $\rho$ is the zero morphism. Identifying $X_{A/N^i}$ with $\mathscr{R}_{A/N^i}(\delta_1,\delta_2)$ and $X_{N^i}$ with $\mathscr{R}_{A/N}(\delta_1,\delta_2) \otimes_{A/N} N^{i}$; we see that $\alpha_1$ is an isomorphism by Proposition \ref{3.28} and $\alpha_3$ is an isomorphism by the inductive step. On the other hand by \cite[Proposition 2.10]{chen2013} the morphism 
\[
H^{0}(A^+, X_A) \to H^{0}(A^+, X_{A/N^i})
\]
is surjective. Thus $\beta$ is the zero morphism. By the commutativity of the first square, this implies that $\beta'$ is the zero morphism. Thus by the 5-Lemma, $\alpha_2$ is an isomorphism and this proves the result.    
\end{proof}

\section{Construction of the correspondence} \label{sec:cooftcor}

In this section we construct, following \cite[Chapter 6]{colmez2015}, the correspondence $\Delta \mapsto \Pi(\Delta)$ for a regular $(\varphi, \Gamma)$-module $\Delta$ over the relative Robba ring $\Robba_A$, interpolating the analogous construction of loc. cit. at the level of points. The construction is inspired from the calculation of the locally analytic vectors in the unitary principal series case (corresponding to the case when $\Delta$ is trianguline and \'etale), cf. \cite{col2008}, \cite{colmezunit}, \cite{colmezunitlat} and involves a detailed study of the Jordan-H\"older components of the $G$-module $\Delta \boxtimes_\omega \P^1$. We will see that the sought-after representation $\Pi(\Delta)$ is cut out from these constituents. 

There are a couple of differences between the approach taken in \cite{colmez2015} and the one taken in the present paper that merits to be pointed out. On the one hand, in constrast to loc. cit. where the construction is carried out more generally for the case of a Lubin-Tate $(\varphi, \Gamma)$-module, we work in the usual cyclotomic context, which simplifies many of the proofs and steps. This is mainly due to the rich structure of the $\Robba_A(\Gamma)$-module $\Delta^{\psi = 0}$ for $\Delta \in \Phi \Gamma(\Robba_A)$. Secondly, we are only able to carry out the construction for \textit{regular} trianguline $(\varphi, \Gamma)$-modules (cf. Definition \ref{def:qaregpmd}). Indeed it is only for those objects that our cohomology comparison theorem of the last chapter works. The last modification between our approach and the one taken in loc. cit. is found in the argument showing that the middle extension of $B_A(\delta_1, \delta_2)^* \otimes \omega$ by $B_A(\delta_1, \delta_2)$ in the proof of Theorem \ref{thm:quas} splits, where we give a more direct method.

\subsection{The main result}

We begin with a definition. 

\begin{defi} \label{def:qaregpmd}
Let $\Delta$ be a trianguline $(\varphi,\Gamma)$-module over $\Robba_A$, which is an extension of $\Robba_A(\delta_2)$ by $\Robba_A(\delta_1)$. We say that $\Delta$ is regular if $\delta_1\delta_2^{-1} \colon \qpe \to A^{\times}$ is regular in the sense of Definition \ref{def:qausreg}.  
\end{defi}

The following theorem is a relative version of \cite[Theorem 6.11]{colmez2015} in the case when the pair $(\delta_1,\delta_2)$ is generic and also pointwise not of the form $x^{-i}$ for some $i \geq 0$, cf. Remark \ref{rem:qausreg}. 

\begin{theo}\label{thm:quas}
Suppose $\Delta$ is a regular $(\varphi,\Gamma)$-module over $\Robba_A$ such that 
\[
0 \to \Robba_A(\delta_1) \to \Delta \to \Robba_A(\delta_2) \to 0. 
\]
Then there exists a locally analytic $A$-representation $\Pi(\Delta)$\footnote{It is probably unique but we are unable to show this. This comes down to knowing that the $\mathrm{Ext}^1$ of certain principal series is a free $A$-module of rank 1.} of $\mathrm{GL}_2(\qp)$, with central character $\omega$, such that we have an exact sequence \[ 0 \to \Pi(\Delta)^* \otimes \omega \to \Delta \boxtimes_\omega \P^1 \to \Pi(\Delta) \to 0. \]
Moreover $\Pi(\Delta)$ is an extension of $B_A(\delta_2, \delta_1)$ by $B_A(\delta_1, \delta_2)$. Furthermore if $\Delta$ is a non-trivial extension of $\Robba_A(\delta_2)$ by $\Robba_A(\delta_1)$ then $\Pi(\Delta)$ is a non-trivial extension of $B_A(\delta_2,\delta_1)$ by $B_A(\delta_1,\delta_2)$. 
\end{theo}

\begin{rema}
Contrary to \cite[Theorem 6.11]{colmez2015}, unless $A$ is a finite extension of $\qp$, then $\Pi(\Delta)$ will not be of compact type (this is because $A$ is not of compact type as a locally convex $\qp$-vector space). Thus $\Pi(\Delta)$ is almost never an admissible $G$-representation in the sense of \cite{schtelaldisad}. It is however of $A$-LB-type, cf. Definition \ref{def:frech}. 
\end{rema}

Before we begin to prove Theorem \ref{thm:quas} we need to define and construct the $G$-module $\Delta \boxtimes_{\omega} \P^1$. 


\subsection{Notations}

We let $\mathrm{Ext}^1(\Robba_A(\delta_2), \Robba_A(\delta_1))$ denote the group of extensions of $\Robba_{A}(\delta_2)$ by $\Robba_{A}(\delta_1)$ in the category of $(\varphi,\Gamma)$-modules over $\Robba_A$. Note that, since every $(\varphi, \Gamma)$-module over $\Robba_A$ is analytic (cf. Remark \ref{analytic}), this last group coincides with the extension group $\mathrm{Ext}^1_{\mathrm{an}}(\Robba_A(\delta_2), \Robba_A(\delta_1))$ in the category of analytic $(\varphi, \Gamma)$-modules over $\Robba_A$ \footnote{This fact can also be seen by using the bijections $\mathrm{Ext}^1(\Robba_A(\delta_2), \Robba_A(\delta_1)) = H^1(A^+, \Robba_A(\delta_1 \delta_2^{-1})) = H^1_{\mathrm{an}}(A^+, \Robba_A(\delta_1 \delta_2^{-1})) = \mathrm{Ext}^1_{\mathrm{an}}(\Robba_A(\delta_2), \Robba_A(\delta_1))$ where the equalities follow from \cite[Lemme 2.2]{chen2013}, Proposition \ref{prop:lazcompcom} and Proposition \ref{prop:ext1h1a}, respectively.}. 

Let $H$ be a finite dimensional locally $\qp$-analytic group. We refer the reader to the appendix for the necessary definitions and properties of the theory of locally analytic $H$-representations in $A$-modules. We just recall that $\mathscr{G}_{H, A}$ (cf. Definition \ref{def:lacatlur}) denotes the category of complete Hausdorff locally convex $A$-modules equipped with a separately continuous $A$-linear $\mathscr{D}(H , A)$-module structure and we let $ \mathrm{Ext}^1_{H}(M,N)$ \footnote{Note that this is called $ \mathrm{Ext}^1_{\mathscr{G}_{H, A}}(M,N)$ in the appendix. We warn the reader that $\mathscr{G}_{H,A}$ is not an abelian category and so one needs to define precisely what the group of extensions means, cf. Definition \ref{def:goodextnoab}.} denote the group of extensions of $M$ by $N$ in the category $\mathscr{G}_{H, A}$.

\begin{exa}
From Lemma \ref{lem:princser} and Proposition \ref{actionG}, it follows that, if $? \in \{ +, -, \emptyset\}$, then the spaces $\Robba^?_A(\delta_i) \boxtimes_\omega \P^1$ are objects of the category $\mathscr{G}_{G, A}$.
\end{exa}

If $H_2$ is a closed locally $\qp$-analytic subgroup of a locally $\qp$-analytic group $H_1$, we have an induction functor $\mathrm{ind}_{H_1}^{H_2} \colon \mathscr{G}_{H_2, A} \to \mathscr{G}_{H_1, A}$, cf. Lemma \ref{lem:predjkz}. We cite the following fact from the appendix that will be of much use to us, cf. Proposition \ref{thm:relshaplem}. 


\begin{prop} [Relative Shapiro's Lemma] \label{Shap}
Let $H_1$ be a locally $\qp$-analytic group and let $H_2$ be a closed locally $\qp$-analytic subgroup. If $M$ and $N$ are objects of $\mathscr{G}_{H_2,A}$ and $\mathscr{G}_{H_1,A}$, respectively, then there are $A$-linear bijections
\[
\Ext^{q}_{H_1}(\mathrm{ind}_{H_2}^{H_1}(M), N) \to \Ext^{q}_{H_2}(M, N)
\]
for all $q \geq 0$. 
\end{prop}


\subsection{Extensions of $\Robba_A^+(\delta_2) \boxtimes_\omega \P^1$ by $\Robba_A(\delta_1) \boxtimes_{\omega} \P^1$}

Denote by $\overline{P} = \begin{pmatrix} \qpe & 0 \\ \qp & 1 \end{pmatrix}$ the lower-half mirabolic subgroup of the lower-half Borel $\overline{B} = \begin{pmatrix} \qpe & 0 \\ \qp & \qpe \end{pmatrix}$ and $\overline{U}^{1} = \begin{pmatrix} 1 & 0 \\ p\zp & 1 \end{pmatrix}$.  We are now ready to state the first result toward a proof of Theorem \ref{thm:quas}, which is essentially a formal consequence of Theorem \ref{3.31}.   

\begin{theo}\label{thm:alps}
Let $\delta_{1}$, $\delta_{2} \colon \qpe \rightarrow A^{\times}$ such that $\delta_{1}\delta_{2}^{-1}$ is regular. Then there is a natural isomorphism \[ \mathrm{Ext}^1_G(\Robba_A^+(\delta_2) \boxtimes_\omega \P^1, \Robba_A(\delta_1) \boxtimes_\omega \P^1) \cong \mathrm{Ext}^1(\Robba_A(\delta_2), \Robba_A(\delta_1)). \]
\end{theo}

\begin{proof} Denote by $\Robba_A(\delta_1, \delta_2) \boxtimes \P^1$ the $\overline{P}$-module\footnote{Here $\delta_2$ is seen as a character of $\overline{P}$, by setting $\delta_2\begin{pmatrix} a & 0 \\ b & 1 \end{pmatrix} = \delta_2(a)$.} $(\Robba_A(\delta_1) \boxtimes_\omega \mathbf{P}^{1}) \otimes \delta_2^{-1}$, so that $\Robba_A(\delta_1,\delta_2)$ is identified with the sub-$\overline{P}^{+}$-module $(\Robba(\delta_1) \boxtimes_{\omega} \zp) \otimes \delta_2^{-1}$ of $\Robba_A(\delta_1, \delta_2) \boxtimes \P^1$. The proof is done in several steps and follows the proof of \cite[Th\'eor\`eme 6.1]{colmez2015} in the case where $A$ is a finite extension of $\qp$. 
\begin{itemize}
\item [(Step 1)] We first descend from $G$ to $\overline{P}$ using Shapiro's Lemma. Since $\Robba_A^+(\delta_2) \boxtimes_\omega \P^1 \cong \mathrm{Ind}_{\overline{B}}^G(\delta_1 \chi^{-1} \otimes \delta_2)^* \otimes \omega$, cf. Lemma \ref{lem:princser}, using Lemma \ref{lem:prinbag} we get that
\[ \Robba_A^+(\delta_2) \boxtimes_\omega \P^1 \cong \mathrm{Ind}_{\overline{B}}^G( \delta_2^{-1} \otimes \delta_1^{-1} \chi)^* \cong \mathrm{ind}_{\overline{B}}^G(\delta_2 \otimes \chi^{-1} \delta_1). \] So by Proposition \ref{Shap} (for $q=1$) we get
\[ \mathrm{Ext}^1_G(\Robba_A^+(\delta_2) \boxtimes_\omega \P^1, \Robba_A(\delta_1) \boxtimes_\omega \P^1) \cong \mathrm{Ext}^1_{\overline{B}}(\delta_2 \otimes \chi^{-1} \delta_1, \Robba_A(\delta_1) \boxtimes_\omega \P^1). \] 
Since we are only interested in locally analytic representations with a central character $\omega$, we don't loose any information by passing from $\overline{B}$ to $\overline{P}$ (since both $\delta_2 \otimes \chi^{-1}\delta_1$ and $\Robba_A(\delta_1) \boxtimes_\omega \P^1$ have the same central character, namely $\omega$) and thus we have 
\[ \mathrm{Ext}^1_{\overline{B}}(\delta_2 \otimes \chi^{-1} \delta_1, \Robba_A(\delta_1) \boxtimes_\omega \P^1) \cong \mathrm{Ext}^1_{\overline{P}}(\delta_2 \otimes \chi^{-1} \delta_1, \Robba_A(\delta_1) \boxtimes_\omega \P^1). \] 
Then, since (as $\overline{P}$-modules) $\Robba_A(\delta_1, \delta_2) \boxtimes \P^1 \cong ( \Robba_A(\delta_1) \boxtimes_\omega \P^1) \otimes (\delta_2^{-1} \otimes \chi \delta_1^{-1})$, the RHS in the above equality is also equal to \[ H^1_{\rm an}(\overline{P}, \Robba_A(\delta_1, \delta_2) \boxtimes \P^1). \]

\item [(Step 2)] We now descend from $\overline{P}$ to $\overline{P}^{+}$. That is the restriction of $\Robba_A(\delta_1, \delta_2) \boxtimes \P^1$ to a $\overline{P}^{+}$-module induces an isomorphism  
$$ H^1_{\rm an}(\overline{P}, \Robba_A(\delta_1, \delta_2) \boxtimes \P^1) = H^1_{\rm an}(\overline{P}^+, \Robba_A(\delta_1, \delta_2) \boxtimes \P^1). $$
This is shown in the exact same way as in \cite[Lemme 6.4]{colmez2015}.

\item [(Step 3)] Finally we descend from $\Robba_A(\delta_1, \delta_2) \boxtimes \P^1$ to $\Robba_A(\delta_1, \delta_2)$. More precisely we show that the inclusion $\Robba_A(\delta_1, \delta_2) \subset \Robba_A(\delta_1, \delta_2) \boxtimes \P^1$ (as $\overline{P}^{+}$-modules) induces an isomorphism
\[ H^1_{\rm an}(\overline{P}^+, \Robba_A(\delta_1, \delta_2)) \cong H^1_{\rm an}(\overline{P}^+, \Robba_A(\delta_1, \delta_2) \boxtimes \P^1). \] 
Indeed by the long exact sequence in cohomology associated to the short exact sequence 
\begin{equation}\label{eq:sdf}
0 \to \Robba_A(\delta_1, \delta_2) \to \Robba_A(\delta_1, \delta_2) \boxtimes \P^1 \to Q \to 0,
\end{equation}
where we define $Q := (\Robba_A(\delta_1,\delta_2) \boxtimes \P^1)/\Robba_A(\delta_1,\delta_2) \boxtimes \zp)$ as a $\overline{P}^+$-module, it suffices to show that $H^0_{\rm an}(\overline{P}^+, Q) = H^1_{\rm an}(\overline{P}^+, Q) = 0$. 

First observe that $Q=\Robba_A(\delta_1, \delta_2) \boxtimes (\P^1 - \zp)$ as $\overline{U}^1$-modules and that $H^1_{\rm an}(\overline{U}^1, Q) = 0$. Indeed, since $\overline{U}^1 = \begin{pmatrix} 0 & 1 \\ p & 0 \end{pmatrix} U^0 \begin{pmatrix} 0 & p^{-1} \\ 1 & 0 \end{pmatrix}$, it is enough to show that $H^1_{\rm an}(U^0, \Robba_A) = 0$, which follows from Lemma \ref{cohom1}. For the same reason $H^0_{\rm an}(\overline{U}^1, Q) = 0$ and so $H^0_{\rm an}(\overline{P}^+, Q) = 0$.  

Finally, let $c \mapsto c_g$ be a locally analytic $1$-cocycle over $\overline{P}^+$ with values in $Q$. By adding a coboundary we can assume that $c_g = 0$ for every $g \in \overline{U}^1$. For $a \in \zp \backslash \{0\}$, let $\alpha(a) = {\matrice a 0 0 1}$. By the relation $\alpha(a) {\matrice 1 0 {ap} 1} = {\matrice 1 0 p 1}\alpha(a)$ we get $c_{\alpha(a)} = {\matrice 1 0 p 1} c_{\alpha(a)}$ and thus $c_{\alpha(a)} = 0$ for every $a \in \zp \backslash \{0\}$ since ${ \matrice 1 0 p 1} - 1$ is injective on $Q$ (since $w \left({ \matrice 1 0 p 1} - 1\right) w = { \matrice 1 p 0 1} - 1$ is injective on $ \Robba(\delta_1, \delta_2) \boxtimes_\omega p \zp$ ). Thus $c_g = 0$ for all $g \in \overline{P}^+$.
\end{itemize}

Steps 1-3 show that
\[
\mathrm{Ext}^1_G(\Robba_A^+(\delta_2) \boxtimes_\omega \P^1, \Robba_A(\delta_1) \boxtimes_\omega \P^1) \cong H^1_{\rm an}(\overline{P}^+, \Robba_A(\delta_1, \delta_2)).
\]
The result now follows from Theorem \ref{3.31} and Proposition \ref{prop:lazcompcom}. 

\end{proof}

\subsection{The $G$-module $\Delta \boxtimes_{\omega} \P^1$}

In this section we show, following \cite[\S 6.3]{colmez2015}, the existence of an unique extension $\Delta \boxtimes_\omega \P^1$ extending that of $\Robba_A^+(\delta_2) \boxtimes_\omega$ by $\Robba_A(\delta_1) \boxtimes_\omega \P^1$ associated to a trianguline $(\varphi, \Gamma)$-module $\Delta \in \mathrm{Ext}^1(\Robba_A(\delta_2), \Robba_A(\delta_1))$ over $\Robba_A$ by Theorem \ref{thm:alps}. As in the introduction, we observe that working in the context of cyclotomic $(\varphi, \Gamma)$ simplifies considerably several proofs and constructions.

We begin by a lemma permitting to extend the involution.

\begin{prop}\label{prop:invext}
Let $\Delta, \Delta_1 \in \Phi\Gamma(\Robba_A)$ be in an exact sequence $0 \to \Delta_1 \to \Delta \xrightarrow{\alpha} \Robba_A(\delta) \to 0$, for $\delta \colon \qpe \to A^\times$ locally analytic, and let $\Delta_+ = \alpha^{-1}(\Robba_A^+(\delta)) \subseteq \Delta$. Let  $j_+ \colon \Robba_A^+(\Gamma) \to \Robba_A^+(\Gamma)$ and $j \colon \Robba_A(\Gamma) \to \Robba_A(\Gamma)$ be the involutions defined by $\sigma_a \mapsto \delta(a)\sigma_a^{-1}$.  The any $\Robba_A^+(\Gamma)$-anti-linear involution $\iota \colon \Delta_+ \boxtimes \zpe \to \Delta_+ \boxtimes \zpe$ with respect to $j_+$\footnote{i.e. satisfying $\iota \circ \lambda = j_+(\lambda) \circ \iota$} stabilizing $\Delta_1 \boxtimes \zpe$ extends uniquely to an $\Robba_A(\Gamma)$-anti-linear involution with respect to $j$ on $\Delta \boxtimes \zpe$.
\end{prop}

\begin{proof}
As $\Robba_A^+(\Gamma)$-modules we have
\[
\Delta_+ \boxtimes \zpe \cong (\Delta_1 \boxtimes \zpe) \oplus \Robba_A^+(\Gamma) \cdot e_2,
\] 
for some $e_2 \in \Delta_+ \boxtimes \zpe$, with $\sigma_a(e_2) = \delta(a) e_2$, $a \in \zpe$; and similarly, as $\Robba_A(\Gamma)$-modules we have
\[
\Delta \boxtimes \zpe \cong (\Delta_1 \boxtimes \zpe) \oplus \Robba_A(\Gamma) \cdot e_2,
\]
so that, in particular, the module $\Delta_+ \boxtimes \zpe$ contains a basis of $\Delta \boxtimes \zpe$ as a $\Robba_A(\Gamma)$-module. Since the involution $\iota'$ we are looking for (on $\Delta \boxtimes \zpe$) extends $\iota$ and is $\Robba_A(\Gamma)$-anti-linear, we are forced to set, for any $z = z_1 + \lambda e_2$, $z_1 \in \Delta_1$, $\lambda \in \Robba_A(\Gamma)$,
\[
\iota'(z) = \iota(z_1) + j(\lambda) \iota(e_2).
\]
Since every element of $\Delta$ can be uniquely written in this way, we deduce the result.

%
%
%

\end{proof}

Denote by $G^{+} = \begin{pmatrix} \zp & \zp \\ p\zp & \zpe \end{pmatrix} \cap G$, $\overline{B}^{+} = \begin{pmatrix} \zp \backslash \{0 \} & 0 \\ p\zp & \zpe  \end{pmatrix}$ and note that $P^+ \subset G^{+}$, $\overline{B}^{+} \subset G^{+}$ and that $G^+$ stabilizes $\zp$ so that, if $M$ is a $G$-equivariant sheaf over $\P^1$, then $M \boxtimes \zp$ inherits an action of $G^+$. The next result explicitly describes the isomorphism of Theorem \ref{thm:alps} and gives the construction of the $G$-module $\Delta \boxtimes_\omega \P^1$ for a regular $(\varphi, \Gamma)$-module $\Delta$ over $\Robba_A$.

\begin{prop}\label{prop:nextal}
Let $M$ be a non-trivial extension of $\Robba_A^+(\delta_2) \boxtimes_\omega \P^1$ by $\Robba_A(\delta_1) \boxtimes_\omega \P^1$. Then:
\begin{enumerate}
\item $M$ constains a unique $G^+$-submodule $\Delta_+$ which is an extension of $\Robba_A^+(\delta_2) \boxtimes_\omega \zp$ by $\Robba_A(\delta_1) \boxtimes_\omega \zp$.
\item There exists a unique $\Delta \in \Phi\Gamma(\Robba_A)$ which is an extension of $\Robba_A(\delta_2)$ by $\Robba_A(\delta_1)$ such that $\Delta_+$ is identified with the inverse image of $\Robba^+_A(\delta_2)$ in $\Delta$. 
\item $\Delta_+ \boxtimes \zpe$ is stable under $w$ and, if we denote by $\iota$ the involution of $\Delta_+ \boxtimes \zpe$ induced by $w$, then $M = \Delta_+ \boxtimes_{\omega,\iota} \P^1$.
\item The involution $\iota$ extends uniquely to a $\Robba_A(\Gamma)$-anti-linear involution (with respect to $j$ defined above) on $\Delta \boxtimes \zpe$ and $\Delta \boxtimes_{w,\iota} \P^1$ is a $G$-module which is an extension of $\Robba_A(\delta_2) \boxtimes_\omega \P^1$ by $\Robba_A(\delta_1) \boxtimes_\omega \P^1$.
\end{enumerate}
\end{prop}

\begin{proof}
We follow the proof of \cite[Proposition 6.7,]{colmez2015}. 
\begin{enumerate}
\item Note that $G^{+}$ stablizes $\zp$ and so $\Robba_A^{+}(\delta_2) \boxtimes_{\omega} \zp$ is a $G^{+}$-module. Step 1 of Theorem \ref{thm:alps} shows that
\[
\Robba_A^{+}(\delta_2) \boxtimes_\omega \zp \cong \mathrm{Ind}_{\overline{B}^+}^{G^+}( \delta_2^{-1} \otimes \delta_1^{-1} \chi)^* = \mathrm{ind}_{\overline{B}^+}^{G^+}( \delta_2 \otimes \chi^{-1} \delta_1).
\]
We have an exact sequence of $G^+$-modules
\[
0 \to \Robba_A(\delta_1) \boxtimes_{\omega} \zp \to \Robba_A(\delta_1) \boxtimes_{\omega} \P^1 \to Q' \to 0,
\]
where we define $Q' := (\Robba_A(\delta_1) \boxtimes_{\omega} \P^1)/ (\Robba_A(\delta_1) \boxtimes_{\omega} \zp)$ as a $G^+$-module.
Proposition \ref{Shap} gives
\begin{align*}
\mathrm{Ext}^1_{G^+}(\Robba_A^+(\delta_2) \boxtimes_\omega \zp, Q') &= \mathrm{Ext}^1_{\overline{B}^{+}}(\delta_2 \otimes \chi^{-1} \delta_1, Q') \\
&\overset{(i)}{=} \mathrm{Ext}^1_{\overline{P}^{+}}(\delta_2 \otimes \chi^{-1} \delta_1, Q') \\
&\overset{(ii)}{=} 0
\end{align*}
where (i) follows from the fact that $\delta_2 \otimes \chi^{-1} \delta_1$ and $\Robba_A(\delta_1) \boxtimes_\omega (\P^1-\zp)$ have the same central character $\omega$, and (ii) follows from Step 3 of the proof of Theorem \ref{thm:alps}. This proves the existence of a $G^{+}$-submodule $\Delta_{+} \subset M$. We now prove uniqueness. If $X \subset M$ is a $G^{+}$-extension of $\Robba_A^{+}(\delta_2) \boxtimes_{\omega} \zp$ by $\Robba_A(\delta_1) \boxtimes_{\omega} \zp$ and $e \in X$ is a lift of $1 \otimes \delta_2 \in \Robba^{+}(\delta_2)$, then $\left({ \matrice 1 0 p 1} - 1\right)\cdot e \in \Robba_A(\delta_1) \boxtimes_{\omega} \zp$. Writing $e = e' + z$ for some $e' \in \Delta_+$ and $z \in Q'$, we see that $\left({ \matrice 1 0 p 1} - 1\right)\cdot z \in \Robba_A(\delta_1) \boxtimes_{\omega} \zp$. Thus $z=0$ as $\left({ \matrice 1 0 p 1} - 1\right)$ is injective on $Q'$ (noting that $Q' = \Robba_A(\delta_1) \boxtimes_\omega (\P^1-\zp)$ as $\overline{U}^1$-modules). 

\item This follows from the fact that extensions (as $A^+$-modules) of $\Robba_A^{+}(\delta_2)$ by $\Robba_A(\delta_1)$ are in correspondence with extensions of $\Robba_A(\delta_2)$ by $\Robba_A(\delta_1)$. 

\item Since $\Delta_+$ is a $G^{+}$-module, by restricting to $P^{+}$, we can think of it as  a $P^{+}$-module living on $\zp$. Thus the notation $\Delta_+ \boxtimes \zpe$ makes sense.

For $i \in \zpe$ the identity
\[
w\begin{pmatrix}p & i \\ 0 & 1 \end{pmatrix} = \begin{pmatrix} p & i^{-1}\\ 0 & 1 \end{pmatrix}\begin{pmatrix} -i^{-1} & 0 \\ p & i \end{pmatrix}
\] 
and the fact that $\begin{pmatrix} -i^{-1} & 0 \\ p & i \end{pmatrix} \in G^{+}$ implies that
\begin{equation}\label{eq:alen}
w (\Delta_+ \boxtimes (i+ p\zp)) \subset \Delta_+ \boxtimes (i^{-1} + p\zp)
\end{equation}
The inclusion in \eqref{eq:alen} is in fact an equality since $w$ is an involution. This shows the stability of $\Delta_+ \boxtimes \zpe$ by $w$.

To show the equality of the statement, it suffices to show that the sequence
\begin{equation}\label{eq:donz}
0 \rightarrow \Delta_+ \boxtimes \zpe \xrightarrow{x \mapsto (x,- wx)} \Delta_+ \oplus \Delta_+ \xrightarrow{(y,z) \mapsto y + w \cdot z} M \to 0
\end{equation}
is exact. The exactness on the middle is clear since any section supported on $\zp$ such that it's involution is also supported on $\zp$ is necessarily supported on $\zpe$. Finally, to prove surjectivity of $\Delta_+ \oplus \Delta_+ \to M$ in \eqref{eq:donz}, it suffices to note that induced applications $\Robba_A(\delta_1) \oplus \Robba_A(\delta_1) \to \Robba_A(\delta_1) \boxtimes_{\omega} \P^1$ and $\Robba_A^+(\delta_2) \oplus \Robba_A^{+}(\delta_2) \to \Robba_A^+(\delta_2) \boxtimes_\omega \P^1$ are surjective. 

\item The existence and uniqueness of $\iota$ follows from Proposition \ref{prop:invext}. For the last part it suffices to show that the action of $\tilde{G}$ on $\Delta \boxtimes_{\omega, \iota} \P^1$ factorizes via $G$. First note that if $A$ is a finite extension of $\qp$ the result follows from \cite[Proposition 6.7(iv)]{colmez2015}. 
 
We now proceed by induction on the index $i \geq 0$ of nilpotence of $A$. Suppose first that $A$ is reduced. Take $(z_1,z_2) \in \Delta \boxtimes_{\omega, \iota} \P^1$ and $g$ in the kernel of $\tilde{G} \to G$. It suffices to show that $y = (g-1)z = 0$. Call $y = (y_1, y_2)$. Let $\mathfrak{m} \subset A$ be a maximal ideal. By Lemma \ref{lem:cast} and by the result for the case of a point, $y_i =0$ mod $\mathfrak{m}$. If we write $y_i = \sum_{n \in \Z} a_{n,i} T^n \oplus \sum_{n \in \Z} a_{n,i} T^n$ for $i=1,2$ we see that $a_{n,i} = 0$ mod $\mathfrak{m}$ and hence $y_i = 0$ so that $y = 0$, as desired.

Suppose now the result is true for every affinoid algebra of index of nilpotence $\leq j$ and let $A$ be an affinoid algebra whose nilradical $N$ satisfies $N^{j+1} = 0$ and $g$ be in the kernel of $\tilde{G} \to G$. We have the following short exact sequence (note that $\Delta$ is a flat $A$-module because it is an extension of flat $A$-modules and so is $\Delta \boxtimes_{\omega,\iota} \P^1$ who is topologically isomorphic to two copies of $\Delta$)
\[ 0 \to (\Delta \boxtimes_{\omega,\iota} \P^1) \otimes_{A} N^j \to \Delta \boxtimes_{\omega, \iota} \P^1 \to (\Delta \boxtimes_{\omega, \iota} \P^1) \otimes_A A/N^j \to 0. \]
We can identify $(\Delta \boxtimes_{\omega, \iota} \P^1) \otimes_A A/N^j$ with $(\Delta \otimes_A A/N^j) \boxtimes_{\omega, \iota} \P^1$ and $(\Delta \boxtimes_{\omega,\iota} \P^1) \otimes_{A} N^j$ with $(\Delta \otimes_{A} A/N) \boxtimes_{\omega,\iota} \P^1 \otimes_{A/N} N^j$. The result now follows by the inductive hypothesis and the base case.

\end{enumerate}

\end{proof}

From now on we denote by $\Delta \boxtimes_\omega \P^1$ the module $\Delta \boxtimes_{\omega, \iota} \P^1$ constructed in Proposition \ref{prop:nextal}. 

\subsection{The representation $\Pi(\Delta)$}
We are now almost ready to construct the representation $\Pi(\Delta)$ and prove Theorem \ref{thm:quas}. We will need some preparation results. We start by showing that $H^1(\overline{P}^{+}, \Robba_A^{-}(\delta_1, \delta_2))$ is a free $A$-module in the \emph{quasi}-regular case\footnote{Here quasi-regular means that pointwise $\delta_1\delta_2^{-1}$ is never of the form $\chi x^{i}$ for some $i \geq 0$. Clearly regular implies quasi-regular.}. 

\begin{prop}\label{prop:detg}
Let $\delta_1, \delta_2 \colon \qpe \to A^\times$ be locally analytic characters such that $\delta_1\delta_2^{-1}$ is quasi-regular. Then there is an isomorphism $H^1(\overline{P}^{+}, \Robba_A^{-}(\delta_1, \delta_2))$ is a free $A$-module of rank $1$.
\end{prop}

\begin{proof}
This will follow from Proposition \ref{prop:LAfree} and Lemma \ref{lem:isogeh} below.
\end{proof}

\begin{lemm}\label{lem:kasd}
$\Robba_A^-$ is a flat $A$-module. 
\end{lemm}

\begin{proof}
For $0 < r \leq s \leq \infty$ rings $\Robba_A^{[r, s]}$ are Banach $A$-algebras of countable type. Thus by \cite[Lemma 1.3.8]{kedpadic}, we can identify $\Robba_A^{[r, s]}/\Robba_A^{[r,\infty]}$ with the completed direct sum $\widehat{\oplus}_{i \in I}Ae_i$ where $(e_i)_{i \in I}$ form a potentially orthonormal basis. We note in the following $(\Robba_A^{[r, s]})^-$ the module $\Robba_A^{[r, s]}/\Robba_A^{[r,\infty]}$. Under this identification, if $I \subseteq A'$ is a finitely generated ideal of $A$, then 
\[
I \otimes_A (\Robba_A^{[r, s]})^- \cong \widehat{\oplus}_{i \in I} Ie_i.
\]
This implies that the morphism $I \otimes_A (\Robba_A^{[r, s]})^- \to (\Robba_A^{[r, s]})^-$ is injective. Thus $(\Robba_A^{[r, s]})^-$ is a flat $A$-module. Observe that, in fact, the quotient $\Robba_A^{[r, s]}/\Robba_A^{[r,\infty]}$ does not depend on $r$ and coincides with $\Robba_A^{]0, s]}/\Robba_A^{]0,\infty]}$ so that this last module is flat. Finally, since filtered colimits are exact, this implies that 
\[
\varinjlim_{s >0} \Robba_A^{]0, s]}/\Robba_A^{]0, \infty]} = \Robba_A/\Robba_A^+ = \Robba_A^-
\]
is a flat $A$-module. 

\end{proof}

%

\begin{lemm}\label{lem:isogeh}
Let $\delta_1, \delta_2 \colon \qpe \to A^\times$ such that $\delta_1\delta_2^{-1}$ is quasi-regular. The restriction morphism
\[
H^{1}(\overline{P}^{+}, \Robba_A^{-}(\delta_1, \delta_2)) \to H^1(A^+, \Robba_A^{-}(\delta))
\]
is an isomorphism.
\end{lemm}

\begin{proof}
This is precisely the same proof as Theorem \ref{3.31} with $\Robba_A^{-}$ replacing $\Robba_A$. The key points are the following
\begin{itemize}
\item The morphism $$C^{\bullet}_{\overline{P}^{+}}(\mathscr{R}^-_{A}(\delta_1, \delta_2)) \rightarrow C^{\bullet}_{A^+}(\mathscr{R}^-_{A}(\delta))$$
is in $\mathcal{D}^{-}_{\mathrm{pc}}(A)$ by Lemmas \ref{H0A+R-} and  \ref{cols} and Proposition \ref{prop:LAfree}. 
\item $\Robba_A^-$ is flat $A$-module, cf. Lemma \ref{lem:kasd}. This means that for any maximal ideal $\mathfrak{m}  \subset A$ we have
$$C^{\bullet}_{\overline{P}^{+}}(\mathscr{R}^-_{A}(\delta_1, \delta_2)) \otimes^{\mathbf{L}} A/\mathfrak{m} \cong C^{\bullet}_{\overline{P}^{+}}(\mathscr{R}^-_{A/\mathfrak{m}}(\delta_1, \delta_2))$$
and
$$C^{\bullet}_{A^+}(\mathscr{R}^-_{A}(\delta)) \otimes^{\mathbf{L}} A/\mathfrak{m} \cong C^{\bullet}_{A^+}(\mathscr{R}^-_{A/\mathfrak{m}}(\delta)).$$
\item Since $\delta_1\delta_2^{-1}$ is quasi-regular, $H^{2}(C^{\bullet}_{A^+}(\mathscr{R}^-_{A}(\delta))) = 0$ by Corollary \ref{H2A+R-cor}.

\item The result is true when $A$ is a finite extension of $\qp$, cf. \cite[Lemme 5.24]{colmez2015}. 
\end{itemize}
\end{proof}

This completes the proof of Proposition \ref{prop:detg}. Finally we need a lemma which identifies $\check{\Delta} \boxtimes_{\omega^{-1}} \P^1$ as the topological dual of $\Delta \boxtimes_{\omega} \P^1$ (equipped with the strong topology). 

\begin{lemm} \label{lem:dualla}
If $\Delta$ is an extension of $\Robba_A(\delta_2)$ by $\Robba_A(\delta_1)$ and if the $G$-module $\Delta \boxtimes_{\omega} \P^1$ exists, then its dual is $\check{\Delta} \boxtimes_{\omega^{-1}} \P^1$. 
\end{lemm}

\begin{proof}
This is the same proof as \cite[Proposition 3.2]{colmez2015}, so we just provide a sketch. It suffices to construct a perfect pairing which identifies $\Delta \boxtimes_{\omega} \P^1$ and $\check{\Delta} \boxtimes_{\omega^{-1}} \P^1$ as the topological duals of one another (as topological $A[G]$-modules). To construct this pairing first note that if $(\Delta, \omega, \iota)$ is compatible, then so is  $(\check{\Delta}, \omega^{-1}, \iota^{*})$, where 
\[
\iota^* \colon \check{\Delta} \boxtimes \zpe \to \check{\Delta} \boxtimes \zpe
\] 
is the involution of $\check{\Delta} \boxtimes \zpe$ adjoint to that of $\iota$ with respect to the pairing $\lbrace \text{ }, \text{ } \rbrace$ (cf. \S\ref{sec:dark} for the definition of $\lbrace \text{ }, \text{ } \rbrace$). To see this one first defines a pairing
\begin{align*}
\lbrace \text{ },\text{ } \rbrace_{\P^1} \colon (\check{\Delta} \boxtimes_{\omega^{-1}, \iota^{*}} \P^1) \times (\Delta \boxtimes_{\omega, \iota} \P^1) \to A
\end{align*} 
by the formula
\[
\lbrace \check{z}, z \rbrace_{\P^1} := \lbrace \Res_{\zp} \check{z}, \Res_{\zp} z \rbrace + \lbrace \Res_{p\zp} w \cdot \check{z}, \Res_{p\zp} w \cdot z \rbrace. 
\]
One then proceeds to check that $\lbrace g\cdot \check{z}, g\cdot z \rbrace_{\P^1} = \lbrace \check{z}, z \rbrace_{\P^1}$ for all $g \in \tilde{G}$. Finally one notes that since $\lbrace \text{ },\text{ } \rbrace$ is perfect, so is $\lbrace \text{ },\text{ } \rbrace_{\P^1}$.  
\end{proof}

\begin{proof}[Proof of Theorem \ref{thm:quas}]
By Proposition \ref{actionG}, Theorem \ref{thm:alps} and Proposition \ref{prop:nextal} we have that\footnote{The notation $M = [M_1 - M_2 - \ldots -M_n]$ means that $M$ admits an increasing filtration $0 \subseteq F_1 \subseteq \hdots \subseteq F_n = M$ by sub-objects such that $M_i = F_i / F_{i - 1}$ for $i = 1, \hdots, n$. }
\[
\Delta \boxtimes_{\omega} \P^1 = [B_A(\delta_2,\delta_1)^* \otimes \omega - B_A(\delta_1, \delta_2) - B_A(\delta_1, \delta_2)^* \otimes \omega - B_A(\delta_2, \delta_1)].
\]
We begin by showing that the middle extension $[B_A(\delta_1, \delta_2) - B_A(\delta_1, \delta_2)^* \otimes \omega]$ is split in the category $\mathscr{G}_{G, A}$. We compute
\begin{align*}
\mathrm{Ext}^1_G(B(\delta_1, \delta_2)^* \otimes \omega, \Robba_A^-(\delta_1) \boxtimes_\omega \P^1) &\overset{(i)}{\cong} \mathrm{Ext}^1_{\overline{B}}(\delta_2 \otimes \chi^{-1} \delta_1, \Robba_A^-(\delta_1) \boxtimes_\omega \P^1) \\
&\overset{(ii)}{\cong} \mathrm{Ext}^1_{\overline{P}}(\delta_2 \otimes \chi^{-1} \delta_1, \Robba_A^-(\delta_1) \boxtimes_\omega \P^1)
\end{align*}
where $(i)$ follows from Proposition \ref{Shap} and $(ii)$ follows from the fact that both $(\delta_2 \otimes \chi^{-1} \delta_1)$ and $\Robba_A^-(\delta_1) \boxtimes_\omega \P^1$ have the same central character. Then as $\overline{P}$-modules 
\[
(\Robba_A^-(\delta_1) \boxtimes_\omega \P^1) \otimes (\delta_2^{-1} \otimes \chi\delta_1^{-1}) \cong (\Robba_A^-(\delta_1) \boxtimes_\omega \P^1) \otimes \delta_2^{-1}. 
\]
Let us denote by $\Robba_A^-(\delta_1,\delta_2) \boxtimes \P^1$ the $\overline{P}$-module $(\Robba_A^-(\delta_1) \boxtimes_\omega \P^1) \otimes \delta_2^{-1}$, so that we get
\[ \mathrm{Ext}^1_{\overline{P}}(\delta_2 \otimes \chi^{-1} \delta_1, \Robba_A^-(\delta_1) \boxtimes_\omega \P^1) = H^1_{\rm an}(\overline{P}, \Robba_A^-(\delta_1,\delta_2) \boxtimes \P^1).\] Finally, by \cite[Lemme 6.4]{colmez2015}
\[
H^1_{\rm an}(\overline{P}, \Robba_A^-(\delta_1,\delta_2) \boxtimes \P^1) = H^1_{\rm an}(\overline{P}^+, \Robba_A^-(\delta_1,\delta_2) \boxtimes \P^1). 
\] Putting this calculations together, we conclude that
\[ \mathrm{Ext}^1_G(B(\delta_1, \delta_2)^* \otimes \omega, \Robba_A^-(\delta_1) \boxtimes_\omega \P^1)  = H^1_{\rm an}(\overline{P}^+, \Robba_A^-(\delta_1,\delta_2) \boxtimes \P^1). \]

Consider the commutative diagram:
$$
\begin{tikzcd} [row sep = large, column sep = large]
H^{1}_{\rm an}(\overline{P}^+, \Robba_A(\delta_1,\delta_2)) \arrow[r, "\sim"] \arrow[d] &
H^{1}_{\rm an}(\overline{P}^+, \Robba_A(\delta_1,\delta_2) \boxtimes \P^1) \arrow[d]  \\
H^{1}_{\rm an}(\overline{P}^+, \Robba_A^-(\delta_1,\delta_2)) \arrow[r] &
H^{1}_{\rm an}(\overline{P}^+, \Robba_A^-(\delta_1,\delta_2) \boxtimes \P^1),
\end{tikzcd}
$$
where the top horizontal arrow is an isomorphism by Theorems \ref{3.31} and \ref{thm:alps}. By reinterpreting the extensions in terms of cohomology classes, we see that showing that the middle extension $[B_A(\delta_1, \delta_2) - B_A(\delta_1, \delta_2)^* \otimes \omega]$ splits is equivalent to showing that the right vertical arrow is the zero morphism. Now, by Proposition \ref{prop:detg}, $H^{1}_{\rm an}(\overline{P}^+, \Robba_A^-(\delta_1,\delta_2))$ is a free $A$-module of rank 1 and the same proof as the first point in \cite[Remarque 5.26]{colmez2015} now shows that the bottom horizontal arrow
\[
H^{1}_{\rm an}(\overline{P}^+, \Robba_A^-(\delta_1,\delta_2)) \to H^{1}_{\rm an}(\overline{P}^+, \Robba_A^-(\delta_1,\delta_2) \boxtimes \P^1)
\]
is the zero morphism and so is the right vertical arrow, proving the claim.

It follows that $\Delta \boxtimes_{\omega} \P^1$ is an extension of $\Pi_1$ by $\Pi_2^* \otimes \omega$, where $\Pi_1$ and $\Pi_2$ are extensions of $B_A(\delta_2,\delta_1)$ by $B_A(\delta_1,\delta_2)$. By Lemma \ref{lem:dualla}, it follows that $\Pi_1 = \Pi_2$. We now define $\Pi(\Delta) := \Pi_2$. 

Furthermore if $\Delta$ is a non-trivial extension of $\Robba_A(\delta_2)$ by $\Robba_A(\delta_1)$, then so is the extension of $B_A(\delta_1,\delta_2)^* \otimes \omega$ by $\Robba_A(\delta_1) \boxtimes_{\omega} \P^1$ (recalling that 
\[
\mathrm{Ext}^1_G(\Robba_A^+(\delta_2) \boxtimes_\omega \P^1, \Robba_A(\delta_1) \boxtimes_\omega \P^1) \cong \mathrm{Ext}^1(\Robba_A(\delta_2), \Robba_A(\delta_1)),\]
cf. Theorem \ref{thm:alps}). This implies that $\Pi(\Delta)$ is a non-trivial extension of $B_A(\delta_2,\delta_2)$ by $B_A(\delta_1,\delta_2)$ and finishes the proof.

\end{proof}

\appendix
\section{The category of locally analytic $G$-representations in $A$-modules} \label{ap:lagrep}

Let $A$ be an affinoid $\qp$-algebra (in the sense of Tate). Let $H$ be a locally $\qp$-analytic group (for applications $H$ will be a closed locally $\qp$-analytic subgroup of $\GL_2(\qp)$). We attempt to generalise some definitions from \cite{schneidnon} in the case where the base coefficient is an $\qp$-affinoid algebra. There are some results in this direction in \cite[\S 3]{joheignnew} although our approach is different. In particular, our aim is to give a reasonable definition of the category of locally analytic $H$-representations in $A$-modules, $\mathrm{Rep}_{A}^{\mathrm{la}}(H)$, analogous to the definition in \cite[\S 3]{tetschlad} (in the case where $A$ is a finite extension of $\qp$), and to study the (locally analytic) cohomology of such a representation.

\subsection{Preliminaries and definitions}

In what follows if $V$ and $W$ are two locally convex $\qp$-vector spaces and in the situation that the bijection
\[
V \otimes_{\qp,\iota} W \to V \otimes_{\qp, \pi} W 
\]
is a topological isomorphism, we write simply $V \otimes_{\qp} W$ and $V \widehat{\otimes}_{\qp} W$ to denote the topological tensor product and its completion, respectively. Note that $H$ is strictly paracompact (this means that every open covering of $H$ admits a locally finite refinement of pairwise disjoint open subsets) and so it admits a covering of pairwise disjoint open compact subsets.  

We need to define a notion of a locally convex $A$-module. First we recall the definition of a \emph{topological} $A$-module.   

\begin{defi}
A topological $A$-module is an $A$-module endowed with a topology such that module addition $+ \colon M \times M \to M$ and scalar multiplication $\cdot \colon A \times M \to M$ are continuous functions (where the domains of these functions are endowed with product topologies). 
\end{defi}

\begin{defi}
Let $M$ be an $A$-module. A seminorm $q$ on $M$ is a function $q \colon M \to \mathbf{R}$ such that
\begin{itemize}
\item $q(am) = \lvert a \lvert q(m)$ for all $a \in A$ and $m \in M$, where $\lvert \cdot \lvert$ is any non-zero multiplicative seminorm on $A$.
\item $q(m+n) \leq \mathrm{max} \{q(m),q(n) \}$ for any $m, n \in M$. 
\end{itemize}
\end{defi}

Let $(q_i)_{i \in I}$ be a family of seminorms on an $A$-module $M$. We define a topology on $M$ to be the coarsest topology on $M$ such that
\begin{enumerate}
\item All $q_i \colon M \to \mathbf{R}$, for $i \in I$, are continuous.
\item All translation maps $m + - \colon M \to M$, for $m \in M$, are continuous.
\end{enumerate} 


\begin{rema} \label{rem:fagt}
One would at a first glance be tempted to define a locally convex $A$-module as a topological $A$-module whose underlying topology is given by a family of seminorms in the above sense. The problem with this definition is twofold. The $\qp$-affinoid algebra $A$ equipped with the topology defined by the Gauss norm say, will not necessarily be a locally convex $A$-module (unless $A$ is reduced). This is essentially due to the fact that the Gauss norm is not necessarily multiplicative on $A$. On the other hand the topology on $A$ defined by the seminorms coming from the Berkovich spectrum $\mathcal{M}(A)$ coincides with the topology induced by the spectral seminorm ($f \in A \mapsto \mathrm{max}_{x \in \mathcal{M}(A)} \lvert f(x) \lvert$). Under this topology $A$ will indeed be a locally convex $A$-module but not necessarily Hausdorff. 
\end{rema}

Due to Remark \ref{rem:fagt} we define a locally convex $A$-module in the following way.

\begin{defi} \label{def:right} 
A locally convex $A$-module is a topological $A$-module whose underlying topology is a locally convex $\qp$-vector space. We let $\mathrm{LCS}_A$ be the category of locally convex $A$-modules. Its morphisms are all continuous $A$-linear maps.
\end{defi}

\begin{rema}
Let us show that this definition is coherent in the case when $A = L$ is a finite extension of $\qp$. That is, that a locally convex $\qp$-vector space equipped with a continuous multiplication by $L$ is also an $L$-locally convex vector space.

We employ the notion of \cite[\S 4]{schneidnon}. Let $L$ be a finite extension of $\qp$. It is clear that any locally convex $L$-vector space in the sense of loc.cit. satisfies the conditions of Definition \ref{def:right}. On the other hand, let $M$ be a locally convex $\qp$-vector space (whose topology is defined by a family of lattices $\mathscr{B}$) equipped with a continuous multiplication by $L$. 
We show that we can equip $M$ with a system of lattices $\mathscr{B}'$ satisfying conditions (lc1) and (lc2) of \cite[\S 4]{schneidnon} with $K$ there replaced by $L$ defining the same topology as $\mathscr{B}$. Let $x_i$ be a $\zp$-basis of $\mathcal{O}_L$. For $U \in \mathscr{B}$, set $U' = \sum_i x_i U$ and denote $\mathscr{B}' = \{ U' : U \in \mathscr{B} \}$. It is easy to see that this family of $\mathcal{O}_L$-lattices satisfies conditions (lc1) and (lc2). We show that the topology defined by $\mathscr{B}'$ coincides with that defined by $\mathscr{B}$.

- Let $U \in \mathscr{B}$. Since multiplication by $x_i$ is an homeomorphism (multiplication by $x_i^{-1}$ being a continuous inverse) $x_i U$ for all $i$ is open (in the topology defined by $\mathscr{B}$). Thus so is $\sum x_i U = U'$, which shows that the topology defined by $\mathscr{B}$ is finer than that defined by $\mathscr{B}'$.

- On the otherhand, since $U' = \sum x_i U$ is open in the topology defined by $\mathscr{B}'$ so is $\cap x_i^{-1}U' = U$ ($L$ is a finite extension of $\qp$, so this is a finite intersection). This shows that the topology defined by $\mathscr{B}'$ is finer than that defined by $\mathscr{B}$. This shows that a locally convex $\qp$-vector space equipped with  a continuous multiplication by $L$ is also an $L$-convex vector space.  

\end{rema}

\begin{lemm}
$A$ equipped with its norm topology is a barrelled, complete Hausdorff locally convex $A$-module.
\end{lemm}

\begin{proof}
Consider the induced Gauss norm on $A$, $\lvert \cdot \lvert$. By \cite[\S 3.1, Proposition 5(i)]{bosh1}, $\lvert \cdot \lvert$ is a $\qp$-algebra norm. Thus $A$ equipped with its norm topology is a locally convex $\qp$-vector space. Moreover since $\lvert \cdot \lvert$ is sub-multiplicative ($\lvert ab \lvert \leq \lvert a \lvert \cdot \lvert b\lvert$), multiplication by $A$ is continuous. Finally since the topology is defined by a norm, it is Hausdorff and $A$ is a $\qp$-Banach algebra for this norm. Note that all Banach spaces are barrelled, cf. \cite[page 40]{schneidnon}. 
\end{proof}

We now prove that there is a well defined Hausdorff completion for a locally convex $A$-module. 

\begin{lemm} \label{lem:hauscomp}
For any locally convex $A$-module $M$ there exists an up to unique topological isomorphism unique complete Hausdorff topological $A$-module $\widehat{M}$ together with a continuous $A$-linear map
\[
c_M  \colon M \to \widehat{M}
\]
such that the following universal property holds: For any continuous $A$-linear map $f \colon M \to N$ into a complete Hausdorff locally convex $A$-module $N$ there is a unique continuous $A$-linear map $\widehat{f} \colon \widehat{M} \to N$ such that
\[
f = \widehat{f} \circ c_M. 
\]
\end{lemm}

\begin{proof}
Uniqueness follows from the universal property. For the existence, replacing $M$ by $M/\overline{\{0\}}$ if necessary, we may assume that $M$ is Hausdorff (note that $\overline{\{0\}}$ is a locally convex $A$-module and thus so is $M/\overline{\{0\}}$). We consider $M$ as a locally convex $\qp$-vector space and let $\widehat{M}$ be as in \cite[Proposition 7.5]{schneidnon}. We show that $\widehat{M}$ is a topological $A$-module. It is easy to see that $\widehat{A \times M} = A \times \widehat{M}$ (as locally convex $\qp$-vector spaces). By the universal property in loc.cit. the $A$-module structure $A \times M \to M$ extends to a continuous morphism $\alpha \colon A \times \widehat{M} \to \widehat{M}$. Since $M \to \widehat{M}$ is an injection and $M$ is dense in $\widehat{M}$, it follows that $\alpha$ exhibits $\widehat{M}$ as a topological $A$-module. Finally since $c_M$ is an injection it is $A$-linear. 
\end{proof}

From now on when $A$ is considered as a locally convex $A$-module, we will assume it is equipped with its Gauss-norm topology, $(A, |\cdot|_A)$. 

We need the notion of the \emph{strong} dual of a locally convex $A$-module. This will be much less well behaved compared to the classical situation when $A$ is a finite extension of $\qp$. For example if the dimension of $H$ is greater than 1 then we do not even know if $\mathrm{LA}(H,A)$ is $A$-reflexive, cf. Conjecture \ref{conj:compLA} and Remark \ref{rem:conrelg}. 

\begin{defi} \label{def:stdua}
Let $M$ be a locally convex $A$-module. As an abstract $A$-module, we define the
\[
M'_b := \mathrm{Hom}_{A,\mathrm{cont}}(M,A)
\]
Now equip $\mathrm{Hom}_{\qp, \mathrm{cont}}(M,A)$ with the strong $\qp$-locally convex topology and give 
\[M'_b \subseteq \mathrm{Hom}_{\qp, \mathrm{cont}}(M,A)
\]
the induced subspace topology. We call $M'_b$ equipped with this topology, the strong dual of $M$. We say $M$ is $A$-reflexive if the canonical morphism $M \to (M'_b)'_b$ is a topological isomorphism. 
\end{defi}

\begin{rema}
Let $M$ be a locally convex $A$-module. The strong dual $M'_b$ can equivalently be defined with the topology obtained by taking the sets $\{f \colon M \to A | f(B) \subseteq U \}$ for $B \subseteq M$ a bounded set and $U \subseteq A$ open, as a system of neighbourhoods of $0$. 
\end{rema}

Let $R \in \{ \qp,A \}$. For $M$ and $N$ locally convex $R$-modules, we will sometimes write $\mathcal{L}_{R,b}(M,N) := \mathrm{Hom}_{R,\mathrm{cont}}(M,N)$ equipped with the strong topology. If $M$ is a locally convex $\qp$-vector space, then we will denote the \emph{classical} strong dual of $M$, cf. \cite[Chapter 1, \S 9]{schneidnon}, by $M'_{\qp,b}$. 

We now prove that $M'_b$ as given by Definition \ref{def:stdua} is indeed a locally convex $A$-module.

\begin{lemm}
If $M$ is a locally convex $A$-module, then so is its strong dual $M'_b$.
\end{lemm}

\begin{proof}
By \cite[\S 5]{schneidnon}, $M'_b$ is a locally convex $\qp$-vector space. So it suffices to show that multiplication by $A$ is continuous. This comes down to chasing definitions. Let $\lvert \cdot \lvert$ be the Gauss norm on $A$. For any bounded set $B$ of $M$ (viewing $M$ as a locally convex $\qp$-vector space), we have the seminorm (on $\mathrm{Hom}_{\qp, \mathrm{cont}}(M,A)$)
\[
p_{B}(f) := \mathrm{sup}_{v \in B} \lvert f(v) \lvert.
\]
The locally convex topology on $\mathrm{Hom}_{\qp, \mathrm{cont}}(M,A)$ is then defined by the family of seminorms $\{ p_B \}_{B \in \mathscr{B}}$ where $\mathscr{B}$ is the set of all bounded subsets of $M$. For any finitely many seminorms $p_{B_1}, p_{B_2}, \ldots, p_{B_r}$ in the given family and any real number $\epsilon >0$ the open sets
\[
\{ f \in \mathrm{Hom}_{\qp, \mathrm{cont}}(M,A) \lvert \text{ }p_{B_1}, p_{B_2}, \ldots, p_{B_r}(f) \leq \epsilon \}
\]
form a basis around 0 of $\mathrm{Hom}_{\qp, \mathrm{cont}}(M,A)$. For any $a \in A$ we have $p_B(af) \leq \lvert a \lvert p_B(f)$ and now it is easy to see that multiplication by $A$ is  continuous on $\mathrm{Hom}_{\qp, \mathrm{cont}}(M,A)$. Thus it is also continuous on $\mathrm{Hom}_{A, \mathrm{cont}}(M,A)$.  
\end{proof}

\begin{exa} \label{ex:afelrob}
By \cite[Lemma 2.1.19]{kedlaya2014} (cf. also \cite[5.5 Proposition]{richcrewfin}), we have  topological isomorphisms \footnote{Note that, identifying $\Robba_A^- = \varinjlim_s \Robba_A^{]0, s]} / \Robba_A^{]0, +\infty]}$, one can see (cf. Lemma \ref{lem:LAgoodfu} and Lemma \ref{homvarinj} below) that $\mathcal{L}_{A, b}(\Robba_A^-, A) = \varprojlim_s \mathcal{L}_{A, b}( \Robba_A^{]0, s]} / \Robba^{]0, +\infty]}, A)$ is Fr\'echet (observe that the space $\Robba_A^{]0, s]} / \Robba_A^{]0, +\infty]}$ with its topology defined by $v^{]0, s]}$, where  $v^{]0, s]}$ is the valuation induced by $v^{[r, s]}$ for any $r < s$, is a Banach space).}
\[ \mathcal{L}_{A, b}(\Robba_A^{-}, A) = \Robba_A^+, \;\;\; \mathcal{L}_{A, b}(\Robba_A, A) = \Robba_A. \] Note that, if we denote by $\Robba_A^\sim$ the sub-$A$-module of $\Robba_A$ given by Laurent series whose non-negative powers of $T$ vanish, equipped with it's induced topology, then we have $(\Robba_A^+)^{\perp} = \Robba_A^+$ and $(\Robba_A^\sim)^{\perp} = \Robba_A^\sim$ (the orthogonal is taken with respect to the natural separately continuous pairing $\Robba_A \times \Robba_A \to A$, $(f, g) \mapsto \text{r\'es}_0(f(T)g(T) dT)$). This shows that $\Robba_A^+$ and $\Robba_A^\sim$ are closed sub-$A$-modules of $\Robba_A$ and that we have a topological decomposition $\Robba_A = \Robba_A^+ \oplus \Robba_A^\sim$ (cf. \cite[Proposition 8.8(ii)]{schneidnon}). We can conclude, by using \cite[Lemma 5.3(ii)]{schneidnon}, that we have a topological isomorphism of locally convex $A$-modules $\Robba_A^- \cong \Robba_A^\sim$ and $\Robba_A^+ \cong \Robba_A / \Robba_A^\sim$. On the other hand, the same argument of \cite[Lemma 2.1.19]{kedlaya2014} shows that \[ \mathcal{L}_{A, b}(\Robba_A^+, A) = \Robba_A^-. \] Indeed, as in loc. cit., the inverse to the natural map $\Robba_A \to \mathrm{Hom}_{A, \mathrm{cont}}(\Robba_A, A)$ induced by the pairing $\text{r\'es}_0$ is given by associating, to any $\mu \in \mathrm{Hom}_{A, \mathrm{cont}}(\Robba_A, A)$, the power series $\sum_{n \in \Z} \mu(T^{-1 - n}) T^n$. This series lies in $\Robba_A^\sim$ if and only if $\mu(T^{-1 - n}) = 0$ for all $n \geq 0$, or equivalently $\mu(\Robba_A^\sim) = 0$. This shows in particular that the spaces $\Robba_A^+$ and $\Robba_A^-$ are $A$-reflexive.
\end{exa}

We need to recall some classical definitions adapted to our context:

\begin{defi}
A Cauchy net in a locally convex $A$-module $M$ is a net $(m_i)_{i \in I}$ in $M$ (a family
of vectors $m_i$ in $M$ where the index set $I$ is directed) such that for every $\epsilon >0$ and every seminorm $q_\alpha$ (defining the topology on $M$), $\exists \kappa$ such that for all $\lambda, \mu > \kappa$, $q_\alpha(x_\lambda - x_\mu) < \epsilon$.  $M$ is complete if and only if every Cauchy net converges.  
\end{defi}

\begin{defi} \label{def:frech}
Let $R \in \{ \qp, A \}$ and let $M$ be a locally convex $R$-module. We say $M$ is a Fr\'echet space if it is metrizable and complete. We say $M$ is $R$-LB-type if it can be written as a countable increasing union of $R$-Banach spaces with $R$-linear injective transition morphisms. We say $M$ is $R$-LF-type if it can be written as a countable increasing union of locally convex $R$-modules which are Fr\'echet spaces with $R$-linear injective transition morphisms. 
\end{defi}

\subsection{Relative non-archimedean functional analysis}

Here are our first (and main) examples of locally convex $A$-modules:

\begin{lemm} \label{lem:magdimo}
$\Robba_A^+$, $\Robba_A^+ \boxtimes \P^1$, $\Robba_A^-$, $\Robba_A^- \boxtimes \P^1$, $\Robba_A$ and $\Robba_A \boxtimes \P^1$ are complete Hausdorff locally convex $A$-modules. Moreover $\Robba_A^+$ and $\Robba_A^+ \boxtimes \P^1$ are Fr\'{e}chet spaces, $\Robba_A^-$ and $\Robba_A^- \boxtimes \P^1$ are of $A$-LB-type and $\Robba_A$ and $\Robba_A \boxtimes \P^1$ are of $A$-LF-type. For $? \in \{ +,-,\emptyset \}$ we have that 
\[
\Robba_A^? = \Robba_{\qp}^? \widehat{\otimes}_{\qp} A,
\]
\[ \Robba_A^? \boxtimes \P^1 = (\Robba_{\qp}^? \boxtimes \P^1) \widehat{\otimes}_{\qp} A. \]
\end{lemm}

\begin{proof}
We first prove the statement for $\Robba_A^?$, $? \in \{ +, -, \emptyset \}$. By \cite[lemme 1.3(i)]{chen2013} we have an isomorphism as Fr\'echet spaces (in the category of locally convex $\qp$-vector spaces) $\Robba_A^+ = \Robba_{\qp}^+ \widehat{\otimes}_{\qp} A$. Thus $\Robba_A^+$ is a locally convex $\qp$-vector space. Multiplication by $A$ is clearly continuous on $\Robba_{\qp}^+ \otimes_{\qp} A$ (the latter is also a locally convex $\qp$-vector space) and so by Lemma \ref{lem:hauscomp}, the completion $\Robba_A^+$ is a locally convex $A$-module. 

By example \ref{ex:afelrob}, $\Robba_A^-$ is $A$-reflexive and so by Remark \ref{rem:conrelg} we have an isomorphism $\Robba_A^- = \Robba_{\qp}^- \widehat{\otimes}_{\qp} A$. Moreover $\Robba_A^- = \varinjlim_s \Robba_A^{]0, s]} / \Robba_A^{]0, +\infty]}$ is $A$-LB-type.

Finally by definition $\Robba_A$ is $A$-LF-type and since $\Robba_A = \Robba_A^+ \oplus \Robba_A^-$ (as topological $A$-modules), it is also Hausdorff and complete. We compute 
\begin{align*}
\Robba_A &= \Robba_A^+ \oplus \Robba_A^- \\
&= (\Robba_{\qp}^+ \widehat{\otimes}_{\qp} A) \oplus (\Robba_{\qp}^- \widehat{\otimes}_{\qp} A) \\
&= (\Robba_\qp^+ \oplus \Robba_\qp^-) \widehat{\otimes}_{\qp} A \\
&= \Robba_\qp \widehat{\otimes}_{\qp} A.
\end{align*} 

We finally observe that, if $M \in \{ \Robba_A^+, \Robba_A^-, \Robba_A \}$, then $M \boxtimes \P^1$ (with its topology induced by the inclusion $M \boxtimes \P^1 \subseteq M \times M$) is topologically isomorphic to $M \times M$, the isomorphism given by $z \mapsto(\mathrm{Res}_\zp(z), \psi(\mathrm{Res}_\zp(wz)) )$ with inverse $(z_1, z_2) \mapsto (z_1, \varphi(z_2) + w(\mathrm{Res}_\zpe(z_1)))$.
\end{proof}

Let $M$ be a Hausdorff locally convex $A$-module. We define a locally convex $A$-module $\mathrm{LA}(H,M)$ of all $M$-valued \emph{locally analytic} functions on $H$. 

\begin{defi}\label{def:repla}
An $M$-index $\mathcal{I}$ on $H$ is a family of triples
\[
\{ (H_i,\phi_i, M_i) \}_{i \in I}
\]
where the $H_i$ are pairwise disjoint open subsets of $H$ which cover $H$, each $\phi_i \colon H_i \to \mathbf{Q}^d_p$ is chart\footnote{To be more precise one takes $\mathbb{H}_i$ an affinoid rigid analytic space over $\qp$ isomorphic to a closed ball, so that $\phi_i$ induces an isomorphism $\phi'_i \colon H_i \xrightarrow{\sim} \mathbb{H}_i(\qp)$.} for the manifold $H$ whose image is an affinoid ball and $M_i \hookrightarrow M$ is a continuous linear injection from an $A$-Banach space $M_i$ into $M$.  Let $\mathcal{F}_{\phi_i}(M_i)$ be the set of all functions $f \colon H_i \to M_i$ such that $f \circ \phi_i^{-1}$ is an $M_i$-valued holomorphic function on the affinoid ball $\phi_i(H_i)$. Note that $\mathcal{F}_{\phi_i}(M_i)$ is an $A$-Banach space. We set
\[
\mathcal{F}_{\mathcal{I}}(M) := \prod_{i \in I} \mathcal{F}_{\phi_i}(M_i),
\]
where $\mathcal{F}_{\mathcal{I}}(M)$ is equipped with the direct product topology (in particular it is a locally convex $A$-module).
We then define
\[
\mathrm{LA}(H,M) := \varinjlim_{\mathcal{I}} \mathcal{F}_{\mathcal{I}}(M)
\]
equipped with the $\qp$-locally convex inductive limit topology.  
\end{defi}

\begin{rema}
In Definition \ref{def:repla}, in order to see that $\mathrm{LA}(H,M)$ is a locally convex $A$-module, one needs to check that multiplication by $A$ on $\mathrm{LA}(H,M)$ is continuous. Indeed since, $\cdot \colon A \times M \to M$ is continuous, then so is $A \times \mathcal{F}_{\phi_i}(M_i) \to \mathcal{F}_{\phi_i}(M_i)$. This implies that multiplication $B_{\mathcal{I}} \colon A \times \mathcal{F}_{\mathcal{I}}(M) \to \mathcal{F}_{\mathcal{I}}(M)$ is continuous. Denote by $B \colon A \times \mathrm{LA}(H,M) \to \mathrm{LA}(H,M)$, the multiplication by $A$ on $\mathrm{LA}(H,M)$. The continuity of $B$ follows from the continuity of the $B_{\mathcal{I}}$ (cf. the 3rd paragraph of the proof of Lemma \ref{lem:LAgoodfu} where a similar problem is proved). 
\end{rema}

\begin{lemm} \label{lem:schdua}
Let $(V, ||\cdot||)$ be a $\qp$-Banach space and let $V'_{0} \subset V'_{\qp,b}$ be the unit ball. Given a constant $C$ and a vector $v \in V$, if $|l(v)|_p \leq C$ for all $l \in V'_{0}$ then $||v|| \leq C$. 
\end{lemm}

\begin{proof}
This is a direct consequence of the Hahn-Banach theorem. Applying \cite[Proposition 9.2]{schneidnon} with $U := V$, $q := ||\cdot||$, $U_{o}:= \qp v$ and $|l_{o}(v)|_p \geq||v||$ for all $a \in \qp$ we obtain a linear form $l \in V'_{0}$ and $|l(v)|_p \geq ||v||$. 
\end{proof}

To kick-start our study of locally convex $A$-modules and their relationship to $\mathrm{Rep}_{A}^{\mathrm{la}}(H)$ (cf. Definition \ref{def:lareph}) we need to know that $\mathrm{LA}(H,A)$ is well-behaved. F\'eaux states explicitly in his thesis (cf. \cite{feauxtob}) that he does not know if $\mathrm{LA}(H,A)$ is complete. The completeness of $\mathrm{LA}(H,A)$ has since become somewhat of a folklore conjecture:

\begin{conj} \label{conj:compLA}
If $H$ is a compact locally $\qp$-analytic group then $\mathrm{LA}(H,A)$ is complete. 
\end{conj}

\begin{rema} \label{rem:fonlbs}
If Conjecture \ref{conj:compLA} is true then one can show that $\mathrm{LA}(H, A) \cong \mathrm{LA}(H, \qp) \widehat{\otimes}_{\qp, \pi} A \cong \mathrm{LA}(H, \qp) \widehat{\otimes}_{\qp, \iota} A$, cf. Remark \ref{rem:conrelg}. If $H$ is of dimension 1 then Conjecture \ref{conj:compLA} is true, cf. lemma \ref{lem:magdimo}. 
\end{rema}

Although we are unable to prove Conjecture \ref{conj:compLA} we show that $\mathrm{LA}(H,A)$ is sufficiently well-behaved for applications.

\begin{defi} \label{def:goodspat}
Let $R \in \{ \qp, A \}$. We call a Hausdorff $R$-LB-type $V = \varinjlim_n V_n$, $R$-\emph{regular} if, for every bounded subset $B$ of $V$, there exists an $n$ such that $V_n$ contains $B$ and $B$ is bounded in $V_n$. 
\end{defi}

\begin{rema}
Let $R \in \{\qp,A \}$. By \cite[Proposition 1.1.10 and 1.1.11]{Emertonred}, a Hausdorff semi-complete $R$-LB-type is $R$-regular. 
\end{rema}

Before we state our next result we need some notation. Via Mahler expansions (cf. \cite[III. Th\'eor\`eme 1.2.4]{Lazardgrp}), the set of continuous functions from $\mathbf{Z}_p^d$ to $\qp$ can be viewed as the space of all series
\[
\sum_{\alpha \in \mathbf{N}^d} c_{\alpha} \binom{x}{\alpha}
\]
with $c_{\alpha} \in \qp$ such that $\lvert c_{\alpha} \lvert \to 0$ as $\lvert \alpha \lvert \to 0$. Here as usual
\[
\binom{x}{\alpha} := \binom{x_1}{\alpha_1} \cdots \binom{x_d}{\alpha_d}
\]
and
\[
\lvert \alpha \lvert := \sum_{i=1}^{d}\alpha_i
\]
for $x = (x_1,\ldots,x_d) \in \mathbf{Z}_p^d$ and multi-indices $\alpha = (\alpha_1,\ldots,\alpha_d) \in \mathbf{N}^d$. We are grateful for G. Dospinescu for supplying us with the idea for the following lemma.

\begin{lemm} \label{lem:LAgoodfu}
If $H$ is a compact locally $\qp$-analytic group then $\mathrm{LA}(H,A)$ is $A$-regular. 
\end{lemm}

\begin{proof}
By choosing a covering of $H$ by a finite number of open subsets isomorphic (as locally $\qp$-analytic manifolds) to $\Z_p^d$ for some $d \in \N$, we suppose that $H = \Z_p^d$, cf. \cite[Corollary 8.34]{promangrp}. For each $h \in \N$ and $f \in \mathrm{LA}(H, A)$, write $f = \sum_{n \in \N^d} a_n(f) {x \choose \alpha}$, $a_n(f) \in A$, the Mahler expansion of the continuous function $f$. By Amice's theorem (cf. \cite[III. Th\'eor\`eme 1.3.8]{Lazardgrp} or \cite[Th\'eor\`eme I.4.7]{colmfonp}), we have
\[
\mathrm{LA}_h(\mathbf{Z}_p^d,\qp) = \underset{n \in \mathbf{N}^d}{\widehat{\bigoplus}} \qp \cdot k_{n,h}\binom{x}{n}
\] 
where $k_{n,h}:= \lfloor p^{-h} n_1\rfloor! \ldots \lfloor p^{-h} n_d\rfloor!$, and where $\mathrm{LA}_h(\mathbf{Z}_p^d,\qp)$ denotes the space of functions which are analytic on every ball of poly-radius $(h, \hdots, h)$. One also obtains
\[
\mathrm{LA}_h(\mathbf{Z}_p^d,A) = \underset{n \in \mathbf{N}^d}{\widehat{\bigoplus}} A \cdot k_{n,h}\binom{x}{n},
\] which shows in particular that each $\mathrm{LA}_h(\mathbf{Z}_p^d,A)$ is complete. By definition, we also have
\[ \mathrm{LA}(\mathbf{Z}_p^d,A) = \varinjlim_{h \in \N} \mathrm{LA}_h(\mathbf{Z}_p^d,A). \] We denote the norms on $\mathrm{LA}_h(H,\qp)$ and $\mathrm{LA}_h(H,A)$ by $|\cdot|_{h,\qp}$ and $|\cdot|_{h,A}$, respectively and $v_{h,\qp}$ and $v_{h,A}$ their respective valuations.

Let $T \subset \mathrm{LA}(H,A)$ be a bounded subset and consider the $\qp$-bilinear application
\begin{align*}
B \colon A'_{\qp,b} \times \mathrm{LA}(H,A) &\to \mathrm{LA}(H,\qp) \\
(l,f) &\mapsto l \circ f.
\end{align*} Note that the Mahler coefficients of $l \circ f$ are then just given by $a_n(l \circ f) = l(a_n(f))$.

We show that the map $B$ above is continuous. Note (as is easily seen by looking at Mahler expansions) that restriction gives, for every $h \in \N$, $\qp$-bilinear forms
\begin{align*}
B_h \colon A'_{\qp,b} \times \mathrm{LA}_h(H,A) &\to \mathrm{LA}_h(H,\qp) \\
(l,f) &\mapsto l \circ f.
\end{align*}
We claim that $B$ is continuous if and only if $B_h$ is continuous for every $h \in \N$. Indeed, by the definition of the locally convex final topology, the topology of $\mathrm{LA}(H,\qp)$ is defined by the family of all lattices $L \subseteq \mathrm{LA}(H,\qp)$ such that $L \cap \mathrm{LA}_h(H,\qp)$ is open (in $\mathrm{LA}_h(H,\qp)$) for all $h \in \N$. So let $L$ be such a lattice. We want to show that $L' := B^{-1}(L)$ is open (in $A'_{\qp,b} \times \mathrm{LA}(H,A)$). Note first that $L'$ is a lattice: if $a \in \zp$ and if $(x, y) \in L'$ then $B(a(x, y)) = a^2 B(x, y) \in L$ so that $a(x, y) \in L'$ and, if $(x, y) \in A'_{\qp,b} \times \mathrm{LA}(H,A)$, then there exists an $a \in \Q_p^\times$ such that $a B(x, y) \in L$ and, writing $a = a' / p^k$, $a' \in \zp$, we also have that $(a')^2 B(x, y) \in L$, whence $a' (x, y) \in L'$, which proves that $L'$ is a lattice. So (again by definition of the locally convex final topology), $L'$ is open if and only if $L' \cap (A'_{\qp,b} \times \mathrm{LA}_h(H,A))$ is open for every $h \in \N$. Noting $B_{h,n} \colon A'_{\qp,b} \times \mathrm{LA}_h(H,A) \to \mathrm{LA}_n(H,\qp)$ for $n \geq h$ the composition of $B_h$ with the natural continuous inclusion $\mathrm{LA}_h(H,\qp) \to \mathrm{LA}_n(H,\qp)$, we have that 
\begin{eqnarray*}
L' \cap (A'_{\qp,b} \times \mathrm{LA}_h(H,A)) &=& B^{-1}(L) \cap (A'_{\qp,b} \times \mathrm{LA}_h(H,A)) \\
&=& B^{-1}(L \cap \cup_{n \geq h} \mathrm{LA}_n(H,\qp)) \cap (A'_{\qp,b} \times \mathrm{LA}_h(H,A)) \\
&=& \cup_{n \geq h} [ B^{-1}(L \cap \mathrm{LA}_n(H,\qp)) \cap (A'_{\qp,b} \times \mathrm{LA}_h(H,A)) ] \\
&=& \cup_{n \geq h} B_{h,n}^{-1}(L \cap \mathrm{LA}_n(H,\qp)),
\end{eqnarray*}
which is open if each $B_h$ (and hence $B_{h,n}$) is continuous. This proves the claim. We finally prove that each map $B_h$ is continuous. Indeed $|B_h(l,f)|_{h,\qp} = |l \circ f|_{h,\qp} \leq ||l|||f|_{h,A}$ (where $||\cdot||$ is the norm on $A'_{\qp,b}$) and so $B$ is continuous.

Now let
\[
S:= A'_{0} \times T \subset A'_{\qp,b} \times \mathrm{LA}(H,A)
\]
where $A'_{0} \subset A'_{\qp,b}$ is the unit ball. Then $S$ is bounded and so is $B(S)$, since $B$ is continuous. By \cite[Proposition 1.1.11]{Emertonred}, $\mathrm{LA}(H,\qp)$ is $\qp$-regular and so for some $h \geq 1$, $B(S)$ is contained in $\mathrm{LA}_h(H, \qp)$ and there exists a constant $C$ such that
\[ v_{h,\qp}(l \circ f) = \inf_{n \in \N^d}  v_p(l(a_n(f))) - v_p (k_{n,h}) \geq C \] for all $l \in A'_0$ and some constant $D$ (which depends on $C$). Lemma \ref{lem:schdua} now implies that $v_A(a_n(f)) - v_p (k_{n,h}) \geq C$ for all $n \in \N^d$. Now
\begin{align*}
v_A(a_n(f)) - v_p (k_{n,h+1}) &= v_A(a_n(f)) - v_p (k_{n,h}) + v_p (k_{n,h}) - v_p (k_{n,h+1}) \\
&\geq C  + v_p (k_{n,h}) - v_p (k_{n,h+1}) \xrightarrow{|n| \to +\infty} + \infty. 
\end{align*}
This implies that $f \in \mathrm{LA}_{h+1}(\mathbf{Z}_p^d,A)$ for all $f \in T$. We now compute
\begin{align*}
v_{h+1,A}(f) &= \inf_{n \in \N^d}  v_A(a_n(f)) - v_p (k_{n,h+1}) \\
&= \inf_{n \in \N^d}  v_A(a_n(f)) - v_p (k_{n,f}) + v_p (k_{n,h}) - v_p (k_{n,h+1}) \\
&\geq C
\end{align*}
since $v_p (k_{n,h}) - v_p (k_{n,h+1}) \geq 0$. This shows that $T$ is contained and bounded in $\mathrm{LA}_{h+1}(\mathbf{Z}_p^d,A)$. Thus $\mathrm{LA}(H,A)$ is $A$-regular. 
\end{proof}

\begin{rema}
For an alternative proof of the fact that the $\qp$-bilinear map $B$ in the proof of Lemma \ref{lem:LAgoodfu} is continuous, note that $B$ is the composition of the morphisms
\begin{align*}
B \colon A'_{\qp,b} \times \mathrm{LA}(H,A) &\xrightarrow{\mathrm{id} \times \alpha}  A'_{\qp,b} \times (A \widehat{\otimes}_{\qp, \pi} \mathrm{LA}(H, \qp)) \\
&\to A'_{\qp,b} \widehat{\otimes}_{\qp, \pi} (A \widehat{\otimes}_{\qp, \pi} \mathrm{LA}(H, \qp))\\
&= (A'_{\qp,b} \widehat{\otimes}_{\qp, \pi} A) \widehat{\otimes}_{\qp, \pi} \mathrm{LA}(H, \qp)\\
&\xrightarrow{\beta} \mathrm{LA}(H,\qp),
\end{align*}
where $\alpha$ is the morphism of \cite[Proposition 2.2.10]{Emertonred} (cf. the discussion immediately after loc.cit.) which is a continuous bijection. The last morphism $\beta$ is induced from the pairing of $A$ and $A'_{\qp,b}$. 
\end{rema}

\subsection{Relative locally analytic representations}

In this section we define the category $\mathrm{Rep}_{A}^{\mathrm{la}}(H)$ and we study the structure of a locally analytic representation over the relative distribution algebras, generalizing some fundamental work of Schneider and Teitelbaum to our relative setting.

The following definition is similar to \cite[D\'efinition 3.2]{BreuHellSch}.

\begin{defi} \label{def:lareph}
An object $M$ in $\mathrm{Rep}_{A}^{\mathrm{la}}(H)$ is a barrelled, Hausdorff, locally convex $A$-module equipped with a topological\footnote{We say that the $H$-action on $M$ is \emph{topological} if $H$ induces continuous endomorphisms of $M$.} $A$-linear action of $H$,  such that, for each $m \in M$, the orbit map $h \mapsto h \cdot m$ is an element in $\mathrm{LA}(H,M)$. Morphisms are given by continuous $A$-linear $H$-maps. 
\end{defi}

\begin{rema} [locally analytic induced representation] \label{rem:lains}
Let $G$ be a locally $\qp$-analytic group, $H$ a closed locally $\qp$-analytic subgroup and suppose that $G/H$ is compact. Let $M$ be an object of $\mathrm{Rep}_{A}^{\mathrm{la}}(H)$, which is Banach. Then
\[
\mathrm{Ind}_{H}^{G}(M) := \left\{f \in \mathrm{LA}(G,M) \text{ } \lvert \text{ } \forall h_i \in H_i: \text{ } f(h_1h_2) = h_2^{-1}\cdot f(h_1) \right\} 
\] 
identifies (as topological $A$-modules) with $\mathrm{LA}(G/H,M)$, cf. \cite[Satz 4.3.1]{feauxtob}. Moreover $\mathrm{Ind}_{H}^{G}(M)$ (equipped with the natural action of $G$: $(g \cdot f)(x) := f(g^{-1}x)$) is an object of $\mathrm{Rep}_{A}^{\mathrm{la}}(G)$, cf. Satz 4.1.5 in loc.cit.

To track the action of $\mathscr{D}(G)$ on $\mathrm{LA}(G/H,M)$ induced by the above isomorphism, we need to explicit this isomorphism. Any choice of a section $G/H \to G$ gives an isomorphism of locally $\qp$ analytic manifolds $G \cong G / H \times H$. This gives an isomorphism $\mathrm{LA}(G, M) \cong \mathrm{LA}(G/H \times H, M)$ and the space $\mathrm{Ind}_H^G(M) \subseteq \mathrm{LA}(G, M)$ is identified with the sub-module $\{ f \; : \; f(\overline{g}, h) = h^{-1} \cdot f(\overline{g}, 1) , \; \overline{g} \in G / H, h \in H\}$ of $\mathrm{LA}(G/H \times H, M)$. On the other hand, the composition
\[ \mathrm{LA}(H, M) \to \mathrm{LA}(H, \mathrm{End}(M)) \times \mathrm{LA}(G / H, M) \to \mathrm{LA}(G/H \times H, M)  \]
\[ \tilde{f} \mapsto (\rho^{-1}, \tilde{f}) \mapsto [(\overline{g}, h) \mapsto h^{-1} \cdot \tilde{f}(\overline{g}) ], \] where we have noted $\rho: G \to \mathrm{GL}(M)$ the representation of $H$ on $M$, induces an isomorphism between $\mathrm{LA}(G / H, M)$ and the image of $\mathrm{Ind}_H^G(M)$ in $\mathrm{LA}(G, M)$.
\end{rema} 

It will turn out that every complete object of $\mathrm{Rep}_{A}^{\mathrm{la}}(H)$ carries a structure of a $\mathscr{D}(H, A)$-module. The following lemma is essentially \cite[Proposition 1.3]{Tayhomcoh}.

\begin{lemm} \label{basechange}
Let $M$ be a locally convex $\qp$-module and let $N$ be a locally convex $A$-module. Then $\tilde{f}(a \otimes x) = a f(x)$ defines a topological $A$-linear isomorphism 
\[
\mathcal{L}_{\qp, b}(M, N) \xrightarrow{\sim} \mathcal{L}_{A, b}(M \otimes_{\qp, \pi} A, N) \\
\]
\end{lemm}

\begin{proof}
For $f \in \mathcal{L}_{\qp, b}(M, N)$, the map $\tilde{f}$ is given by the composition of the continuous map $M \otimes_{\qp, \pi} A \to N \otimes_{\qp, \pi} A$ induced by $a \otimes x \mapsto a \otimes f(x)$ and the (continuous) map $N \otimes_{\qp, \pi} A \to N$ induced by the $A$-module structure on $N$, so it is well defined. The inverse of $f \mapsto \tilde{f}$ is given by $g \mapsto g_0$, where $g_0(x) = g(x \otimes 1)$. This shows that the map of the statement induces an $A$-linear bijection.

The fact that it's a topological isomorphism follows from \cite[Corollary 17.5(iii)]{schneidnon}: the map $\alpha \colon M \to M \otimes_{\qp, \pi} A$ defined by $m \mapsto m \otimes 1$ induces an homeomorphism between $M$ and it's image, and so the image of any bounded subset of $M$ is bounded in $M \otimes_{\qp, \pi} A$ and, conversely, the inverse image (i.e its intersection with $M \subseteq M \otimes_{\qp, \pi} A$) of any bounded subset of $M \otimes_{\qp, \pi} A$ is bounded in $M$. To show that $f \mapsto \tilde{f}$ is continuous, it is enough to show that, if $B \subseteq M \otimes_{\qp, \pi} A$ is bounded and if $U \subseteq N$ is an open set, then the inverse image of the set $\{ f \colon M \otimes_{\qp, \pi} A \to N \; | \; f(B) \subseteq U \}$ is open in $\mathcal{L}_{\qp, b}(M, N)$, but this inverse image is nothing but $\{ f \colon M \to N \; | \; f(\alpha^{-1}(B)) \subseteq U)\}$, which is open since $\alpha^{-1}(B)$ is bounded in $M$. Conversely, if $B \subseteq M$ is bounded and $U \subseteq N$ is open, then the inverse image of $\{ f \colon M \to N \; | \; f(B) \subseteq U\}$ by the map $g \mapsto g_0$ is $\{ g \colon M \otimes_{\qp, \pi} A \to N \; | \; g(\alpha(B)) \subseteq U \}$, which is open since $\alpha(B)$ is bounded in $M \otimes_{\qp, \pi} A$, and shows that the inverse map is continuous. This completes the proof.
\end{proof}

\begin{coro} \label{basechangecoro}
In the setting of Lemma \ref{basechange}, if in addition $N$ is Hausdorff and complete, then \[
\mathcal{L}_{\qp, b}(M, N) \xrightarrow{\sim} \mathcal{L}_{A, b}(M \widehat{\otimes}_{\qp, \pi} A, N) \\
\]
\end{coro}

\begin{proof}
This is an immediate consequence of Lemma \ref{basechange} above and the universal property of Hausdorff completion, cf. \cite[Corollary 7.7]{schneidnon}. 
\end{proof}

\begin{rema}
Since \cite[Corollary 17.5(iii)]{schneidnon} holds for projective as well as inductive tensor product, both Lemma \ref{basechange} and Corollary \ref{basechangecoro} hold with inductive tensor product replacing the projective one.
\end{rema}

\begin{defi} \label{defaltdisalf}
We define the space of distributions on $H$ with values in $A$ as the strong dual of $\mathrm{LA}(H,A)'_b$ (cf. Definition \ref{def:stdua}) \[ \mathscr{D}(H, A) := \mathrm{LA}(H, A)'_b. \]
\end{defi}

\begin{lemma} \label{homvarinj}
Let $R \in \{\qp,A \}$. If $V = \varinjlim_n V_n$ is $R$-regular (cf. Definition \ref{def:goodspat}) then for any locally convex $R$-module, the natural map $\mathcal{L}_{R, b}(V, W) \to \varprojlim_n \mathcal{L}_{R, b}(V_n, W)$ is a topological isomorphism which is $R$-linear.
\end{lemma}

\begin{proof}
This is the same proof as \cite[Proposition 1.1.22]{Emertonred}. The crucial point in loc.cit. is that a Hausdorff semi-complete $R$-LB-type is $R$-regular. 
\end{proof}

\begin{lemm} \label{distbasechange}
Let $H$ is a compact locally $\qp$-analytic group. We have an isomorphism of locally convex $A$-modules \[ \mathscr{D}(H, A) = \mathscr{D}(H, \qp) \widehat{\otimes}_{\qp, \pi} A. \] In particular, $\mathscr{D}(H, A)$ is an $A$-Fr\'echet space.
\end{lemm} 


\begin{proof}
By \cite[Proposition 20.9]{schneidnon} we have \[ \mathscr{D}(H, \qp) \widehat{\otimes}_{\qp, \pi} A = \mathcal{L}_{\qp, b}(D(H, \qp)'_{\qp, b}, A). \] We conclude by observing that (we use the notation from the proof of Lemma \ref{lem:LAgoodfu})
\begin{eqnarray*}
\mathcal{L}_{\qp, b}(D(H, \qp)'_{\qp, b}, A) &\overset{(i)}{=}& \mathcal{L}_{\qp, b}(\mathrm{LA}(H, \qp), A) \\
&\overset{(ii)}{=}& \varprojlim_n \mathcal{L}_{\qp, b}(\mathrm{LA}_n(H , \qp), A) \\
&\overset{(iii)}{=}& \varprojlim_n \mathcal{L}_{A, b}(\mathrm{LA}_n(H , \qp) \widehat{\otimes}_{\qp, \pi} A, A) \\
&\overset{(iv)}{=}& \varprojlim_n \mathcal{L}_{A, b}(\mathrm{LA}_n(H , A), A) \\
&\overset{(v)}{=}& \mathcal{L}_{A, b}(\mathrm{LA}(H , A), A) \\
&\overset{(vi)}{=}& \mathscr{D}(H, A),
\end{eqnarray*}
where (i) follows by reflexivity of $\mathrm{LA}(H, \qp)$ (cf. \cite[Lemma 2.1 and Theorem 1.1]{tetschlad}), (ii) follows from Lemma \ref{homvarinj}, (iii) follows from Corollary \ref{basechangecoro}, (iv) is an immediate consequence of the definition of $\mathrm{LA}_n(H, A)$, (v) is a consequence of Lemmas \ref{lem:LAgoodfu} and \ref{homvarinj} and (vi) is by definition. The last assertion follows since $\mathscr{D}(H, \qp)$ is Fr\'echet and $A$ is Banach so their completed projective tensor product is Fr\'echet. This completes the proof.
\end{proof}

\begin{rema} \label{rem:conrelg}
For $H$ a compact locally $\qp$-analytic group, the natural morphism \[ \alpha \colon \mathrm{LA}(H, A) \to \mathrm{LA}(H, \qp) \widehat{\otimes}_{\qp, \pi} A \] is (cf. the discussion immediately after \cite[Proposition 2.2.10]{Emertonred}) a continuous bijection. We do not know whether it is actually a topological isomorphism, cf. Conjecture \ref{conj:compLA}. By \cite[Theorem 2]{compprojten}, this is the case if $\mathrm{LA}(H, A)$ is complete (note that $\mathrm{LA}_h(H,A)$ has a Schauder basis, by Amice's theorem, so it has the approximation property). We claim that $\alpha$ is a topological isomorphism if and only if $\mathrm{LA}(H, A)$ is a reflexive $A$-module. Indeed, this follows from the following equalities:
\begin{eqnarray*}
\mathcal{L}_{A, b}(\mathscr{D}(H, A), A) &=& \mathcal{L}_{A, b}(\mathscr{D}(H, \qp) \widehat{\otimes}_{\qp, \pi} A, A) \\
&=& \mathcal{L}_{\qp, b}(\mathscr{D}(H, \qp), A) \\
&=& \mathcal{L}_{\qp, b}(\mathscr{D}(H, \qp), \qp) \widehat{\otimes}_{\qp, \pi} A  \\
&=&  \mathrm{LA}(H, \qp) \widehat{\otimes}_{\qp, \pi} A. 
\end{eqnarray*}
\end{rema}

\begin{lemm}\label{lem:decodist}
Let $\{ H_i\}_{i\in I}$ be pairwise disjoint compact open subsets which cover $H$. Then there is a ($A$-linear) topological isomorphism
\[
\mathscr{D}(H,A) = \bigoplus_{i} \mathscr{D}(H_i,A).
\]
Moreover $\mathscr{D}(H,A)$ is complete and Hausdorff. 
\end{lemm}

\begin{proof}
We have a topological isomorphism
\[
\mathrm{LA}(H,A) = \prod_{i}\mathrm{LA}(H_i,A).
\]
The claim now follows from the fact that there is a topological isomorphism 
\[
(\prod_{i}\mathrm{LA}(H_i,A) )'_b = \bigoplus_{i} \mathrm{LA}(H_i,A)'_b
\]
To see this last fact, one repeats the same proof for \cite[Proposition 9.11]{schneidnon}. Finally $\mathscr{D}(H,A)$ is complete and Hausdorff follows from \cite[Corollary 5.4 and Lemma 7.8]{schneidnon}.
\end{proof}

\begin{lemm}\label{lem:distrbcompl}
Let $H$ be a locally $\qp$-analytic group. We have an isomorphism of locally convex $A$-modules 
\[ \mathscr{D}(H, A) = \mathscr{D}(H, \qp) \widehat{\otimes}_{\qp, \iota} A. \] 
\end{lemm} 

\begin{proof}
This is an immediate consequence of Lemmas \ref{distbasechange} and \ref{lem:decodist}. Let $\{H_i \}_{i \in I}$ be pairwise disjoint compact open subsets which cover $H$.

$$
\begin{tikzcd} [row sep = large, column sep = large]
\bigoplus_{i} \left[ \mathscr{D}(H_i,\qp) \otimes_{\qp,\iota} A \right] \arrow[r, "\sim"] \arrow[d] &
\mathscr{D}(H,\qp) \otimes_{\qp,\iota} A \arrow[d]  \\
\bigoplus_{i} \left[ \mathscr{D}(H_i,\qp) \widehat{\otimes}_{\qp,\iota} A \right] \arrow[r] &
\mathscr{D}(H,\qp) \widehat{\otimes}_{\qp,\iota} A.
\end{tikzcd}
$$
Now \cite[Lemma 1.1.30]{Emertonred} implies that the top horizontal arrow is a topological isomorphism. By definition the right vertical arrow is a topological embedding (since $\mathscr{D}(H,\qp)$ is Hausdorff, so is $\mathscr{D}(H,\qp) \otimes_{\qp} A$, cf. \cite[Corollary 17.5(i)]{schneidnon}) that identifies its target with the completion of its source. We will show that the same is true for the left vertical arrow which will imply that the bottom horizontal arrow is a topological isomorphism, as required. Since the composite of the top horizontal arrow and the right vertical arrow is a topological embedding, the same is true for the left vertical arrow. It clearly has dense image and the target ($= \mathscr{D}(H,A)$) is complete. This completes the proof.   
\end{proof}

\begin{rema}
In the setting of Lemma \ref{lem:distrbcompl}, \cite[I.1.3 Proposition 6]{groprodten} shows that 
\[ \mathscr{D}(H, A) = \mathscr{D}(H, \qp) \widehat{\otimes}_{\qp, \pi} A \]
is a topological isomorphism. 
\end{rema}

The following is a relative version of the integration map constructed in \cite[Theorem 2.2]{tetschlad}. 

\begin{lemma} \label{locandist}
Let $H$ be a locally $\qp$-analytic group and let $M$ be a complete Hausdorff locally convex $A$-module. There is a unique $A$-linear map \[ I: \mathrm{LA}(H, M) \to \mathrm{Hom}_{A, \mathrm{cont}}(\mathscr{D}(H, A), M), \] satisfying $ I(\phi)(\delta_{h_i} \otimes 1) = \phi(h_i)$ for all $\phi \in \mathrm{LA}(H_i, M)$ and all $h_i \in H_i$. Here $\{H_i\}_{i \in I}$ are pairwise disjoint compact open subsets covering $H$, and $\delta_{h_i} \in \mathscr{D}(H_i,\qp)$ such that $\delta_{h_i}(f) := f(h_i)$ for all $f \in \mathrm{LA}(H_i,\qp)$.

 Moreover, if $M$ is A-$LB$-type (cf. Definition \ref{def:frech}) then this map is a bijection.
\end{lemma}

\begin{proof}
By \cite[Theorem 2.2]{tetschlad}  (cf. also the comment immediately after its proof), one has a unique map \[ I_\qp \colon \mathrm{LA}(H, M) \to \mathrm{Hom}_{\qp, \mathrm{cont}}(\mathscr{D}(H, \qp), M), \] satisfying $I_\qp(\phi)(\delta_h \otimes 1) = \phi(h)$, $h \in H$, and which is bijective if $M$ is of $\qp$-LB-type. Note that this map is clearly $A$-linear.

By Lemma \ref{lem:decodist}, one reduces to show the result for $H$ compact. So assume $H$ is compact. By Corollary \ref{basechangecoro} (where we forget the topologies), since $M$ is Hausdorff and complete, there is an $A$-linear bijection \[r \colon \mathrm{Hom}_{\qp, \mathrm{cont}}(\mathscr{D}(H, \qp), M) \xrightarrow{\sim} \mathrm{Hom}_{A, \mathrm{cont}}(\mathscr{D}(H, \qp) \widehat{\otimes}_{\qp, \pi} A, M). \] Moreover, Lemma \ref{distbasechange} gives an isomorphism
\[s \colon \mathrm{Hom}_{A, \mathrm{cont}}(\mathscr{D}(H, \qp) \widehat{\otimes}_{\qp, \pi} A, M) \xrightarrow{\sim} \mathrm{Hom}_{A, \mathrm{cont}}(\mathscr{D}(H, A), M). \]
The composition of all these maps ($s\circ r \circ I_\qp$) gives the desired map \[ I: \mathrm{LA}(H, M) \to \mathrm{Hom}_{A, \mathrm{cont}}(\mathscr{D}(H, A), M). \] The result follows.
\end{proof}

Let $\mathrm{Rep}_{A}^{\mathrm{la},LB}(H) \subseteq \mathrm{Rep}_{A}^{\mathrm{la}}(H)$ be the full subcategory consisting of spaces which are $A$-LB-type and complete. As a result we obtain the following corollary. 

\begin{coro} \label{cor:imded}
The category of $\mathrm{Rep}_{A}^{\mathrm{la},LB}(H)$ is equivalent to the category of complete, Hausdorff locally convex $A$-modules which are of $A$-LB-type equipped with a separately continuous $\mathscr{D}(H,A)$-action (more precisely the module structure morphism $\mathscr{D}(H,A) \times M \to M$ is $A$-bilinear and separately continuous) with morphisms all continuous $\mathscr{D}(H,A)$-linear maps. 
\end{coro}

\begin{proof}
This is an immediate consequence of Lemma \ref{locandist}. 
\end{proof}

\subsection{Locally analytic cohomology and Shapiro's lemma}

In this section we prove Shapiro's Lemma for a relative version of the cohomology theory developed by Kohlhaase in \cite{kohl2011}. We should warn the reader that Lazard's definition of locally analytic cohomology of a locally $\qp$-analytic group via analytic cochains, cf. \cite[Chapitre V, \S 2.3]{Lazardgrp} (or \cite{modlaztr} for a modern treatment), does not always coincide with the cohomology groups defined in \cite{kohl2011}. Futhermore the cohomology groups defined by Kohlhasse are finer than that of Lazard, in the sense that they themeselves carry a locally convex topology. In what follows however, we will ignore this extra structure. Let us first explain the setup. Let $H$ be a locally $\qp$-analytic group (for applications $H$ will be a closed locally $\qp$-analytic subgroup of $\GL_2(\qp)$). We will follow closely the treatment in \cite{kohl2011}, albeit in a relative setting. In particular we are able to reduce many of the arguments to the case considered in loc.cit. The key is lemma \ref{lem:distrbcompl}. 

\begin{defi}\label{def:lacatlur}
Let $\mathscr{G}_{H,A}$ denote the category of complete Hausdorff locally convex $A$-modules with the structure of a separately continuous $A$-linear $\mathscr{D}(H,A)$-module, taking as morphisms all continuous $\mathscr{D}(H,A)$-linear maps. More precisely we demand that the module structure morphism
\[
\mathscr{D}(H,A) \times M \to M
\] 
is $A$-bilinear and separately continuous.
\end{defi}

\begin{rema}
Alternatively, one sees that $\mathscr{G}_{H,A}$ can be also defined as the category of complete Hausdorff locally convex $\qp$-modules with the structure of a separately continuous $\mathscr{D}(H,A)$-module, taking as morphisms all continuous $\mathscr{D}(H,A)$-linear maps.
\end{rema}

We begin by showing that the convolution product on $(\mathscr{D}(H,\qp, *)$, cf. \cite[\S 2]{tetschlad}, extends to $\mathscr{D}(H,A)$. Indeed by Lemma \ref{lem:distrbcompl} we have an isomorphism of locally convex $A$-modules
\[ \mathscr{D}(H, A) = \mathscr{D}(H, \qp) \widehat{\otimes}_{\qp, \iota} A. \]
We define for $h_1,h_2 \in H$ an $A$-bilinear, separately continuous map
\begin{align*}
*_A \colon (\mathscr{D}(H,\qp) \otimes_{\qp,\iota} A) \times (\mathscr{D}(H,\qp) \otimes_{\qp,\iota} A) &\to (\mathscr{D}(H,\qp) \otimes_{\qp,\iota} A)\\
(\delta_{h_1} \otimes 1) \times (\delta_{h_2} \otimes 1) &\mapsto (\delta_{h_1}*\delta_{h_2} \otimes 1).
\end{align*}
Since the Dirac distributions $\delta_h$ for $h\in H$ are dense in $\mathscr{D}(H,\qp)$, cf. \cite[Lemma 3.1]{tetschlad}, $*_A$ is well defined. Note that $*_A$ is separately continuous since $*$ is separately continuous, cf. \cite[Proposition 2.3]{tetschlad}. It is clear that $*_A$ extends uniquely to an $A$-bilinear, separately continuous map (which we denote by $*$, abusing notation)
\[
* \colon \mathscr{D}(H,A) \times \mathscr{D}(H,A) \to \mathscr{D}(H,A).
\]
The following lemma summarizes the above discussion. 

\begin{lemm}\label{lem:algprodf}
$(\mathscr{D}(H,A), *)$ is an associative $A$-algebra with $\delta_1 \otimes 1$ ($1 \in H$ is the unit element) as the unit element. Futhermore the convolution $*$ is separately continuous and $A$-bilinear. 
\end{lemm}

As a consequence of Lemma \ref{lem:algprodf} and the fact that $\mathscr{D}(H,A)$ is complete and Hausdorff (cf. Lemma \ref{lem:decodist}) the convolution product $(\mathscr{D}(H,A), *)$ induces a unique continuous $A$-linear map
\begin{equation}\label{eq:whjg}
\mathscr{D}(H,A) \widehat{\otimes}_{A,\iota} \mathscr{D}(H,A) \to \mathscr{D}(H,A).
\end{equation}

We now endow $\mathscr{G}_{H,A}$ (and $\mathrm{LCS}_A$) with the structure of an exact category. A sequence in $\mathscr{G}_{H,A}$ (or $\mathrm{LCS}_A$)
\[
\cdots \to M_{i-1} \xrightarrow{\alpha_{i-1}} M_i \xrightarrow{\alpha_i} M_{i+1} \xrightarrow{\alpha_{i+1}} \cdots
\] 
is called \emph{s-exact} if $M_i = K_i \oplus L_i$ (as topological $A$-modules) where $K_i := \mathrm{ker}(\alpha_i)$ and $\alpha_i$ induces an isomorphism (as topological $A$-modules) between $L_i$ and $K_{i+1}$. 

\begin{rema}
A sequence in $\mathscr{G}_{H,A}$
\[
0 \to M \to N \to P \to 0
\]
is s-exact iff it is split in the category of topological $A$-modules. 
\end{rema}

\begin{defi}
An object $P$ of $\mathscr{G}_{H,A}$ is called \emph{s-projective} if the functor $\mathrm{Hom}_{\mathscr{G}_{H,A}}(P, \cdot)$ transforms all short s-exact sequences
\[
0 \to M_1 \to M_2 \to M_3 \to 0
\]
in $\mathscr{G}_{H,A}$ into exact sequences of $A$-modules. 
\end{defi}

\begin{lemm}
If $M$ is any complete Hausdorff locally convex $A$-module, then
\[
\mathscr{D}(H,A) \widehat{\otimes}_{A,\iota} M
\]
is an object of $\mathscr{G}_{H,A}$.
\end{lemm} 

\begin{proof}
Indeed $\mathscr{D}(H,A) \widehat{\otimes}_{A,\iota} M$, being Hausdorff and complete by definition, it suffices to remark that by tensoring the identity map on $M$ with \eqref{eq:whjg} we obtain a continuous $A$-linear map
\[
\mathscr{D}(H,A) \widehat{\otimes}_{A,\iota} (\mathscr{D}(H,A) \widehat{\otimes}_{A,\iota} M) \cong (\mathscr{D}(H,A) \widehat{\otimes}_{A,\iota} \mathscr{D}(H,A)) \widehat{\otimes}_{A,\iota} M \to \mathscr{D}(H,A) \widehat{\otimes}_{A,\iota} M.
\]
\end{proof}

We'll call an object of the form $\mathscr{D}(H,A) \widehat{\otimes}_{A,\iota} M$ (for $M$ any complete Hausdorff locally convex $A$-module) in $\mathscr{G}_{H,A}$ \emph{s-free}. As one expects, s-projective modules can be viewed as direct summands of an s-free module.

\begin{lemm}\label{lem:projfres}
An object $P$ of $\mathscr{G}_{H,A}$ is s-projective if and only if it is a direct summand (in $\mathscr{G}_{H,A}$) of an s-free module. 
\end{lemm} 

\begin{proof}
First note that for any complete Hausdorff locally convex $A$-module $M$ and any object $N$ of $\mathscr{G}_{H,A}$ there is a natural continuous $A$-linear bijection
\[
\mathrm{Hom}_{\mathscr{G}_{H,A}}(\mathscr{D}(H,A) \widehat{\otimes}_{A,\iota} M, N) \to \mathrm{Hom}_{A,\mathrm{cont}}(M,N).
\]
This is the same proof as the first paragraph of the proof of Lemma \ref{basechange} (with $A$ replaced by $\mathscr{D}(H,A)$). The result now follows from \cite[Proposition 1.4]{Tayhomcoh}. 
\end{proof}

We will be interested in considering the cohomology of objects in $\mathscr{G}_{H,A}$ and so we need the notion of a resolution.

\begin{defi}
If $M$ is an object of $\mathscr{G}_{H,A}$ then by an \emph{s-projective s-resolution} of $M$ we mean an s-exact sequence
\[
\cdots \to X_1 \xrightarrow{d_1} X_0 \xrightarrow{d_0} M
\]
in $\mathscr{G}_{H,A}$ in which all objects $X_i$ are s-projective. 
\end{defi}

For an object $M$ of $\mathscr{G}_{H,A}$ let $B_{-1}(H,M) := M$ and for $q \geq 0$ let
\[
B_q(H,M) := \mathscr{D}(H,A) \widehat{\otimes}_{A,\iota} B_{q-1}(H,M)
\] 
with its structure of an s-free module. For $q \geq 0$ define
\[
d_q(\delta_0 \otimes \cdots \otimes \delta_q \otimes m) := \sum_{i=0}^{q-1} (-1)^i \delta_0 \otimes \ldots \delta_i\delta_{i+1} \otimes \cdots \otimes \delta_q \otimes m +(-1)^q\delta_0 \otimes \cdots \otimes \delta_{q-1} \otimes \delta_qm.
\] 

\begin{lemm}
For any object $M$ of $\mathscr{G}_{H,A}$ the sequence $(B_q(H,M), d_q)_{q \geq 0}$ is an s-projective s-resolution of $M$ in $\mathscr{G}_{H,A}$. 
\end{lemm}

\begin{proof}
This is an immediate consequence of \cite[Proposition 2.4]{kohl2011}. The critical point is that by Lemma \ref{lem:distrbcompl}
\[
B_q(H,M) = \mathscr{D}(H,A) \widehat{\otimes}_{A,\iota} B_{q-1}(H,M) = \mathscr{D}(H,\qp) \widehat{\otimes}_{\qp,\iota} B_{q-1}(H,M)
\]
and so the $B_q(H,M)$ defined above coincide with the ones defined in \cite{kohl2011}. 
\end{proof}

\begin{defi}\label{def:goodextnoab}
If $M$ and $N$ are objects of $\mathscr{G}_{H,A}$ we define $\Ext^{q}_{\mathscr{G}_{H,A}}(M, N)$ to be the $q$th cohomology group of the complex $\mathrm{Hom}_{\mathscr{G}_{H,A}}(B_\bullet(H,M),N)$ for any $q \geq 0$. 
\end{defi}

\begin{rema}
For any two objects $M$ and $N$ of $\mathscr{G}_{H,A}$, as usual $\Ext^{1}_{\mathscr{G}_{H,A}}(M, N)$ is the set of equivalence classes of s-exact sequences
\[
0 \to N \to E \to M \to 0
\]
with objects $E$ of $\mathscr{G}_{H,A}$.
\end{rema}

As in the setting of \cite{kohl2011}, one can identify the categories of separately continuous left and right $\mathscr{D}(H,A)$-modules. If $M$ and $N$ are objects of $\mathscr{G}_{H,A}$, $M$ a right module, we define $M \widetilde{\otimes}_{\mathscr{D}(H,A),\iota} N$ to be the quotient of $M \widehat{\otimes}_{A,\iota} N$ by the image of the natural map 
\begin{align*}
M \widehat{\otimes}_{A,\iota} \mathscr{D}(H,A) \widehat{\otimes}_{A,\iota} N &\to M \widehat{\otimes}_{A,\iota} N \\
m \otimes \delta \otimes n &\mapsto m\delta \otimes n - m \otimes \delta n,
\end{align*}
where $m \in M$, $n \in N$ and $\delta \in \mathscr{D}(H,A)$. The induced topology is the quotient topology. 

\begin{lemm}\label{lem:tstenpr}
For any complete Hausdorff locally convex $A$-module $M$ and any object $N$ of $\mathscr{G}_{H,A}$ there is a natural $A$-linear topological isomorphism
\[
(M \widehat{\otimes}_{A,\iota} \mathscr{D}(H,A)) \widetilde{\otimes}_{\mathscr{D}(H,A), \iota} N \cong M \widehat{\otimes}_{A,\iota}N.
\]
If the object $P$ of $\mathscr{G}_{H,A}$ is s-projective then the functor $P \widetilde{\otimes}_{\mathscr{D}(H,A),\iota} (\cdot)$ takes s-exact sequences in $\mathscr{G}_{H,A}$ to exact sequences of $A$-modules. If $P$ is s-free this functor takes s-exact sequences in $\mathscr{G}_{H,A}$ to s-exact sequences in $\mathrm{LCS}_A$.  
\end{lemm}

\begin{proof}
The first part is \cite[Proposition 1.5]{Tayhomcoh}. The second part follows from that fact that $(-) \widehat{\otimes}_{\qp,\iota} M$ preserves the s-exactness of sequences of locally convex $A$-modules and Lemma \ref{lem:projfres}. 
\end{proof}

Let $H_1$ be a locally $\qp$-analytic group and let $H_2$ be a closed locally $\qp$-analytic subgroup. For an object $M$ of $\mathscr{G}_{H_2,A}$ we set

\begin{equation}\label{eq:ind}
\mathrm{ind}_{H_2}^{H_1}(M) := \mathscr{D}(H_1,A) \widetilde{\otimes}_{\mathscr{D}(H_2,A),\iota} M.
\end{equation}

For \eqref{eq:ind} to be a functor, we need the following lemma. 

\begin{lemm}\label{lem:dsfrs}
The (right) $\mathscr{D}(H_2,A)$-module $\mathscr{D}(H_1,A)$ is s-free. In particular there is an $A$-linear topological isomorphism
\[
\mathscr{D}(H_1,A) \cong \mathscr{D}(H_1/H_2,A) \widehat{\otimes}_{A,\iota} \mathscr{D}(H_2,A). 
\] 
\end{lemm}

\begin{proof}
The proof of \cite[Lemma 5.2]{kohl2011} gives that
\begin{equation}\label{eq:unfb}
\mathscr{D}(H_1,\qp) \cong \mathscr{D}(H_1/H_2,\qp) \widehat{\otimes}_{\qp,\iota} \mathscr{D}(H_2,\qp)
\end{equation}
We now compute
\begin{align*}
\mathscr{D}(H_1,A) &\overset{(i)}{=} \mathscr{D}(H_1,\qp) \widehat{\otimes}_{\qp,\iota} A \\
&\overset{(ii)}{=} \mathscr{D}(H_1/H_2,\qp) \widehat{\otimes}_{\qp,\iota} \mathscr{D}(H_2,\qp) \widehat{\otimes}_{\qp,\iota} A \\
&\overset{(iii)}{=} \mathscr{D}(H_1/H_2,\qp) \widehat{\otimes}_{\qp,\iota} \mathscr{D}(H_2,A) \\
&= \mathscr{D}(H_1/H_2,\qp) \widehat{\otimes}_{\qp,\iota} A \widehat{\otimes}_{A,\iota} \mathscr{D}(H_2,A) \\
&\overset{(iv)}{=} \mathscr{D}(H_1/H_2,A) \widehat{\otimes}_{A,\iota} \mathscr{D}(H_2,A)
\end{align*}
where (i), (iii) and (iv) follows from Lemma \ref{lem:distrbcompl}, and (ii) follows from \eqref{eq:unfb}. This completes the proof. 
\end{proof}

We are now ready to prove the following lemma. 

\begin{lemm} \label{lem:predjkz}
The functor
\begin{align*}
\mathrm{ind}_{H_2}^{H_1} \colon \mathscr{G}_{H_2,A} &\to \mathscr{G}_{H_1,A}\\
M &\mapsto \mathrm{ind}_{H_2}^{H_1}(M)
\end{align*}
takes s-exact sequences to s-exact sequences and s-projective objects to s-projective objects. 
\end{lemm}

\begin{proof}
Lemmas \ref{lem:tstenpr} and \ref{lem:dsfrs} imply that there is a natural $A$-linear topological isomorphism
\begin{equation}\label{eq:easindfor}
\mathrm{ind}_{H_2}^{H_1}(M) = \mathscr{D}(H_1/H_2,A) \widehat{\otimes}_{A,\iota} M. 
\end{equation}
Thus $\mathrm{ind}_{H_2}^{H_1}(M)$ is Hausdorff and complete. Its structure of a separately continuous $\mathscr{D}(H_1,A)$-module is the one induced from the s-free module $\mathscr{D}(H_1,A) \widehat{\otimes}_{A,\iota} M$. The final assertion follows from Lemmas \ref{lem:tstenpr} and \ref{lem:dsfrs}, and the fact that $\mathrm{ind}_{H_2}^{H_1}(M)$ respects direct sums.
\end{proof}

The next result relates locally analytic induction $\mathrm{Ind}_{H_2}^{H_1}$, cf. Remark \ref{rem:lains} and the functor $\mathrm{ind}_{H_2}^{H_1}$. 

\begin{lemm}\label{lem:prinbag}
Let $\delta \colon H_2 \to A^\times$ be a locally analytic character and suppose $H_1/H_2$ is compact and of dimension 1. Then $\mathrm{Ind}_{H_2}^{H_1} \delta$ and $\left( \mathrm{Ind}_{H_2}^{H_1} \delta \right)'_b$  are objects of $\mathscr{G}_{H_1,A}$ (where $\left( \mathrm{Ind}_{H_2}^{H_1} \delta \right)'_b$ is equipped with $H_1$-action: $(h_1 \cdot F)(f) := F(h_1^{-1}\cdot f)$ for $F \in \left( \mathrm{Ind}_{H_2}^{H_1} \delta \right)'_b$, $f \in \mathrm{Ind}_{H_2}^{H_1} \delta$ and $h_1 \in H_1$) , and we have an isomorphism
\[
\left( \mathrm{Ind}_{H_2}^{H_1} \delta \right)'_b \cong \mathrm{ind}_{H_2}^{H_1} \delta^{-1}
\]
in the category $\mathscr{G}_{H_1,A}$. 
\end{lemm}

\begin{proof}
Indeed by Remark \ref{rem:lains}, $\mathrm{Ind}_{H_2}^{H_1} \delta \cong \mathrm{LA}(H_1/H_2,A)$ as locally convex $A$-modules. By Remark \ref{rem:fonlbs} the latter is a complete locally convex $A$-module of $A$-LB-type. Thus $\mathrm{Ind}_{H_2}^{H_1} \delta$ is an object of $\mathrm{Rep}_{A}^{\mathrm{la},LB}(H_1)$. By Corollary \ref{cor:imded}, it follows that $\mathrm{Ind}_{H_2}^{H_1} \delta$ is an object of $\mathscr{G}_{H_1,A}$ as claimed.

Now as locally convex $A$-modules 
\begin{align*}
\left( \mathrm{Ind}_{H_2}^{H_1} \delta \right)'_b &\cong \mathrm{LA}(H_1/H_2,A)'_b \\
&\overset{(i)}\cong \mathscr{D}(H_1/H_2,A) \\
&\overset{(ii)}{\cong} \mathrm{ind}_{H_2}^{H_1} \delta^{-1}
\end{align*}
where (i) is by definition and (ii) follows from \eqref{eq:easindfor}. By \cite[\S 4.3]{feauxtob} we get a continuous $H_1$-equivariant $A$-linear surjection
\begin{equation} \label{eq:futhtal}
\mathscr{D}(H_1,A) \twoheadrightarrow \left( \mathrm{Ind}_{H_2}^{H_1} \delta \right)'_b,
\end{equation}
which by \cite[Theorem 1.1.17]{Emertonred} is strict. It is easy to check that the closure of the $A$-linear span of
\[
X := \{ h_1h_2 - \delta^{-1}(h_2)h_1 \text{ } \lvert \text{ } h_1 \in H_1, h_2 \in H_2 \}
\]
is the kernel of \eqref{eq:futhtal}. Indeed $X$ is contained in the kernel. For the converse pick a section $s \colon H_1/H_2 \to H_1$ of the canonical surjection $\pi \colon H_1 \to H_1/H_2$. Then we have an isomorphism
\begin{align*}
H_1/H_2 \times H_2 &\xrightarrow{\sim} H_1 \\
(x,b) &\mapsto s(x)\cdot b.
\end{align*}
Suppose now $T := \sum_{1 \leq i \leq m}a_ih_i$ is in the kernel of \eqref{eq:futhtal} with $a_i \in A$ and $h_i \in H_1$. Via the isomorphism $H_1/H_2 \times H_2 \xrightarrow{\sim} H_1$ we can rewrite $T$ uniquely of the form 
\[
\sum_{j}s(x_j)\left( \sum_i a_{j,i}b_{j,i} \right)
\] 
with $x_j \in H_1/H_2$ (pairwise distinct), $a_{j,i} \in A$ and $b_{j,i} \in H_2$. Under this identification, $\mathrm{Ind}_{H_2}^{H_1} \delta$ is identified with $\mathrm{LA}(H_1/H_2,A)$. Thus for each $j$ we must have
\[
\sum_{i}a_{j,i}\delta^{-1}(b_{j,i}) = 0.
\]
It is now clear that $T$ is contained in the $A$-linear span of $X$.

Thus we get a continuous $H_1$-equivariant $A$-linear topological isomorphism
\[
\alpha \colon \mathrm{ind}_{H_2}^{H_1} \delta^{-1} \xrightarrow{\sim} \left( \mathrm{Ind}_{H_2}^{H_1} \delta \right)'_b.
\]
Thus $\left( \mathrm{Ind}_{H_2}^{H_1} \delta \right)'_b$ is a locally analytic $H_1$-representation in $A$-modules. By \cite[Proposition 3.2]{tetschlad}, it follows that $\left( \mathrm{Ind}_{H_2}^{H_1} \delta \right)'_b$ is an object of $\mathscr{G}_{H_1,A}$. Since \eqref{eq:futhtal} is $H_1$-invariant, by continuity, $\alpha$ is $\mathscr{D}(H_1,A)$-linear. This completes the proof.    
\end{proof}

\begin{lemm} [Relative Frobenius reciprocity] \label{lem:relfrob}
If $M$ and $N$ are objects of $\mathscr{G}_{H_2,A}$ and $\mathscr{G}_{H_1,A}$, respectively, then there is an $A$-linear bijection
\[
\mathrm{Hom}_{\mathscr{G}_{H_1,A}}(\mathrm{ind}_{H_2}^{H_1}(M), N) \to \mathrm{Hom}_{\mathscr{G}_{H_2,A}}(M,N)
\] 
\end{lemm}

\begin{proof}
From the proof of Lemma \ref{lem:projfres} we have an $A$-linear bijection
\begin{align*}
\alpha \colon \mathrm{Hom}_{\mathscr{G}_{H_1,A}}(\mathscr{D}(H_1,A) \widehat{\otimes}_{A,\iota} M, N) &\to \mathrm{Hom}_{A,\mathrm{cont}}(M,N) \\
g &\mapsto \alpha(g)
\end{align*}
where $\alpha(g)(m) := g(1 \otimes m)$. It follows directly from the definitions that a continuous $\mathscr{D}(H_1,A)$-linear map $g$ from the left factors through $\mathrm{ind}_{H_2}^{H_1}(M)$ ($= \mathscr{D}(H_1/H_2,A) \widehat{\otimes}_{A,\iota} M$) if and only if $\alpha(g)$ is $\mathscr{D}(H_2,A)$-linear.  
\end{proof}

\begin{prop} [Relative Shapiro's Lemma] \label{thm:relshaplem}
Let $H_1$ be a locally $\qp$-analytic group and let $H_2$ be a closed locally $\qp$-analytic subgroup. If $M$ and $N$ are objects of $\mathscr{G}_{H_2,A}$ and $\mathscr{G}_{H_1,A}$, respectively, then there are $A$-linear bijections
\[
\Ext^{q}_{\mathscr{G}_{H_1,A}}(\mathrm{ind}_{H_2}^{H_1}(M), N) \to \Ext^{q}_{\mathscr{G}_{H_2,A}}(M, N)
\]
for all $q \geq 0$. 
\end{prop}

\begin{proof}
Choose an s-projective s-resolution $X_{\bullet} \to M$ in $\mathscr{G}_{H_2,A}$ (e.g. $(B_q(H,M), d_q)_{q \geq 0}$). By Lemma \ref{lem:predjkz}, the complex $\mathrm{ind}_{H_2}^{H_1}(X_{\bullet}) \to \mathrm{ind}_{H_2}^{H_1}(M)$ is an s-projective s-resolution of $\mathrm{ind}_{H_2}^{H_1}(M)$. By Lemma \ref{lem:relfrob} there is an $A$-linear bijection of complexes
\[
\mathrm{Hom}_{\mathscr{G}_{H_1,A}}(\mathrm{ind}_{H_2}^{H_1}(X_{\bullet}), N) \to \mathrm{Hom}_{\mathscr{G}_{X_{H_2},A}}(X_\bullet,N). 
\] 
The result now follows. 
\end{proof}

\bibliographystyle{acm}
\bibliography{library}

\end{document}